\documentclass[12pt,british,a4paper,reqno]{amsart}
\usepackage[T1]{fontenc}
	\usepackage[utf8]{inputenc}
	\usepackage[british]{babel}

	\usepackage{amsmath}
	\usepackage{amssymb}
	\usepackage{amsthm}
	\usepackage{braket}
 	\usepackage{bbm}
	\usepackage{xcolor}
    \usepackage{appendix}
	\usepackage{enumerate}
	\usepackage{enumitem}
	\usepackage{fancyhdr}
	\usepackage{geometry}
	\usepackage{graphicx}
\usepackage{wasysym}
    \usepackage[pdftex, colorlinks, linkcolor=red!75!black, citecolor=blue!85!black, urlcolor=green!66!black]{hyperref}
 \usepackage{bookmark}
	\usepackage{ifsym}
	\usepackage{indentfirst}
	\usepackage{mathrsfs}
	\usepackage{mathtools}
	\usepackage{proof}
	\usepackage{qtree}
	\usepackage{setspace}
	\usepackage{tensor}
	\usepackage{tikz}
	\usepackage{tikz-cd}
	\usepackage{tabstackengine}
    \usepackage{soul}
\usepackage{bbold}
        \usepackage[all]{xy}
\usepackage{circuitikz}
	\setstackgap{L}{.75\baselineskip}
\usetikzlibrary{chains}

\usepackage[T2A]{fontenc}
\DeclareSymbolFont{cyrletters}{T2A}{wncyr}{m}{n}
\DeclareMathSymbol{\rusI}{\mathalpha}{cyrletters}{"E8}
\DeclareMathSymbol{\RusI}{\mathalpha}{cyrletters}{"C8}
\DeclareMathSymbol{\rusEl}{\mathalpha}{cyrletters}{"EB}
\DeclareMathSymbol{\RusEl}{\mathalpha}{cyrletters}{"CB}
\DeclareMathSymbol{\RusZhe}{\mathalpha}{cyrletters}{"C6}
\DeclareMathSymbol{\RusSha}{\mathalpha}{cyrletters}{"D8}
\DeclareMathSymbol{\RusDe}{\mathalpha}{cyrletters}{"C4}

	\bookmarksetup{numbered,open}

	\renewcommand{\geq}{\geqslant}
	\renewcommand{\leq}{\leqslant}

      \newcommand{\tennsor}[3]{\tensor*[_{#1}]{#2}{_{#3}}}
      \newcommand{\tennsorp}[3]{\tensor*[_{#1}]{#2}{_{#3}^{\prime}}}

	\providecommand{\corollaryname}{Corollary}
	\providecommand{\definitionname}{Definition}
	\providecommand{\examplename}{Example}
	\providecommand{\lemmaname}{Lemma}
	\providecommand{\propositionname}{Proposition}
	\providecommand{\remarkname}{Remark}
	\providecommand{\theoremname}{Theorem}
	\providecommand{\setupname}{Setup}
	\providecommand{\notationname}{Notation}
	\providecommand{\conjecturename}{Conjecture}
	\providecommand{\questionname}{Question}
	\providecommand{\objectivename}{Objective}
	\providecommand{\aimname}{Aim}

	\theoremstyle{plain}
		\newtheorem{thm}{\protect\theoremname}[section] 
		
		\newtheorem{lem}[thm]{\protect\lemmaname}
		\newtheorem{cor}[thm]{\protect\corollaryname}

	\theoremstyle{definition}
		\newtheorem{defn}[thm]{\protect\definitionname}
		\newtheorem{example}[thm]{\protect\examplename}
		\newtheorem{setup}[thm]{\protect\setupname}
		\newtheorem{nota}[thm]{\protect\notationname}

	\theoremstyle{remark}
		\newtheorem{rem}[thm]{\protect\remarkname}
		
	\numberwithin{figure}{section}
	\numberwithin{equation}{section}

	\usetikzlibrary{matrix,arrows,decorations.pathmorphing,positioning,decorations.pathreplacing}
	\tikzset{commutative diagrams/.cd, 
		mysymbol/.style = {start anchor=center, end anchor = center, draw = none}}

	\let\amph=& 

    \newcommand{\bb}[1]{\mathbb{#1}}

	\newcommand{\cl}[1]{\mathcal{#1}}

	\newcommand{\sr}[1]{\mathscr{#1}}

	\newcommand{\ff}[1]{\mathsf{#1}}

	\newcommand{\bff}[1]{\boldsymbol{\mathsf{#1}}}

    \newcommand{\fk}[1]{\mathfrak{#1}}

		\newcommand{\Ab}{\operatorname{\mathsf{Ab}}\nolimits}

		\newcommand{\op}{\mathrm{op}}

        \newcommand{\Coeq}{\operatorname{Coeq}\nolimits}

		\newcommand{\Hom}{\operatorname{Hom}\nolimits}
		\newcommand{\End}{\operatorname{End}\nolimits}
		
		\newcommand{\rad}{\operatorname{rad}\nolimits}

		\newcommand{\cok}{\operatorname{coker}\nolimits}
		\newcommand{\im}{\operatorname{im}\nolimits}

  \newcommand{\rMod}[1]{\operatorname{\mathsf{Mod}}\nolimits\hspace{-0.5mm}\text{-}{#1}}
		\newcommand{\lMod}[1]{{#1}\text{-}\hspace{-0.5mm}\operatorname{\mathsf{Mod}}\nolimits}
     
		\newcommand{\Bimod}[1]{{#1}\operatorname{--\, \mathsf{Bimod}}\nolimits} 
        \newcommand{\Byemod}{\operatorname{\mathbb{BIMOD}}}
				\newcommand{\Rings}{\operatorname{\mathsf{Ring}}}
                \newcommand{\Comms}{\operatorname{\mathsf{Com}}}
                \newcommand{\Central}{\operatorname{\mathbb{CENTRAL}}}

		\newcommand{\Ext}{\operatorname{Ext}\nolimits}
		
   \newcommand{\iden}{\mathbb{1}}
		
		\newcommand{\Rep}[1]{\operatorname{\mathsf{Rep}}({#1})}

		\newcommand{\lmod}[1]{{#1}\operatorname{--\,\mathsf{mod}}\nolimits}

	\makeatletter
\DeclareRobustCommand{\rvdots}{%
  \vbox{
    \baselineskip4\p@\lineskiplimit\z@
    \kern-\p@
    \hbox{.}\hbox{.}\hbox{.}
  }}
\makeatother

\colorlet{daisie}{red!50!blue}
\definecolor{betterred}{RGB}{200,0,0}

\title{
Representations of infinite species 
}
\author{Raphael Bennett-Tennenhaus,  Job Daisie Rock}
\date{}

\begin{document}
\begin{abstract}
    We consider species, consisting of a possibly infinite set of rings,  and bimodules between them. 
    Simson realised the category of representations as a functor category, which we prove is  hereditary when each of the rings is  semisimple. 
    We use purity to provide sufficient conditions, in order for a representation to decompose into indecomposables with local endomorphism rings.  
    For any  bifunctor  valued in  bimodules,  we functorially  construct species equipped with commutativity conditions.  
    This generates  examples coming from a range of topics, such as subobject lattices in abelian length categories, the field choice problem in persistent homology, and topological field theories with  defects. 
\end{abstract}

\maketitle

\setcounter{tocdepth}{1}

\section{Introduction}

\subsection{Context and state-of-the-art}

Classically, such as in work of Gabriel \cite{Gabriel-indecomposable-reps-II}, Dlab--Ringel \cite{Dlabringstructures,Dlab-Ringel-graphs-1976} and Ringel \cite{Ringelspecies}, a \emph{species} is  a finite collection of division rings and bimodules between them. 
Since then, various authors have generalised this setting, relaxing the requirement of having  division rings. 
Drozd \cite{Drozd1982THESO} considered triangular rings, 
and gave necessary and sufficient conditions for when a hereditary ring is triangular.  
Gubareni \cite{Gubareni_2024} considered finite indexed species consisting of  discrete valuation rings. 
K{\"u}lshammer \cite{KlshammerProSpecies} considered finite indexed species consisting of left-right projective bimodules between finite-dimensional algebras. 
In this article, we allow infinitely indexed species. 

For the entire article $\ff{C}$ will be a fixed set with arbitrary cardinality. 

\begin{defn}\label{defn-species}
A \emph{species} $\fk{S}=(\tennsor{}{R}{x},\tennsor{y}{A}{x})$ consists of a (unital, associative) ring  $\tennsor{}{R}{x}$ for each $x\in \ff{C}$, and an $\tennsor{}{R}{y}$-$\tennsor{}{R}{x}$-bimodule $\tennsor{y}{A}{x}$ for each $(y,x)\in \ff{C}^2$.  

An $\fk{S}$-\emph{representation}  $(\tennsor{x}{M}{},\tennsor{y}{\alpha}{x})$ consists of a left $\tennsor{}{R}{x}$-module $\tennsor{x}{M}{}$ for each $x\in \ff{C}$, and an $\tennsor{}{R}{y}$-linear map $\tennsor{y}{\alpha}{x}\colon \tennsor{y}{A}{x}\otimes_{\tennsor{}{R}{x}} \tennsor{x}{M}{}\to \tennsor{y}{M}{}$ for each $(y,x)\in \ff{C}^{2}$.

A \emph{morphism} $(\tennsor{x}{f}{})\colon (\tennsor{x}{M}{},\tennsor{y}{\alpha}{x})\to (\tennsor{x}{N}{},\tennsor{y}{\beta}{x})$ of   $\fk{S}$-representations consists of an $\tennsor{}{R}{x}$-linear map  $\tennsor{x}{f}{}\colon \tennsor{x}{M}{}\to \tennsor{x}{N}{}$  for each $x\in\ff{C}$, such that $\tennsor{y}{\beta}{x}\circ (\tennsor{}{\iden}{\tennsor{y}{A}{x}}\otimes  \tennsor{x}{f}{}) = \tennsor{y}{f}{} \circ \tennsor{y}{\alpha}{x}$ for each $(y,x)\in \ff{C}^2$. 
%
To each species $\fk{S}$ we  consider the category $\Rep{\fk{S}}$ of $\fk{S}$-representations. 
\end{defn}


\emph{Phyla} were  introduced and studied in work of  Gao--K{\"u}lshammer--Kvamme--Psaroudakis \cite{GKKP}. 
The species we consider are examples of phyla, in which the involved abelian categories are module categories, and the involved functors are given by tensoring with a bimodule. 
These authors study the \emph{monomorphism category}, and establish certain adjunctions in order to study the preservation of almost-split sequences, when they exist. 
Note that, for the setting we work in,  almost-split sequences need not exist; see for example work of Paquette \cite{Paquette2011ANT}. 
We focus instead on  the construction of projective presentations, and the existence of Krull--Remak--Schmidt--Azumaya decompositions.

\subsection{Organization and contributions}

In Section~\ref{section:species-representations} we develop  preliminaries on species and their representations. 
Simson \cite{Simson1979}  noticed that the category of representations of a species $\fk{S}$ is equivalent to the module category of the \emph{preadditive tensor category} $\sr{T}(\fk{S})$, which generalises the tensor ring of a bimodule. 
Exploiting this, we show that co/limits are computed pointwise, and provide a setting for \emph{change-of-ring} results for preadditive tensor categories and their module categories.   
Using Yonedas lemma, we then study projectives, and using ideas of Ringel \cite{Ringelspecies}, extensions. 
In doing so, we establish our first main result. 
\begin{thm}[Lemmas \ref{prop-first-term-standard-res}~and~\ref{lem-when-proj},~Corollary~\ref{cor-yay-hereditary}]\label{introthm-hereditary}
    Let $M=(\tennsor{x}{M}{},\tennsor{y}{\alpha}{x})$ be a representation of a species $\fk{S}=(\tennsor{}{R}{x},\tennsor{y}{A}{x})$ such that each  $\tennsor{x}{M}{}$ is projective over $\tennsor{}{R}{x}$ and each $\tennsor{y}{A}{x}$ is projective over $\tennsor{}{R}{x}$ and over $\tennsor{}{R}{y}$.  
    Then $M$ has projective dimension at most $1$. 
    In particular, if each $\tennsor{}{R}{x}$ is semisimple artinian, then $\Rep{\fk{S}}$ is hereditary.
\end{thm}

Indecomposables  with local endomorphism rings are known as \emph{strongly indecomposable} representations. 
The  Krull--Remak--Schmidt--Azumaya theorem says  any decomposition into strongly indecomposable representations is unique.   
In Section~\ref{section-decomposition},  we turn our attention to the existence of such decompositions, and establish our second main result. 
%

\begin{thm}[Lemmas~\ref{lem-finiteness-cond-for-existence-strong-decompo},~\ref{lem-finiteness-cond-for-existence-strong-decompo-application},~\ref{lem-finite-ness-conds-for-botnan--crawley-boevey-argument}~and~\ref{lem-finite-length-means-nice-decomposition}]\label{introthm-decomposition}
    Let $M=(\tennsor{x}{M}{},\tennsor{y}{\alpha}{x})$ be a representation of a species $\fk{S}=(\tennsor{}{R}{x},\tennsor{y}{A}{x})$,   such that either (1) or (2) below holds. 
    \begin{enumerate}
        \item Each $\tennsor{x}{M}{}$ is artinian over $\End_{\Rep{\fk{S}}}(M)$. 
        For example, this holds if each  $\tennsor{x}{M}{}$ is artinian over a common central subring $S$ of each $\tennsor{}{R}{x}$  acting centrally on each $\tennsor{y}{A}{x}$.
        \item Each $\tennsor{x}{M}{}$ admits a bound on the length of a sequence  of non-zero $\tennsor{}{R}{x}$-submodules  whose sum is direct. 
        For example, this holds if each $\tennsor{x}{M}{}$ is finite length over $\tennsor{}{R}{x}$. 
    \end{enumerate}
    Then $M$ decomposes into a direct sum of strongly indecomposable representations of $\fk{S}$.
\end{thm}

Theorem~\ref{introthm-decomposition} above recovers well-known decomposition results, such a the decomposition theorem for persistence modules, due to Botnan--Crawley-Boevey  \cite[Theorem~1]{botnan-crawley-boevey-persistence}; see Corollary~\ref{cor-actually-understandable-decomposition}, which says that pointwise artinian functors enjoy such a decomposition.

%


In Section~\ref{section-relations} we consider ideals  $\sr{I}$ in $\sr{T}(\fk{S})$, discuss when they are compatible with the aforementioned change-of-ring results, and view the module category of $\sr{T}(\fk{S})/\sr{I}$ as a subcategory $\Rep{\fk{S},\sr{I}}$ of $\Rep{\fk{S}}$. 
Examples of ideals come from  \emph{associativity conditions},  defined by collections $\boldsymbol{c}$ of $\tennsor{}{R}{z}$-$\tennsor{}{R}{x}$ bimodule maps of the form  $\tennsor{}{c}{zyx}\colon\tennsor{z}{A}{y}\otimes \tennsor{y}{A}{x}\to\tennsor{z}{A}{x}$. 
Such conditions generalise the commutativity conditions defined in \cite{Simson1979}.

\begin{thm}[Corollaries~\ref{cor-cool-nice-cheers-ta-nice1}~and~\ref{cor-double-functors-giving-species}]
\label{introthm-double}
There exists a diagram of double functors
    \[
    \begin{tikzcd}
    &
    \Rings\arrow[dr, "\RusDe"']
    \arrow[rr, "\RusZhe"]
    &
    &
    \Byemod
    \\
    &
    &
\Central_{\Comms}^{\Rings}\arrow[ur, "\RusSha"']
    &
    \end{tikzcd}
    \]
    where $\Rings$ is the category of rings, $\Byemod$ is the double category of bimodule over rings, and $\Central_{\Comms}^{\Rings}$ is the bicategory consisting of cospans of central ring maps from commutative rings. 
    Furthermore, for any double category $\cl{D}$ there is a functor $\sr{T}(\fk{S}(-))/\sr{I}(\boldsymbol{c}(-))$ from the category of double functors $\cl{D}\to \Byemod$ to the category of small preadditive categories. 
\end{thm}

In Section~\ref{sec-examples} we turn our attention to examples and applications.
To begin we ground our work in something familiar, showing how  Theorem~\ref{introthm-decomposition} applies to the  category  of pointwise artinian  functors valued over a commutative ring (Corollary~\ref{cor-actually-understandable-decomposition}). 
We then detail more exotic examples, recovered using Theorem~\ref{introthm-double},  and applications of Theorem~\ref{introthm-decomposition}. 

\begin{itemize}
\item[\ref{subsec:abelian-length-krause}] Abelian length categories and distributive lattices, following work of Krause \cite{HENNING-ORDERS}, give examples of infinite species of noncommutative division rings. 
\item[\ref{sec-field-choice}] Persistent homology over varying  fields, as in work of Boissonnat--Maria \cite{Boissonnatmariacomputing} and Obayashi--Yoshiwaki  \cite{ObayashiYoshiwakiFieldChoice}, gives  infinite species of rings of the form $\bb{Z}/m\bb{Z}$. 
\item[\ref{sec-thread-quivers},~\ref{sec-infinite-Dynkin}] Thread quivers, in the sense of Berg--Van Roosmalen \cite{berg-vanroosmalen-thread1}, studied later by Paquette--Rock--Yildrim   \cite{paquette-rock-yildirim-thread2}, lead to the notion of a  \emph{threaded species}, allowing for a reasonable notion of infinite species of Dynkin  types $\bb{B}$,  $\bb{C}$, $\widetilde{\bb{B}}$,  $\widetilde{\bb{C}}$ and $\widetilde{\bb{BC}}$. 
    \item[\ref{sec-top-field-theory}] Topological field theories with defects, as in work of Davydov--Kong--Runkel \cite{field-theories}, is a natural setting where infinite species of Frobenius algebras arise. 
\end{itemize}



The end of the paper includes two appendices: the first on additive tensor categories (Appendix~\ref{appendix-tensor-cat}) and the second on  double categories and  double functors (Appendix~\ref{appendix-double-cats}).

    \subsection*{Acknowledgements}
    The first author thanks Henning Krause for his talk in Bielefeld on abelian length categories and Birkhoff's representation theorem.

    \subsection*{Funding}
    The first author was supported by the Deutsche Forschungsgemeinschaft (SFB-TRR 358/1
2023- 491392403). 
The first  author is grateful for this support. 
    
	The second author is supported by FWO grant 1298325N.
	Work on this project began while the second author was supported by BOF grant 01P12621 from Universiteit Gent.
	The second author was also partially supported by the FWO grants G0F5921N (Odysseus) and G023721N, and by the KU Leuven grant iBOF/23/064.





\section{Preliminaries}
\label{section:species-representations}

In this section we study species and their representations. 
In particular, we exploit the equivalence between representations of a species and a modules over a preadditive category defined by the species (Lemma~\ref{theorem-equivalence-of-reps-of-species-and-tensor-cat}).
Near the end we show that, under mild conditions, the category of representations of a species is hereditary (Corollary~\ref{cor-yay-hereditary}).


%
     %
%
\begin{nota}
A (\emph{length}-$\ff{n}$) \emph{path} is a sequence  $\underline{x}=(x_{\ff{n}},\dots,x_{0})$ with $\ff{n}\geq 0$. 
For any $h,t\in\ff{C}$, we let $\ff{C}^{\ff{n}}(h,t)$ be the set of such paths $\underline{x}$ where $x_{\ff{n}}=h$ and $x_{0}=t$, and we define the $\tennsor{}{R}{h}$-$\tennsor{}{R}{t}$-bimodule  $\tennsor{h\,}{\underline{A}}{\,t}= \bigoplus_{\ff{n}\geq0,\,\underline{x}\in \ff{C}^{\ff{n}}(h,t)} \tennsor{}{A}{\underline{x}}$  whose summands  are defined in  cases, by
    \[
     \begin{array}{ccc}
    A_{(h)}=\tennsor{}{R}{h},
    &
    \tennsor{}{A}{(h,t)}=\tennsor{h}{A}{t}, 
    &
    \tennsor{}{A}{\underline{x}}=\tennsor{h}{A}{x_{\ff{n}-1}}\otimes_{\tennsor{}{R}{x_{\ff{n}-1}}}\,\dots \,\tennsor{}{\otimes}{\tennsor{}{R}{x_{2}}} \tennsor{x_{2}}{A}{x_{1}}\tennsor{}{\otimes}{\tennsor{}{R}{x_{1}}} \tennsor{x_{1}}{A}{t}\quad(\ff{n}\geq 2). 
    \end{array}
    \]

For $x\in\ff{C}$ and $\underline{x}\in \ff{C}^{\ff{n}}(h,t)$, define the left $\tennsor{}{R}{h}$-linear map $\tennsor{}{\alpha}{\underline{x}}\colon \tennsor{}{A}{\underline{x}}\tennsor{}{\otimes}{\tennsor{}{R}{t}}\tennsor{t}{M}{}\to \tennsor{h}{M}{}$ by
\[
  \begin{array}{ccc}
\tennsor{}{\alpha}{(x)}=\tennsor{}{\iden}{\tennsor{x}{M}{}},
&
\tennsor{}{\alpha}{(h,t)}=\tennsor{h}{\alpha}{t}, 
&
\tennsor{}{\alpha}{\underline{x}}=
\tennsor{h}{\alpha}{x_{\ff{n}-1}}(\tennsor{}{\iden}{\tennsor{}{A}{(h,x_{\ff{n}-1})}}\otimes\tennsor{x_{\ff{n}-1}}{\alpha}{x_{\ff{n}-2}})
\cdots 
(\tennsor{}{\iden}{\tennsor{}{A}{(h,\dots,x_{1})}}\otimes\tennsor{x_{1}}{\alpha}{t}). 
\end{array}
\]
Combining over $\underline{x}$ defines a map  $\tennsor{h\,}{\underline{\alpha}}{\,t}\colon \tennsor{h\,}{\underline{A}}{\,t}\tennsor{}{\otimes}{\tennsor{}{R}{t}}\tennsor{t}{M}{}\to \tennsor{h}{M}{}$. 
   %
\end{nota}
Notice the indexing for paths starts at 0 on the \emph{right}.
This choice of convention allows for easier notation regarding composition.
Note also that  $\tennsor{x\,}{\underline{A}}{\,x}$ is  a ring for each $x\in\ff{C}$.

%


\begin{lem}
    \label{lem-elementary-underline-maps}
    Let $\fk{S}=(\tennsor{}{R}{x},\tennsor{y}{A}{x})$ be a species, $h,t\in\ff{C}$,  $\underline{a}\in\tennsor{h\,}{\underline{A}}{\,t}$ and $M=(\tennsor{x}{M}{},\tennsor{y}{\alpha}{x})$ be a representation. Then  the map $\tennsor{h\,}{\underline{\alpha}}{\,t}(\underline{a}\otimes -)$ given by $m\mapsto\tennsor{h\,}{\underline{\alpha}}{\,t}(\underline{a}\otimes m)$ is $\End_{\Rep{\fk{S}}}(M)$-linear. 
\end{lem}

\begin{proof}
    Since $\underline{a}$ is a finite sum of pure tensors in the $\tennsor{}{R}{h}$-$\tennsor{}{R}{t}$-bimodules $\tennsor{}{A}{\underline{x}}$, it suffices to assume $\underline{a}=\tennsor{}{a}{(hx_{\ff{n}-1}}\otimes \cdots \otimes\tennsor{}{a}{x_{1}t}$ where $\ff{n}\geq 0$,  $\underline{x}\in \ff{C}^{\ff{n}}(h,t)$, and  $\tennsor{}{a}{x_{\ff{i}+1}x_{\ff{i}}}\in \tennsor{x_{\ff{i}+1}}{A}{x_{\ff{i}}}$ for $0\leq \ff{i}<\ff{n}$.

    If $\ell\in\ff{C}$, $f\in \End_{\Rep{\fk{S}}}(M)$ and $m\in \tennsor{\ell}{M}{}$ then $f\cdot m=\tennsor{\ell}{f}{}(m)$. 
    By definition, 
    \begin{gather*}
\tennsor{h}{\alpha}{x_{\ff{n}-1}}(
\cdots (
(\tennsor{}{\iden}{\tennsor{}{A}{(h,\dots,x_{1})}}\otimes\tennsor{x_{1}}{\alpha}{t})(\tennsor{}{a}{hx_{\ff{n}-1}}\otimes \cdots \tennsor{}{a}{x_{1}t}\otimes \tennsor{t}{f}{}(m)))\cdots)
\\
=
\tennsor{h}{\alpha}{x_{\ff{n}-1}}(
\cdots ((\tennsor{}{\iden}{\tennsor{}{A}{(h,\dots,x_{2})}}\otimes\tennsor{x_{2}}{\alpha}{x_{1}})
(\tennsor{}{a}{hx_{\ff{n}-1}}\otimes \cdots \otimes \tennsor{x_{1}}{\alpha}{t}(\tennsor{}{a}{x_{1}t}\otimes \tennsor{t}{f}{}(m)))\cdots)
\\
=
\tennsor{h}{\alpha}{x_{\ff{n}-1}}(
\cdots ((\tennsor{}{\iden}{\tennsor{}{A}{(h,\dots,x_{2})}}\otimes\tennsor{x_{2}}{\alpha}{x_{1}})
(\tennsor{}{a}{hx_{\ff{n}-1}}\otimes \cdots \otimes \tennsor{x_{1}}{f}{}(\tennsor{x_{1}}{\alpha}{t}(\tennsor{}{a}{x_{1}t}\otimes m))\cdots)
\\
\rvdots
\\
  =\tennsor{h}{f}{}(\tennsor{h}{\alpha}{x_{\ff{n}-1}}(\tennsor{}{a}{hx_{\ff{n}-1}}\otimes 
\tennsor{x_{\ff{n}-1}}{\alpha}{x_{\ff{n}-2}}(\tennsor{}{a}{x_{\ff{n}-1}x_{\ff{n}-2}}\otimes \cdots \otimes\tennsor{x_{2}}{\alpha}{x_{1}}( \tennsor{}{a}{x_{2}x_{1}}\otimes\tennsor{x_{1}}{\alpha}{t}(\tennsor{}{a}{x_{1}t}\otimes m))))). 
    \end{gather*}
    This means $\tennsor{}{\alpha}{\underline{x}}(\underline{a}\otimes f\cdot m))
        =f\cdot (\tennsor{}{\alpha}{\underline{x}}(\underline{a}\otimes m))$, as required. 
\end{proof}

\begin{defn}
\label{def-tensor-cat}
\cite{Simson1979} 
    The \emph{preadditive tensor category} $\sr{T}(\fk{S})$ has object set $\ff{C}$, and has morphism sets given by  $\Hom_{\sr{T}(\fk{S})}(t,h)=\tennsor{h\,}{\underline{A}}{\,t}$ for each $h,t\in\ff{C}$. 
\end{defn}

\begin{lem}
    \label{theorem-equivalence-of-reps-of-species-and-tensor-cat}
For each species $\fk{S}$  there is an additive equivalence $\Rep{\fk{S}}\to\lMod{\sr{T}(\fk{S})}$.  
\end{lem}

\begin{proof}
The argument from the proof of \cite[Theorem 2.1]{Simson1979}  generalises directly. Associated to any small preadditive category $\cl{C}$ and any biadditive functor $\sr{A}\colon \cl{C}^{\op}\times\cl{C}\to\Ab$, Simson \cite{Simson1979} defined the \emph{additive tensor category} $\tennsor{}{\sr{T}}{\cl{C}}(\sr{A})$. For more detail see Appendix~\ref{appendix-tensor-cat}.

For a species $\fk{S}=(\tennsor{}{R}{x},\tennsor{y}{A}{x})$ we have $\sr{T}(\fk{S})=\tennsor{}{\sr{T}}{\cl{C}(\fk{S})}(\sr{A}_{\fk{S}})$ where $\cl{C}(\fk{S})$ has  object set $\ff{C}$,  $\End_{\cl{C}(\fk{S})}(x)=\tennsor{}{R}{x}$, $\Hom_{\cl{C}(\fk{S})}(x,y)=0$ if $x\neq y$,  $\sr{A}_{\fk{S}}(x,y)=\tennsor{y}{A}{x}$, and $\sr{A}_{\fk{S}}(s,r)(a)=s a r$, for all $x,y\in\ff{C}$,  $s\in \tennsor{}{R}{y}$, $ a\in \tennsor{y}{A}{x}$ and $ r\in  \tennsor{}{R}{x}$. 
Any $\fk{S}$-representation $M=(\tennsor{x}{M}{},\tennsor{y}{\alpha}{x})$ defines an additive functor  
$\sr{T}(\fk{S})\to \Ab$ 
%
%
  by $x\mapsto \tennsor{x}{M}{}$ and   $\tennsor{h\,}{\underline{A}}{\,t}\ni \underline{a}\mapsto\tennsor{h\,}{\underline{\alpha}}{\,t}(\underline{a}\otimes -)$.  
\end{proof}

If $w\in\ff{C}$ and $\fk{S}=(\tennsor{}{R}{x},\tennsor{x}{A}{y})$ is a species then  we consider the canonical functor 
\begin{equation}\label{eqn-w-Pi}
\begin{array}{ccc}
\tennsor{w}{\Pi}{}\colon\Rep{\fk{S}}\to \lMod{\tennsor{}{R}{w}},
&
(\tennsor{x}{M}{},\tennsor{y}{\alpha}{x})\mapsto \tennsor{w}{M}{}, 
&
(\tennsor{x}{f}{})\mapsto \tennsor{w}{f}{}
\end{array}
\end{equation}
 induced from the inclusion of $R_w$, considered a one-object preadditive category, into $\sr{T}(\fk{S})$. 
\begin{lem}
    \label{lemma-co-limits-in-rep-of-species-cat}
Let $\fk{S}=(\tennsor{}{R}{x},\tennsor{y}{A}{x})$ be a species. 
Let $\sr{D}\colon \cl{I}\to \Rep{\fk{S}}$ be a diagram given on objects by $i\mapsto (\tennsor{x}{M}{}(i),\tennsor{y}{\alpha}{x}(i))$ and on morphisms by $\alpha\mapsto(\tennsor{x}{f}{}(\alpha))$. 
Then the collections
\[
\begin{array}{cc}
\left[(\tennsor{x}{\ell}{}(i))\colon (\tennsor{x}{L}{},\tennsor{y}{\lambda}{x})\to (\tennsor{x}{M}{}(i),\tennsor{y}{\alpha}{x}(i))\right]_{i\in\cl{I}},
&
\left[(\tennsor{x}{c}{}(i))\colon (\tennsor{x}{M}{}(i),\tennsor{y}{\alpha}{x}(i))\to (\tennsor{x}{C}{},\tennsor{y}{\gamma}{x})\right]_{i\in\cl{I}}
\end{array}
\]
defined as follows, respectively give the 
limit and colimit of the diagram. 
\begin{enumerate}
    \item $\left[\tennsor{x}{\ell}{}(i)\colon\tennsor{x}{L}{}\to \tennsor{x}{M}{}(i)\right]_{i\in\cl{I}}$ is the limit of $\tennsor{x}{\Pi}{}\sr{D}\colon \cl{I}\to \Rep{\fk{S}}\to \lMod{\tennsor{}{R}{x}}$. 
    \item $\tennsor{y}{\lambda}{x}$ is  universal, unique such that $\tennsor{y}{\alpha}{x}(i)\circ  (\tennsor{}{\iden}{\tennsor{y}{A}{x}}\otimes  \tennsor{x}{\ell}{}(i)) = \tennsor{y}{\ell}{}(i) \circ \tennsor{y}{\lambda}{x}$ for each $i$. 
    \item $\left[\tennsor{x}{c}{}(i)\colon \tennsor{x}{M}{}(i)\to \tennsor{x}{C}{}\right]_{i\in\cl{I}}$ is the colimit of $\tennsor{x}{\Pi}{}\sr{D}\colon\cl{I}\to \Rep{\fk{S}}\to \lMod{\tennsor{}{R}{x}}$. 
    \item $\tennsor{y}{\gamma}{x}$ is  universal, unique such that $\tennsor{y}{c}{}(i) \circ \tennsor{y}{\alpha}{x}(i)
    =
    \tennsor{y}{\gamma}{x}\circ (\tennsor{}{\iden}{\tennsor{y}{A}{x}}\otimes  \tennsor{x}{c}{}(i))$ for each $i$. 
\end{enumerate}
\end{lem}

\begin{proof}
    The (co)limit of any diagram  $\cl{I}
    \to \lMod{\sr{T}(\fk{S})}$ exists, and is  computed pointwise.  
    Thus $\cl{I}$-shaped (co)limits exist in $\Rep{\fk{S}}$ by Lemma~\ref{theorem-equivalence-of-reps-of-species-and-tensor-cat}. 

    The forgetful functor $ \lMod{\tennsor{}{R}{x}}\to \Ab$ preserves (co)limits, and the image of each collection in the statement is sent to the (co)limit under the equivalence  $\Rep{\fk{S}}\to \lMod{\sr{T}(\fk{S})}$. 
\end{proof}


\begin{lem}
    \label{lem-functors-between-reps-of-species}
    Let $\fk{S}=(\tennsor{}{R}{x},\tennsor{y}{A}{x})$ and $\fk{T}=(\tennsor{}{S}{x},\tennsor{y}{B}{x})$ be species, $\tennsor{x}{\sr{G}}{}\colon \lMod{\tennsor{}{R}{x}}\to \lMod{\tennsor{}{S}{x}}$ be an additive functor for each $x\in\ff{C}$, and  ${\tennsor{y}{\Omega}{x}}\colon \tennsor{y}{B}{x}\otimes \tennsor{x}{\sr{G}}{}(-)\Rightarrow \tennsor{y}{\sr{G}}{}(\tennsor{y}{A}{x}\otimes-)$ be a natural transformation for each $(y,x)\in\ff{C}^{2}$. 
    %
%
%
    %
        Then there is an additive functor 
\[
\tennsor{}{\sr{G}}{\Omega}\colon \Rep{\fk{S}}\to \Rep{\fk{T}},
\,
(\tennsor{x}{M}{},\tennsor{y}{\alpha}{x})\mapsto (\tennsor{x}{\sr{G}}{}(\tennsor{x}{M}{}),\tennsor{y}{\sr{G}}{}(\tennsor{y}{\alpha}{x})\circ \tennsor{}{{(\tennsor{y}{\Omega}{x})}}{{\tennsor{x}{M}{}}}),
\,
(\tennsor{x}{f}{})\mapsto (\tennsor{x}{\sr{G}}{}(\tennsor{x}{f}{})).
\]
Furthermore, there are natural transformations ${\tennsor{y\,}{\underline{\Omega}}{\,x}}\colon \tennsor{y\,}{\underline{B}}{\,x}\otimes \tennsor{x}{\sr{G}}{}(-)\Rightarrow \tennsor{y}{\sr{G}}{}(\tennsor{y\,}{\underline{A}}{\,x}\otimes-)$ for each $(y,x)\in\ff{C}^{2}$ such that $\tennsor{y\,}{\underline{\omega}}{\,x} =\tennsor{y}{\sr{G}}{}(\tennsor{y\,}{\underline{\alpha}}{\,x})\circ \tennsor{}{{(\tennsor{y\,}{\underline{\Omega}}{\,x})}}{{\tennsor{x}{M}{}}}$ whenever $\tennsor{}{\sr{G}}{\Omega}(\tennsor{x}{M}{},\tennsor{y}{\alpha}{x})=(\tennsor{x}{N}{},\tennsor{y}{\omega}{x})$. 
\end{lem}

\begin{proof}
   Notice $\tennsor{}{\sr{G}}{\Omega}$ sends objects to objects.
   It is straightforward to see that the functoriality of each $\tennsor{x}{\sr{G}}{}$, and the naturality of each $\tennsor{y}{\Omega}{x}$, ensure $\tennsor{}{\sr{G}}{\Omega}$ is functorial.  
    %
    %
%
%

 %
For each $\underline{x}\in \ff{C}^{\ff{n}}(h,t)$ define ${\tennsor{}{\Omega}{\underline{x}}}\colon \tennsor{}{B}{\underline{x}}\otimes \tennsor{t}{\sr{G}}{}(-)\Rightarrow \tennsor{h}{\sr{G}}{}(\tennsor{}{A}{\underline{x}}\otimes-)$ as follows, where we use the terminology of \emph{whiskering} natural transformations with functors, in the sense of \cite{MacLane-categories-for-the-working-mathematician}: see \S II.5, \S XII.3, and in particular item (11) on page 275. 
If $\ff{n}=0$ then $\underline{x}=(x)$ and let ${\tennsor{}{\Omega}{\underline{x}}}$ be the identity  on $\tennsor{x}{\sr{G}}{}$. 
Otherwise $\ff{n}\geq 1$, and we define ${\tennsor{}{\Omega}{\underline{x}}}$ to be the composition
\[\begin{tikzcd}
	{\tennsor{}{B}{(h,\dots,t)}\otimes\tennsor{t}{\sr{G}}{} (-)} & {\tennsor{}{B}{(h,\dots,x_{1})}\otimes\tennsor{x_{1}}{\sr{G}}{} (\tennsor{x_{1}}{A}{t}\otimes -)} & \cdots & {\tennsor{x_{\ff{n}}}{\sr{G}}{} (\tennsor{}{A}{(h,\dots,t)}\otimes -).}
	\arrow[from=1-1, to=1-2]
	\arrow[from=1-2, to=1-3]
	\arrow[from=1-3, to=1-4]
\end{tikzcd}\]
Here ${\tennsor{}{B}{(h,\dots,x_{\ff{i}})}\otimes\tennsor{x_{\ff{i}}}{\sr{G}}{} (\tennsor{}{A}{(x_{\ff{i}},\dots,t)}\otimes -)}\to {\tennsor{}{B}{(h,\dots,x_{\ff{i}+1})}\otimes\tennsor{x_{\ff{i}+1}}{\sr{G}}{} (\tennsor{}{A}{(x_{\ff{i}+1},\dots,t)}\otimes -)}$ is the whiskering 
\[\begin{tikzcd}
	&&& {\lMod{\tennsor{}{S}{x_{\ff{i}}}}} &&& \\
	{\lMod{\tennsor{}{R}{t}}} && {\lMod{\tennsor{}{R}{x_{\ff{i}}}}} && {\lMod{\tennsor{}{S}{x_{\ff{i}+1}}}} && {\lMod{\tennsor{}{S}{h}}.} \\
	&&& {\lMod{\tennsor{}{R}{x_{\ff{i}+1}}}}
	\arrow["{{\tennsor{x_{\ff{i}+1}}{B}{x_{\ff{i}}}\otimes-}}", from=1-4, to=2-5]
	\arrow["{{{\tennsor{x_{\ff{i}+1}}{\Omega}{x_{\ff{i}}}}}}"{description}, Rightarrow, from=1-4, to=3-4]
	\arrow["{{{\tennsor{}{A}{(x_{\ff{i}},\dots,t)}}\otimes - }}", from=2-1, to=2-3]
	\arrow["{{\tennsor{x_{\ff{i}}}{\sr{G}}{} }}", from=2-3, to=1-4]
	\arrow["{{\tennsor{x_{\ff{i}+1}}{A}{x_{\ff{i}}}\otimes-}}"', from=2-3, to=3-4]
	\arrow["{{{\tennsor{}{B}{(h,\dots,x_{\ff{i}+1})}}\otimes - }}", from=2-5, to=2-7]
	\arrow["{{\tennsor{x_{\ff{i}+1}}{\sr{G}}{} }}"', from=3-4, to=2-5]
\end{tikzcd}\]
Define ${\tennsor{y\,}{\underline{\Omega}}{\,x}}$ by combining the transformations ${\tennsor{}{\Omega}{\underline{x}}}$, using the distributivity of the tensor product over direct sums. 
If $(k,j,i)\in\ff{C}^{3}$ then using $\tennsor{j}{\alpha}{i}$ in the naturality of $\tennsor{k}{\Omega}{j}$ gives
\begin{align*}
\tennsor{k}{\sr{G}}{}\left(\tennsor{k}{\alpha}{j}\right) \circ \tennsor{}{{\left(\tennsor{k}{\Omega}{j}\right)}}{{\tennsor{j}{M}{}}}\circ 
    \left(\tennsor{}{\iden}{\tennsor{k}{B}{j}}\otimes \left(
    \tennsor{j}{\sr{G}}{}\left(\tennsor{j}{\alpha}{i}\right)\circ 
\tennsor{}{{\left(\tennsor{j}{\Omega}{i}\right)}}{{\tennsor{i}{M}{}}}\right)\right)=
\\
        \tennsor{k}{\sr{G}}{}\left(\tennsor{k}{\alpha}{j}\right) \circ \tennsor{}{{\left(\tennsor{k}{\Omega}{j}\right)}}{{\tennsor{j}{M}{}}}\circ 
    \left(\tennsor{}{\iden}{\tennsor{k}{B}{j}}\otimes 
    \tennsor{j}{\sr{G}}{}\left(\tennsor{j}{\alpha}{i}\right)\right)\circ
    \left(\tennsor{}{\iden}{\tennsor{k}{B}{j}}\otimes 
\tennsor{}{{\left(\tennsor{j}{\Omega}{i}\right)}}{{\tennsor{i}{M}{}}}\right)
    =\\
    \tennsor{k}{\sr{G}}{}\left(\tennsor{k}{\alpha}{j}\right) \circ\tennsor{k}{\sr{G}}{}\left(\tennsor{}{\iden}{\tennsor{k}{A}{j}}\otimes\tennsor{j}{\alpha}{i}\right)\circ \tennsor{}{{\left(\tennsor{j}{\Omega}{i}\right)}}{{\tennsor{j}{A}{i}\otimes \tennsor{i}{M}{}}}\circ
    \left(\tennsor{}{\iden}{\tennsor{k}{B}{j}}\otimes 
\tennsor{}{{\left(\tennsor{j}{\Omega}{i}\right)}}{{\tennsor{i}{M}{}}}\right)=
\\
     \tennsor{k}{\sr{G}}{}\left(\tennsor{k}{\alpha}{j}\circ \left(\tennsor{}{\iden}{\tennsor{k}{A}{j}}\otimes\tennsor{j}{\alpha}{i}\right)\right)\circ \tennsor{}{{\left(\tennsor{j}{\Omega}{i}\right)}}{{\tennsor{j}{A}{i}\otimes \tennsor{i}{M}{}}}\circ
    \left(\tennsor{}{\iden}{\tennsor{k}{B}{j}}\otimes 
\tennsor{}{{\left(\tennsor{j}{\Omega}{i}\right)}}{{\tennsor{i}{M}{}}}\right).
\end{align*}
By induction, $\tennsor{}{\omega}{\underline{x}} =\tennsor{h}{\sr{G}}{}(\tennsor{}{\alpha}{\underline{x}})\circ \tennsor{}{{(\tennsor{}{\Omega}{\underline{x}})}}{{\tennsor{t}{M}{}}}$ for each $\underline{x}\in\ff{C}^{\ff{n}}(h,t)$, as required.  
\end{proof}

\begin{lem}
\label{lem-ring-and-bilinear-maps-induce-functor-between-tensor-cat}
    Let $\fk{S}=(\tennsor{}{R}{x},\tennsor{y}{A}{x})$ and $\fk{T}=(\tennsor{}{S}{x},\tennsor{y}{B}{x})$ be species,  $\tennsor{}{\rusEl}{x}\colon \tennsor{}{S}{x}\to \tennsor{}{R}{x}$ be a ring map for each $x\in\ff{C}$, and $\tennsor{y}{\RusEl}{x}\colon \tennsor{y}{B}{x}\to\tennsor{y}{A}{x}$ be an additive map for each $x,y\in\ff{C}$, such that
    \begin{equation}
    \label{eqn-ring-and-bilinear-maps-induce-functor-between-tensor-cat}
        \begin{array}{cc}
    \tennsor{y}{\RusEl}{x}(sbt)=\tennsor{}{\rusEl}{y}(s)\tennsor{y}{\RusEl}{x}(b)\tennsor{}{\rusEl}{x}(t),
    &
    (s\in \tennsor{}{S}{y},\,b\in \tennsor{y}{B}{x},\,t\in\tennsor{}{S}{x}).
    \end{array}
    \end{equation}
    Then there is an additive functor $(\rusEl,\RusEl)\colon \sr{T}(\fk{T})\to \sr{T}(\fk{S})$ which is the identity on objects. 
\end{lem}

\begin{proof}
We need to define the functor on morphisms. 
For each  $(h,t)\in\ff{C}^{2}$ and $\underline{x}\in\ff{C}^{n}(h,t)$ we define an $\tennsor{}{R}{h}$-$\tennsor{}{R}{t}$-bilinear map $\tennsor{}{\RusEl}{\underline{x}}\colon \tennsor{}{B}{\underline{x}}\to \tennsor{}{A}{\underline{x}}$ inductively on $\ff{n}\geq 0$. 
For $\ff{n}=0$ we have $\underline{x}=(x)$ and we let $\tennsor{}{\RusEl}{\underline{x}}=\tennsor{}{\rusEl}{x}$. 
Now suppose $\ff{n}\geq 1$, let $\ell=x_{\ff{n}-1}$ and assume $\tennsor{}{\RusEl}{\underline{w}}\colon \tennsor{}{B}{\underline{w}}\to \tennsor{}{A}{\underline{w}}$ has been defined, where $\underline{w}=(\ell,\dots, t)\in\ff{C}^{n-1}(\ell,t)$ is found by removing $h$ from $\underline{x}$. 

Consider the map $\tennsor{h}{B}{\ell}\times \tennsor{}{B}{\underline{w}}\to \tennsor{}{A}{\underline{x}}$ given by $(b,\underline{b})\mapsto \tennsor{h}{\RusEl}{\ell}(b)\otimes \tennsor{}{\RusEl}{\underline{w}}(\underline{b})$. 
Taking $x=\ell$ and $y=h$ in \eqref{eqn-ring-and-bilinear-maps-induce-functor-between-tensor-cat}, and using the $\tennsor{}{R}{\ell}$-$\tennsor{}{R}{t}$-bilinearity of $\tennsor{}{\RusEl}{\underline{w}}$, it follows that this map is $\tennsor{}{R}{\ell}$-balanced, left $\tennsor{}{R}{h}$-linear, and  right  $\tennsor{}{R}{t}$-linear. 
We define $\tennsor{}{\RusEl}{\underline{x}}$ by the universal property of the tensor product. 
Using the universal property of the coproduct over $\underline{x}$, the maps $\tennsor{}{\RusEl}{\underline{x}}$ combine to define a morphism $\tennsor{h\,}{\underline{\RusEl}}{\,t}\colon \tennsor{h\,}{\underline{B}}{\,t}\to \tennsor{h\,}{\underline{A}}{\,t}$, unique such that
\[
\tennsor{h\,}{\underline{\RusEl}}{\,t}(\tennsor{}{b}{hx_{\ff{n}-1}}\otimes \cdots \otimes \tennsor{}{b}{x_{1}t})=\tennsor{h}{\RusEl}{x_{\ff{n}-1}}(\tennsor{}{b}{hx_{\ff{n}-1}})\otimes \cdots \otimes \tennsor{x_{1}}{\RusEl}{t}(\tennsor{}{b}{x_{1}t})\in \tennsor{h}{A}{x_{\ff{n}-1}}\otimes \cdots \otimes \tennsor{x_{1}}{A}{t}
\]
for each pure tensor $\tennsor{}{b}{hx_{\ff{n}-1}}\otimes \cdots \otimes \tennsor{}{b}{x_{1}t}\in \tennsor{h}{B}{x_{\ff{n}-1}}\otimes \cdots \otimes \tennsor{x_{1}}{B}{t}$. 
\end{proof}

\begin{rem}
    Adopt the set up of Lemma~\ref{lem-ring-and-bilinear-maps-induce-functor-between-tensor-cat}. 
Precomposition with  $(\rusEl,\RusEl)\colon \sr{T}(\fk{T})\to \sr{T}(\fk{S})$ defines a functor $\Rep{\fk{S}}\to \Rep{\fk{T}}$. 
This can also be seen by applying Lemma~\ref{lem-functors-between-reps-of-species}. 
To see this, we construct  a  natural transformation  $ \tennsor{y}{B}{x}\otimes \rusEl^{*}_{x}(-)\Rightarrow \rusEl^{*}_{y}(\tennsor{y}{A}{x}\otimes-)$. 
    Given   $\tennsor{}{L}{x}$ in $\lMod{\tennsor{}{R}{x}}$ define $\tennsor{}{{(\tennsor{y}{\Omega}{x})}}{\tennsor{}{L}{x}}\colon \tennsor{y}{B}{x}\otimes \rusEl^{*}_{x}(\tennsor{}{L}{x})\to \rusEl^{*}_{y}(\tennsor{y}{A}{x}\otimes \tennsor{}{L}{x})$ by considering 
    %
    \[
    \begin{array}{cc}
    \tennsor{y}{B}{x}\times \rusEl^{*}_{x}(\tennsor{}{L}{x})\to \rusEl^{*}_{y}(\tennsor{y}{A}{x}\otimes \tennsor{}{L}{x}),
         & 
         (b,\ell)\mapsto \tennsor{y}{\RusEl}{x}(b)\otimes \ell. 
    \end{array}
    \]
    Since $\tennsor{y}{\RusEl}{x}(sbt)=\tennsor{}{\rusEl}{y}(s)\tennsor{y}{\RusEl}{x}(b)\tennsor{}{\rusEl}{x}(t)$, this map is $S_{y}$-linear and $S_{x}$-balanced. 
    Furthermore, the maps $\tennsor{}{{(\tennsor{y}{\Omega}{x})}}{\tennsor{}{L}{x}}$ are natural in $\tennsor{}{L}{x}$, for if $\lambda\colon \tennsor{}{L}{x}\to \tennsor{}{L}{x}'$ is a morphism in $\lMod{\tennsor{}{R}{x}}$ then 
    \begin{align*}
        \rusEl^{*}_{y}(\tennsor{y}{A}{x}\otimes\lambda)\left(\tennsor{}{{(\tennsor{y}{\Omega}{x})}}{\tennsor{}{L}{x}}(b\otimes \ell)\right)
    =
    \rusEl^{*}_{y}(\tennsor{y}{A}{x}\otimes\lambda)\left(\tennsor{y}{\RusEl}{x}(b)\otimes \ell\right)
    \\
    =
    \tennsor{y}{\RusEl}{x}(b)\otimes \lambda(\ell)
    =\tennsor{}{{(\tennsor{y}{\Omega}{x})}}{\tennsor{}{L}{x}'}(b\otimes \ell)
    =\tennsor{}{{(\tennsor{y}{\Omega}{x})}}{\tennsor{}{L}{x}'}\left(\left(\tennsor{y}{B}{x}\otimes \rusEl^{*}_{x}(\tennsor{}{L}{x}')\right)(b\otimes \ell)\right).
    \end{align*}
\end{rem}

    For a species $\fk{S}=(\tennsor{}{R}{x},\tennsor{y}{A}{x})$  the category $\Rep{\fk{S}}$ is abelian by Lemma~\ref{theorem-equivalence-of-reps-of-species-and-tensor-cat}, so we can consider $\Ext_{\Rep{\fk{S}}}\colon (\Rep{\fk{S}})^{\op}\times \Rep{\fk{S}}\to \Ab$. 
    For representations $L=(\tennsor{x}{L}{},\tennsor{y}{\alpha}{x})$ and $N=(\tennsor{x}{N}{},\tennsor{y}{\gamma}{x})$ the Baer sum defines the addition of the group  $\Ext_{\Rep{\fk{S}}}(N,L)$ of Yoneda equivalence classes $[0\to L\to M\to N\to 0]$ of short exact sequences in $\Rep{\fk{S}}$. 
\begin{nota}
    \label{defn-little-phi-of-little-omega}
    Given  $\omega=(\tennsor{}{\omega}{yx})\in \tennsor{}{\prod }{(x,y)\in\ff{C}^2}\Hom_{{\lMod{\tennsor{}{R}{y}}}}(\tennsor{y}{A}{x}\otimes \tennsor{x}{N}{},\tennsor{y}{L}{})$ we denote 
\[
\varphi(\omega)\colon\quad \begin{tikzcd}[ampersand replacement = \&, column sep = 4em]
    0\to
    (\tennsor{x}{L}{},\tennsor{y}{\alpha}{x})
    \arrow[r, "{\left(\begin{matrix}
        0 \\ \tennsor{}{\iden}{\tennsor{x}{L}{}}
    \end{matrix}\right)}"]
    \&
    \left(\tennsor{x}{N}{}\oplus\tennsor{x}{L}{},{\left(\begin{matrix}
        \tennsor{y}{\gamma}{x} & 0\\ \tennsor{}{\omega}{yx}& \tennsor{y}{\alpha}{x}
\end{matrix}\right)}\right)\arrow[r, "{\left(\begin{matrix}
        \tennsor{}{\iden}{\tennsor{x}{N}{}}& 0
    \end{matrix}\right)}"]
    \&
    (\tennsor{x}{N}{},\tennsor{y}{\gamma}{x})\to
    0.
\end{tikzcd}
\]
\end{nota}

\begin{lem}
\label{lem-pointwise-ext-0-implication}
Let $[\varphi]\in \Ext_{\Rep{\fk{S}}}(N,L)$ for representations $L=(\tennsor{x}{L}{},\tennsor{y}{\alpha}{x})$ and $N=(\tennsor{x}{N}{},\tennsor{y}{\gamma}{x})$ of a species $\fk{S}=(\tennsor{}{R}{x},\tennsor{y}{A}{x})$. 
If $\Ext_{\lMod{\tennsor{}{R}{x}}}(\tennsor{x}{N}{},\tennsor{x}{L}{})=0$ for each $x\in\ff{C}$, then $[\varphi]=[\varphi(\omega)]$ for some $\omega\in \tennsor{}{\prod }{(x,y)\in\ff{C}^2}\Hom_{{\lMod{\tennsor{}{R}{y}}}}(\tennsor{y}{A}{x}\otimes \tennsor{x}{N}{},\tennsor{y}{L}{})$. 
\end{lem}
\begin{proof}
    Let $\varphi\colon 0\to L\to M\to N\to 0$ be exact in $\Rep{\fk{S}}$, defined by $M=(\tennsor{x}{M}{},\tennsor{y}{\beta}{x})$, and morphisms $f=(\tennsor{x}{f}{})\colon L\to M$ and $g=(\tennsor{x}{g}{})\colon M\to N$. 
    By Lemma~\ref{lemma-co-limits-in-rep-of-species-cat}, that $\varphi$ is exact says each $0\to \tennsor{x}{L}{}\to \tennsor{x}{M}{}\to \tennsor{x}{N}{}\to 0$ is exact, and so split exact by assumption. 

    So there exist $\tennsor{}{R}{x}$-module isomorphisms $\tennsor{}{\theta}{x}\colon \tennsor{x}{N}{}\oplus\tennsor{x}{L}{}\to \tennsor{x}{M}{}$, with  inverses $\tennsor{}{\psi}{x}$, such that $\tennsor{}{\psi}{x}\circ \tennsor{x}{f}{}=\tennsor{x}{\iota}{L}$, the canonical inclusion, and $\tennsor{x}{g}{}\circ \tennsor{}{\psi}{x}=\tennsor{x}{\pi}{N}$, the canonical projection. 

Let $K=(\tennsor{x}{N}{}\oplus\tennsor{x}{L}{},\tennsor{y}{\delta}{x})$ where $\tennsor{y}{\delta}{x}=\tennsor{}{\psi}{y}\circ \tennsor{y}{\beta}{x}\circ (\tennsor{}{\iden}{{\tennsor{y}{A}{x}}}\otimes \tennsor{}{\theta}{x})$.   
    By construction the collection $\theta=(\tennsor{}{\theta}{x})\colon M\to K$ is an (iso)morphism in $\Rep{\fk{S}}$, and 
    \[
    \tennsor{y}{\delta}{x}\circ (\tennsor{}{\iden}{{\tennsor{y}{A}{x}}}\otimes \tennsor{x}{\iota}{L})=\tennsor{}{\psi}{y}\circ\tennsor{y}{\beta}{x} \circ (\tennsor{}{\iden}{{\tennsor{y}{A}{x}}}\otimes \tennsor{x}{f}{})
    =\tennsor{}{\psi}{y}\circ \tennsor{y}{f}{}\circ \tennsor{y}{\alpha}{x}
    =
    \tennsor{y}{\iota}{L}\circ\tennsor{y}{\alpha}{x}. 
    \]
    Let $\omega = (\tennsor{}{\omega}{yx})$ where $\tennsor{}{\omega}{yx}=\tennsor{y}{\pi}{L}\circ \tennsor{y}{\delta}{x}\circ (\tennsor{}{\iden}{{\tennsor{y}{A}{x}}}\otimes \tennsor{x}{\iota}{N}))$. 
    It follows that $\varphi(\omega)$ has the form $0\to L\to K\to N\to 0$ and that $\theta$ induces a Yoneda equivalence between  $\varphi$ and $\varphi(\omega)$. 
\end{proof}

Following a result from work of Ringel \cite{Ringelspecies} we can describe $\Ext_{\Rep{\fk{S}}}(N,L)$  explicitly. 

\begin{nota}
    For representations $L=(\tennsor{x}{L}{},\tennsor{y}{\alpha}{x}),N=(\tennsor{x}{N}{},\tennsor{y}{\gamma}{x})$  of $\fk{S}=(\tennsor{}{R}{x},\tennsor{y}{A}{x})$ let
   \begin{equation}
       \label{eqn-ext-map}
       \begin{aligned}
            \tennsor{N}{\Lambda}{L} \colon \tennsor{}{\prod }{x\in\ff{C}}\Hom_{{\lMod{\tennsor{}{R}{x}}}}(\tennsor{x}{N}{},\tennsor{x}{L}{})\to \tennsor{}{\prod }{(x,y)\in\ff{C}^2}\Hom_{{\lMod{\tennsor{}{R}{y}}}}(\tennsor{y}{A}{x}\otimes \tennsor{x}{N}{},\tennsor{y}{L}{}),
\\
\left(\tennsor{}{\theta}{x}\right)
\longmapsto 
\left[a\otimes n\mapsto 
\tennsor{y}{\alpha}{x}(a\otimes \tennsor{}{\theta}{x}(n) )-\tennsor{}{\theta}{y} (\tennsor{y}{\gamma}{x}(a\otimes n))\right].
       \end{aligned}
   \end{equation}
\end{nota}

\begin{lem}
    \label{lem-describing-ext}
 If $L$ and $N$ are representations of a species $\fk{S}$ then 
 \[
 \begin{array}{cc}
  \ker(\tennsor{N}{\Lambda}{L})=\Hom_{\Rep{\fk{S}}}(N,L),
  &
  \cok(\tennsor{N}{\Lambda}{L})\hookrightarrow \Ext_{\Rep{\fk{S}}}(N,L). 
 \end{array}
 \]
 Furthermore, if $\Ext_{\lMod{\tennsor{}{R}{x}}}(\tennsor{x}{N}{},\tennsor{x}{L}{})=0$ for each $x\in\ff{C}$, then the embedding is onto.  
\end{lem}

\begin{proof}
For completeness we repeat details from the proof of \cite[\S2.1,~p.277,~Lemma]{Ringelspecies}. 
By construction, $\ker(\tennsor{N}{\Lambda}{L})=\Hom_{\Rep{\fk{S}}}(N,L)$.
Let $L=(\tennsor{x}{L}{},\tennsor{y}{\alpha}{x})$, $N=(\tennsor{x}{N}{},\tennsor{y}{\beta}{x})$, and  
\[
\begin{array}{c}
\tennsor{N}{\Phi}{L}\colon \tennsor{}{\prod }{x,y\in\ff{C}}\Hom_{{\lMod{\tennsor{}{R}{y}}}}(\tennsor{y}{A}{x}\otimes \tennsor{x}{N}{},\tennsor{y}{L}{})\to \Ext_{\Rep{\fk{S}}}(N,L),\quad \omega=\left(\tennsor{}{\omega}{yx}\right)\mapsto [\varphi(\omega)]. 
\end{array}
\]
as in Lemma~\ref{lem-pointwise-ext-0-implication}. 
If $\tennsor{}{\omega}{yx}=0$ for each $x,y\in\ff{C}$ then $\varphi(\omega)$ is the canonical split exact sequence $0\to L\to N\oplus L\to N\to 0$, so $\tennsor{N}{\Phi}{L}$ is well-defined. 
Conversley, suppose $[\varphi(\omega)]=0$. 
By definition of Yoneda equivalence, there exists an isomorphism 
\[
f=(\tennsor{x}{f}{})=\left(\left(\begin{matrix}
        \tennsor{}{\mu}{x} & \tennsor{}{\lambda}{x} \\\tennsor{}{\theta}{x} & \tennsor{}{\eta}{x}
\end{matrix}\right)
\right) \colon\left(\tennsor{x}{N}{}\oplus\tennsor{x}{L}{},{\left(\begin{matrix}
        \tennsor{y}{\gamma}{x} & 0\\ \tennsor{}{\omega}{yx}& \tennsor{y}{\alpha}{x}
\end{matrix}\right)}\right)\to \left(\tennsor{x}{N}{}\oplus\tennsor{x}{L}{},{\left(\begin{matrix}
        \tennsor{y}{\gamma}{x} & 0\\  0& \tennsor{y}{\alpha}{x}
\end{matrix}\right)}\right),
\]
  such that, by Lemma~\ref{lemma-co-limits-in-rep-of-species-cat}, and in the notation of the proof of Lemma~\ref{lem-pointwise-ext-0-implication}, for each $x\in\ff{C}$ we have $\tennsor{x}{f}{}\circ\tennsor{x}{\iota}{N} =\tennsor{x}{\iota}{N}$ and $\tennsor{x}{\pi}{L}= \tennsor{x}{\pi}{L}\circ\tennsor{x}{f}{}$, meaning $\tennsor{}{\mu}{x}=\tennsor{}{\iden}{\tennsor{x}{N}{}}$, $\tennsor{}{\lambda}{x}=0$ and $\tennsor{}{\eta}{x}=\tennsor{}{\iden}{\tennsor{x}{L}{}}$.   

    Furthermore, as $f$ is a homomorphism of representations, 
 $\tennsor{N}{\Lambda}{L}((\tennsor{}{\theta}{x}))=\omega$. 
    A similar argument shows that $[\varphi(\omega)]=0$ for any $\omega\in\im(\tennsor{N}{\Lambda}{L})$.  
   Now apply Lemma~\ref{lem-pointwise-ext-0-implication}.  
\end{proof}

Let $\fk{S}=(\tennsor{}{R}{x},\tennsor{y}{A}{x})$ be a species.
Using that tensor products and direct sums commute, we can and do define and denote the canonical inclusions by
\begin{align*}
    \tennsor{yx}{\varsigma}{\ell}\colon \tennsor{y}{A}{x}\otimes \tennsor{x\,}{\underline{A}}{\,\ell}
\cong \left(\bigoplus_{\ff{n}\geq 0,\,\underline{u}\in \ff{C}^{\ff{n}}(x,\ell)} \tennsor{y}{A}{x}
\otimes \tennsor{}{A}{\underline{u}}\right)
= \left(\bigoplus_{\ff{n}>0,\,\underline{v}=(y,x,\dots,\ell)\in \ff{C}^{\ff{n}}(y,\ell)} \tennsor{}{A}{\underline{v}}\right)
\subseteq 
\tennsor{x\,}{\underline{A}}{\,\ell}. 
\end{align*}

For $\ell\in \ff{C}$, and for  each left $\tennsor{}{R}{\ell}$-module $X$, we let 
\[
 \begin{array}{cc}
\tennsor{?\,}{\underline{A}}{\,\ell}=(\tennsor{x\,}{\underline{A}}{\,\ell},\tennsor{yx}{\varsigma}{\ell}),
&
\tennsor{?\,}{\underline{A}}{\,\ell}\otimes X=(\tennsor{x}{\underline{A}}{\,\ell}\otimes X,\tennsor{yx}{\varsigma}{\ell}\otimes \tennsor{}{\iden}{X}).
\end{array}
\]

\begin{lem}
\label{lem-when-proj}
Let $\ell\in\ff{C}$ and $\fk{S}=(\tennsor{}{R}{x},\tennsor{y}{
A
}{x})$ be a species such that each left $\tennsor{}{R}{y}$-module $\tennsor{y}{
A
}{x}$ is   projective. 
 If $\tennsor{}{P}{\ell}$ is a projective left $\tennsor{}{R}{\ell}$-module then  $\tennsor{?\,}{\underline{A}}{\,\ell}\otimes \tennsor{}{P}{\ell}$ is projective in $\Rep{\fk{S}}$.
\end{lem}
\begin{proof}
Observe, for any rings $S$ and $T$, any $S$-$T$-bimodule $A$ and any projective left $T$-module $V$, that if $A$ is  projective  as a left $S$-module, then so is the tensor project  $A\otimes V$.  
%
%

It follows by induction that, for any path $\underline{w}=(w_{\ff{n}},\dots,w_{0})\in\ff{C}^{\ff{n}}(x,\ell)$, the left $\tennsor{}{R}{x}$-module $\tennsor{}{A}{\underline{w}}\otimes \tennsor{}{P}{\ell}$ is projective. 
Let $N=\tennsor{?\,}{\underline{A}}{\,\ell}\otimes \tennsor{}{P}{\ell}$. 
It is necessary and sufficient to prove that $\Ext_{\Rep{\fk{S}}}(N,L)=0$ for any representation $L$. 
As observed, the $\tennsor{}{R}{x}$-module $\tennsor{x}{N}{}=\tennsor{x\,}{\underline{A}}{\,\ell}\otimes \tennsor{}{P}{\ell}$, a coproduct of modules of the form $\tennsor{}{A}{\underline{w}}\otimes \tennsor{}{P}{\ell}$, is  projective. 
So $\Ext_{\lMod{\tennsor{}{R}{x}}}(\tennsor{x}{N}{},\tennsor{x}{L}{})=0$ for each $x\in\ff{C}$. 
    By Lemma~\ref{lem-describing-ext} it suffices to prove that the map $\tennsor{N}{\Lambda}{L} $ from \eqref{eqn-ext-map} is onto. 

    Suppose we have $\tennsor{}{\omega}{yx}\in \Hom_{{\lMod{\tennsor{}{R}{y}}}}(\tennsor{y}{A}{x}\otimes \tennsor{x}{N}{},\tennsor{y}{L}{})$ for each $x,y\in\ff{C}$. 
    If $x\in \ff{C}$ define an   $\tennsor{}{R}{x}$-linear map $\tennsor{}{\theta}{x}\colon \tennsor{x}{N}{}\to \tennsor{x}{L}{}$ by the universal property of the coproduct over $\underline{v}=(\tennsor{}{v}{\ff{n}},\dots,\tennsor{}{v}{0})\in \ff{C}^{\ff{n}}(x,\ell)$. 
    In case $\ff{n}=0$, $\underline{v}=(\ell)$, and we let $\tennsor{}{{(\tennsor{}{\theta}{x})}}{(\ell)}=0$. 
    Suppose  $\ff{n}\geq 1$. 
    Define 
    $\tennsor{}{{(\tennsor{}{\theta}{x})}}{\underline{v}}\colon \tennsor{}{A}{\underline{v}}\otimes P\to \tennsor{x}{L}{}$ as follows. 
    Given $\underline{c}=\tennsor{}{c}{{\tennsor{}{xv}{\ff{n}-1}}}
    \otimes
    \cdots
    \otimes \tennsor{}{c}{{\tennsor{}{v}{1}\ell}}\in  \tennsor{}{A}{\underline{v}}$ let $\tennsor{}{\underline{c}}{\,\ff{i}\geq }=\tennsor{}{c}{{\tennsor{}{v}{\ff{i}}\tennsor{}{v}{\ff{i}-1}}}\otimes
    \cdots
    \otimes \tennsor{}{c}{{\tennsor{}{v}{1}\ell}}$ when $1\leq \ff{i}\leq \ff{n}$ and let $\tennsor{}{\underline{c}}{\,\geq \ff{i}}=\tennsor{}{c}{{\tennsor{}{xv}{\ff{n}-1}}}
    \otimes
    \cdots
                \otimes \tennsor{}{c}{{\tennsor{}{v}{\ff{i}+1}\tennsor{}{v}{\ff{i}}}}$ when $0\leq \ff{i}\leq \ff{n}-1$. 
                So $\tennsor{}{\underline{c}}{\,\geq 0}=\underline{c}=\tennsor{}{\underline{c}}{\,\ff{n}\geq }$. 
                Let $\tennsor{}{\underline{c}}{\,0\geq }=\tennsor{}{1}{\ell}$ and $\tennsor{}{\underline{c}}{\,\geq\ff{n} }=\tennsor{}{1}{x}$.
                In this notation, and for each  $p\in \tennsor{}{P}{\ell}$, we let
    \[
    \tennsor{}{{(\tennsor{}{\theta}{x})}}{\underline{v}}(\underline{c}\otimes m)=\tennsor{}{\omega}{{\tennsor{}{xv}{\ff{n}-1}}}(\tennsor{}{c}{{\tennsor{}{xv}{\ff{n}-1}}}\otimes \tennsor{}{\underline{c}}{\,\ff{n}-1\geq }\otimes p)+
    \sum^{\ff{n}-1}_{\ff{i}=1}
     \tennsor{x\,}{\underline{\alpha}}{{\tennsor{}{v}{\,\ff{i}}}}(\tennsor{}{\underline{c}}{\,\geq \ff{i}} \otimes \tennsor{}{\omega}{{\tennsor{}{v}{\ff{i}}\tennsor{}{v}{\ff{i}-1}}}(\tennsor{}{c}{{\tennsor{}{v}{\ff{i}}\tennsor{}{v}{\ff{i}-1}}}\otimes \tennsor{}{\underline{c}}{\,\ff{i}-1\geq }\otimes p)).
    \]
We claim, for each $a\in\tennsor{y}{A}{x}$, that $\tennsor{}{\omega}{yx}(a\otimes \underline{c}\otimes p)=\tennsor{}{\theta}{y}(a\otimes \underline{c}\otimes p)-\tennsor{y}{\alpha}{x}(a\otimes \tennsor{}{\theta}{x}(\underline{c}\otimes p))$. 
On the one hand, adjusting the expression above, by replacing $\underline{c}$ with $\underline{b}=a\otimes \underline{c}$, yields
  \[
    \tennsor{}{\omega}{yx}(a\otimes \tennsor{}{\underline{c}}{}\otimes p)+
    \sum^{\ff{n}}_{\ff{i}=1}
     \tennsor{y\,}{\underline{\alpha}}{{\tennsor{}{v}{\,\ff{i}}}}(a\otimes\tennsor{}{\underline{c}}{\,\geq \ff{i}} \otimes \tennsor{}{\omega}{{\tennsor{}{v}{\ff{i}}\tennsor{}{v}{\ff{i}-1}}}(\tennsor{}{c}{{\tennsor{}{v}{\ff{i}}\tennsor{}{v}{\ff{i}-1}}}\otimes \tennsor{}{\underline{c}}{\,\ff{i}-1\geq }\otimes p)).
    \]
On the other hand, tensoring the expression on the left by $a$, and applying $\tennsor{y}{\alpha}{x}$, yields
\begin{align*}
    \tennsor{y}{\alpha}{x}(a\otimes \tennsor{}{\omega}{{\tennsor{}{xv}{\ff{n}}}}(\tennsor{}{c}{{\tennsor{}{xv}{\ff{n}-1}}}\otimes \tennsor{}{\underline{c}}{\,\ff{n}-1\geq }\otimes p))+
    \sum^{\ff{n}-1}_{\ff{i}=1}
     \tennsor{y}{\alpha}{x}(a\otimes\tennsor{x\,}{\underline{\alpha}}{{\tennsor{}{v}{\,\ff{i}}}}(\tennsor{}{\underline{c}}{\,\geq \ff{i}} \otimes \tennsor{}{\omega}{{\tennsor{}{v}{\ff{i}}\tennsor{}{v}{\ff{i}-1}}}(\tennsor{}{c}{{\tennsor{}{v}{\ff{i}}\tennsor{}{v}{\ff{i}-1}}}\otimes \tennsor{}{\underline{c}}{\,\ff{i}-1\geq }\otimes p))).
\end{align*}
The claim follows by noting that, for each $\ff{i}\leq \ff{n}-1$, we have
\begin{align*}
    \tennsor{y}{\alpha}{x}(a\otimes \tennsor{}{\omega}{{\tennsor{}{xv}{\ff{n}-1}}}(\tennsor{}{c}{{\tennsor{}{xv}{\ff{n}-1}}}\otimes \tennsor{}{\underline{c}}{\,\ff{n}-1\geq }\otimes p))=\tennsor{y\,}{\alpha}{{\tennsor{}{v}{\,\ff{n}}}}(\tennsor{}{\underline{b}}{\,\geq \ff{n}} \otimes \tennsor{}{\omega}{{\tennsor{}{v}{\ff{n}}\tennsor{}{v}{\ff{n}-1}}}(\tennsor{}{b}{{\tennsor{}{v}{\ff{n}}\tennsor{}{v}{\ff{n}-1}}}\otimes \tennsor{}{\underline{b}}{\,\ff{n}-1\geq }\otimes p)) ,
    \\
    \tennsor{y}{\alpha}{x}(a\otimes\tennsor{x\,}{\underline{\alpha}}{{\tennsor{}{v}{\,\ff{i}}}}(\tennsor{}{\underline{c}}{\,\geq \ff{i}} \otimes \tennsor{}{\omega}{{\tennsor{}{v}{\ff{i}}\tennsor{}{v}{\ff{i}-1}}}(\tennsor{}{c}{{\tennsor{}{v}{\ff{i}}\tennsor{}{v}{\ff{i}-1}}}\otimes \tennsor{}{\underline{c}}{\,\ff{i}-1\geq }\otimes p)))= \tennsor{y\,}{\underline{\alpha}}{{\tennsor{}{v}{\,\ff{i}}}}(\tennsor{}{\underline{b}}{\,\geq \ff{i}} \otimes \tennsor{}{\omega}{{\tennsor{}{v}{\ff{i}}\tennsor{}{v}{\ff{i}-1}}}(\tennsor{}{b}{{\tennsor{}{v}{\ff{i}}\tennsor{}{v}{\ff{i}-1}}}\otimes \tennsor{}{\underline{b}}{\,\ff{i}-1\geq }\otimes p))).
\end{align*} 
It follows, from the claim, that $\tennsor{N}{\Lambda}{L}((\tennsor{}{\theta}{x}))=(\tennsor{}{\omega}{yx})$, completing the proof.  
\end{proof}

\begin{lem}
\label{prop-first-term-standard-res}
    Any representation $M$ of a species $\fk{S}$ admits an exact sequence   
    \[
    \begin{tikzcd}
        0\arrow[r]
        &
        \tennsor{}{{\bigoplus}}{j,k\in\ff{C}}\kern 1mm \tennsor{?\,}{\underline{A}}{\,j}\otimes\tennsor{j}{A}{k}\otimes  \tennsor{k}{M}{}
        \arrow[r]
        &
        \tennsor{}{\bigoplus}{\ell\in\ff{C}}\kern .5mm \tennsor{?\,}{\underline{A}}{\,\ell}\otimes \tennsor{\ell}{M}{}
        \arrow[r]
        &
        M\arrow[r]
        &
        0.
    \end{tikzcd}
    \]
\end{lem}

    \begin{proof}
Let $M=(\tennsor{x}{M}{},\tennsor{y}{\alpha}{x})$, $P=  (\tennsor{x}{P}{},\tennsor{y}{\beta}{x})$ where $P= \tennsor{}{\bigoplus}{\ell\in\ff{C}}\tennsor{?\,}{\underline{A}}{\,\ell}\otimes \tennsor{\ell}{M}{}$, and $Q=  (\tennsor{x}{Q}{},\tennsor{y}{\gamma}{x})$ where $Q=\tennsor{}{{\bigoplus}}{j,k\in\ff{C}}\tennsor{?\,}{\underline{A}}{\,j}\otimes\tennsor{j}{A}{k}\otimes  \tennsor{k}{M}{}$. 
For each $\ell$ note that $\tennsor{?\,}{\underline{\alpha}}{\,\ell}=(\tennsor{x\,}{\underline{\alpha}}{\,\ell})$ defines a morphism $\tennsor{?\,}{\underline{A}}{\,\ell}\otimes \tennsor{\ell}{M}{}\to M$ in $\Rep{\fk{S}}$.
%
    %
    Define $f\colon P\to M$ unique  such that $f\circ \tennsor{}{\iota}{\ell}=\tennsor{?\,}{\underline{\alpha}}{\,\ell}$ where 
    $\tennsor{}{\iota}{\ell}\colon\tennsor{?\,}{\underline{A}}{\,\ell}\otimes \tennsor{\ell}{M}{}\to P$ is the canonical monomorphism. 
    Consider the canonical morphisms
    \begin{align*}
         \tennsor{h}{\varsigma}{p t}\colon \tennsor{h\,}{\underline{A}}{\,p}\otimes\tennsor{p}{A}{t}
\cong {\bigoplus}_{\ff{n}\geq 0,\,\underline{u}\in \ff{C}^{\ff{n}}(h,p)}  \tennsor{}{A}{\underline{u}}\otimes \tennsor{p}{A}{t}
={\bigoplus}_{\ff{n}>0,\,\underline{v}=(h,\dots,p,t)\in \ff{C}^{\ff{n}}(h,t)} \tennsor{}{A}{\underline{v}}
\subseteq 
\tennsor{h\,}{\underline{A}}{\,t}. 
    \end{align*}

    Define $g=(\tennsor{x}{g}{})\colon Q\to P$ by, for each $\underline{a}\in\tennsor{x\,}{\underline{A}}{\,\ell}$, $b\in \tennsor{j}{A}{k}$ and $m\in \tennsor{k}{M}{}$, setting
\[
\tennsor{x}{g}{}(\underline{a}\otimes b\otimes m) =\tennsor{x}{\varsigma}{jk}(\underline{a}\otimes b)\otimes m-\underline{a}\otimes \tennsor{j}{\alpha}{k}(b\otimes m)\in \left(\tennsor{x\,}{\underline{A}}{\,j}\otimes \tennsor{j}{M}{}\right)\oplus \left(\tennsor{x\,}{\underline{A}}{\,k}\otimes \tennsor{k}{M}{}\right). 
\]
In this notation, and for each $c\in\tennsor{y}{A}{x}$, we have
\begin{align*}
    \tennsor{y}{\beta}{x}((\tennsor{}{\iden}{\tennsor{y}{A}{x}}\otimes \tennsor{x}{g}{})(c\otimes \underline{a}\otimes b\otimes m))
     =
    \tennsor{y}{\beta}{x}(c\otimes (\tennsor{x}{\varsigma}{jk}(\underline{a}\otimes b)\otimes m
     -
     \underline{a}\otimes \tennsor{j}{\alpha}{k}(b\otimes m)))
     \\
       =
    (\tennsor{yx}{\varsigma}{k}\otimes \tennsor{}{\iden}{{\tennsor{k}{M}{}}})(c\otimes \tennsor{x}{\varsigma}{jk}(\underline{a}\otimes b)\otimes m)
     -
     (\tennsor{yx}{\varsigma}{j}\otimes \tennsor{}{\iden}{{\tennsor{j}{M}{}}})(c\otimes\underline{a}\otimes \tennsor{j}{\alpha}{k}(b\otimes m))
     \\
    =\tennsor{y}{g}{}(\tennsor{yx}{\varsigma}{j}(c\otimes \underline{a})\otimes b\otimes m)=
    \tennsor{y}{g}{}(\tennsor{y}{\gamma}{x}(c\otimes \underline{a}\otimes b\otimes m)).
    \end{align*}
    Thus $g$ is a morphism of representations. 
Notice $f\circ g=0$. 
To see that $0\to Q\to P\to M\to 0$ is exact, it is enough to see that each $0\to \tennsor{x}{Q}{}\to \tennsor{x}{P}{}\to \tennsor{x}{M}{}\to 0$  is exact, by Lemma~\ref{lemma-co-limits-in-rep-of-species-cat}. 
It suffice to prove that each such complex of $\tennsor{}{R}{x}$-modules is  contractible, meaning the identity map is null homotopic; see for example \cite[Exercise~1.4.3]{Weibel-introduction-to-homological-algebra}.

Define $\tennsor{}{p}{x}\colon \tennsor{x}{M}{}\to \tennsor{x}{P}{}$ by $\tennsor{}{p}{x}(m)=\tennsor{}{1}{x}\otimes m\in \tennsor{}{R}{x}\otimes \tennsor{x}{M}{}\subseteq \tennsor{x\,}{\underline{A}}{\,x}\otimes \tennsor{x}{M}{}$ for each $m\in \tennsor{x}{M}{}$. 
For each $\ell\in\ff{C}$ and each $\underline{w}=(\tennsor{}{w}{\ff{n}},\dots,\tennsor{}{w}{0})\in\ff{C}^{\ff{n}}(x,\ell)$, we define  an $\tennsor{}{R}{x}$-module map $\tennsor{\underline{w}}{q}{x}\colon \tennsor{}{A}{\underline{w}}\otimes \tennsor{\ell}{M}{}\to \tennsor{x}{Q}{}$. 
If $\ff{n}=0$, so if $\ell=x$ and $\underline{w}=(x)$, let $\tennsor{\underline{w}}{q}{x}=0$. 
If  $\ff{n}\geq 1$, let
\[
\tennsor{\underline{w}}{q}{x}(\underline{a}\otimes m)
=
 \sum_{\ff{i}=0}^{\ff{n}-1}
(\tennsor{}{a}{w_{\ff{n}}w_{\ff{n}-1}}\otimes \cdots \otimes \tennsor{}{a}{w_{\ff{i}+2}w_{\ff{i}+1}})\otimes \tennsor{}{a}{w_{\ff{i}+1}w_{\ff{i}}}\otimes \tennsor{w_{\ff{i}}\,}{\underline{\alpha}}{\,\ell}(\tennsor{}{a}{w_{\ff{i}}w_{\ff{i}-1}}\otimes \cdots \otimes \tennsor{}{a}{w_{1}w_{0}}\otimes m)
\]
for each $\underline{a}=\tennsor{}{a}{w_{\ff{n}}w_{\ff{n}-1}}\otimes \cdots \otimes \tennsor{}{a}{w_{1}w_{0}}\in\tennsor{}{A}{\underline{w}}$ and each $m\in \tennsor{\ell}{M}{}$.
By combining the maps $\tennsor{\underline{w}}{q}{x}$ as $\underline{w}$ runs through the paths, we define an $\tennsor{}{R}{x}$-linear map $\tennsor{}{q}{x}\colon \tennsor{x}{P}{}\to \tennsor{x}{Q}{}$. 
Since $\tennsor{x\,}{\underline{\alpha}}{\,x}(\tennsor{}{1}{x}\otimes m)=m$ we have that $\tennsor{x}{f}{}\circ \tennsor{}{p}{x}=\tennsor{}{\iden}{\tennsor{x}{M}{}}$. 
    We claim $\tennsor{}{q}{x}\circ \tennsor{x}{g}{}=\tennsor{}{\iden}{\tennsor{x}{Q}{}}$. 
    %
    %
    Suppose $\underline{a}\in\tennsor{}{A}{\underline{w}}$ $\underline{w}=(\tennsor{}{w}{\ff{n}},\dots,\tennsor{}{w}{0})\in\ff{C}^{\ff{n}}(x,j)$.
    %
When $\ff{n}=0$ the claim is straightforward.
So, assume $\ff{n}\geq  1$, and write $\underline{a}=\tennsor{}{a}{w_{\ff{n}}w_{\ff{n}-1}}\otimes \cdots \otimes \tennsor{}{a}{w_{1}w_{0}}$. 
    Let $k=w_{-1}$ and $b= \tennsor{}{a}{w_{0}w_{-1}}$. 
    In this notation, 
    \begin{align*}
       \tennsor{}{q}{x}(\tennsor{x}{g}{}(\underline{a}\otimes b\otimes m))= \tennsor{}{q}{x}(\underline{a}\otimes b\otimes m)
        -\tennsor{}{q}{x}(\underline{a}\otimes \tennsor{j}{\alpha}{k}(b\otimes m))
        \\
        = \sum_{\ff{i}=-1}^{\ff{n}-1}
(\tennsor{}{a}{w_{\ff{n}}w_{\ff{n}-1}}\otimes \cdots \otimes \tennsor{}{a}{w_{\ff{i}+2}w_{\ff{i}+1}})\otimes \tennsor{}{a}{w_{\ff{i}+1}w_{\ff{i}}}\otimes \tennsor{w_{\ff{i}}\,}{\underline{\alpha}}{\,k}(\tennsor{}{a}{w_{\ff{i}}w_{\ff{i}-1}}\otimes \cdots \otimes \tennsor{}{a}{w_{0}w_{-1}}\otimes m)
\\
        - \sum_{\ff{i}=0}^{\ff{n}-1}
(\tennsor{}{a}{w_{\ff{n}}w_{\ff{n}-1}}\otimes \cdots \otimes \tennsor{}{a}{w_{\ff{i}+2}w_{\ff{i}+1}})\otimes \tennsor{}{a}{w_{\ff{i}+1}w_{\ff{i}}}\otimes \tennsor{w_{\ff{i}}\,}{\underline{\alpha}}{\,j}(\tennsor{}{a}{w_{\ff{i}}w_{\ff{i}-1}}\otimes \cdots \otimes \tennsor{}{a}{w_{1}w_{0}}\otimes \tennsor{j}{\alpha}{k}(b\otimes m))
 \\
=(\tennsor{}{a}{w_{\ff{n}}w_{\ff{n}-1}}\otimes \cdots \otimes \tennsor{}{a}{w_{1}w_{0}})\otimes \tennsor{}{a}{w_{0}w_{-1}}\otimes m=\underline{a}\otimes b\otimes m.
    \end{align*}
This proves the claim. 
Similarly, it follows that $\tennsor{x}{g}{}(\tennsor{\underline{w}}{q}{x}(\underline{a}\otimes m))$ is equal to
    \begin{align*}
 \sum_{\ff{i}=0}^{\ff{n}-1}
\tennsor{x}{\varsigma}{w_{\ff{i}+1}w_{\ff{i}}}((\tennsor{}{a}{w_{\ff{n}}w_{\ff{n}-1}}\otimes \cdots \otimes \tennsor{}{a}{w_{\ff{i}+2}w_{\ff{i}+1}})\otimes \tennsor{}{a}{w_{\ff{i}+1}w_{\ff{i}}})\otimes \tennsor{w_{\ff{i}}\,}{\underline{\alpha}}{\,\ell}(\tennsor{}{a}{w_{\ff{i}}w_{\ff{i}-1}}\otimes \cdots \otimes \tennsor{}{a}{w_{1}w_{0}}\otimes m)
\\
-
 \sum_{\ff{i}=0}^{\ff{n}-1}
(\tennsor{}{a}{w_{\ff{n}}w_{\ff{n}-1}}\otimes \cdots \otimes \tennsor{}{a}{w_{\ff{i}+2}w_{\ff{i}+1}})\otimes \tennsor{w_{\ff{i}+1}}{\alpha}{w_{\ff{i}}}(\tennsor{}{a}{w_{\ff{i}+1}w_{\ff{i}}}\otimes \tennsor{w_{\ff{i}}\,}{\underline{\alpha}}{\,\ell}(\tennsor{}{a}{w_{\ff{i}}w_{\ff{i}-1}}\otimes \cdots \otimes \tennsor{}{a}{w_{1}w_{0}}\otimes m))
\\
=
\underline{a}\otimes m +
 \sum_{\ff{i}=1}^{\ff{n}-1}
\tennsor{}{a}{w_{\ff{n}}w_{\ff{n}-1}}\otimes \cdots \otimes \tennsor{}{a}{w_{\ff{i}+2}w_{\ff{i}+1}}\otimes \tennsor{}{a}{w_{\ff{i}+1}w_{\ff{i}}}\otimes \tennsor{w_{\ff{i}}\,}{\underline{\alpha}}{\,\ell}(\tennsor{}{a}{w_{\ff{i}}w_{\ff{i}-1}}\otimes \cdots \otimes \tennsor{}{a}{w_{1}w_{0}}\otimes m)
\\
-\tennsor{x\,}{\underline{\alpha}}{\,\ell}(\underline{a}\otimes m)-
 \sum_{\ff{i}=0}^{\ff{n}-2}
 (\tennsor{}{a}{w_{\ff{n}}w_{\ff{n}-1}}\otimes \cdots \otimes \tennsor{}{a}{w_{\ff{i}+2}w_{\ff{i}+1}})\otimes  \tennsor{w_{\ff{i}+1}\,}{\underline{\alpha}}{\,\ell}(\tennsor{}{a}{w_{\ff{i}+1}w_{\ff{i}}}\otimes  \cdots \otimes \tennsor{}{a}{w_{1}w_{0}}\otimes m).
    \end{align*}
Thus  $\tennsor{x}{g}{}\circ \tennsor{}{q}{x}+\tennsor{}{p}{x}\circ \tennsor{x}{f}{}=\tennsor{}{\iden}{\tennsor{x}{P}{}}$, giving the required homotopy.
    \end{proof}
\begin{cor}\label{cor-yay-hereditary}
    If $\fk{S}=(\tennsor{}{R}{x},\tennsor{y}{A}{x})$ is a species such that each ring 
    $\tennsor{}{R}{x} $ is semisimple artinian (such as a field or division ring), then the category $\Rep{\fk{S}}$ is hereditary.
\end{cor}

\begin{proof}
Consider the exact sequence from Proposition~\ref{prop-first-term-standard-res}. 
The semisimplicity assumption means each $\tennsor{y}{A}{x}$ is projective on either side, and also that, for each representation $M=(\tennsor{x}{M}{},\tennsor{y}{\beta}{x})$ the module $\tennsor{x}{M}{}$ is projective. 
By Proposition~\ref{prop-first-term-standard-res} there is a resolution of $M$ of the form $0\to Q\to P\to M\to 0$. 
By Lemma~\ref{lem-when-proj}, both $P$ and $Q$ are projective. 
\end{proof}

Objects in  the \emph{additive closure} $\tennsor{}{\sr{T}}{\oplus}(\fk{S})$ of  $\sr{T}(\fk{S})$ are  paths $\underline{x}$.
Morphisms in $\tennsor{}{\sr{T}}{\oplus}(\fk{S})$ are matrices of morphisms in $\sr{T}(\fk{S})$, and composition is   matrix multiplication,
so
\[
\begin{array}{cc}
\Hom_{\tennsor{}{\sr{T}}{\oplus}(\fk{S})}(\underline{w},\underline{x})=\{(\tennsor{}{f}{ij})\mid \tennsor{}{f}{ij}\in \Hom_{\sr{T}(\fk{S})}(\tennsor{}{w}{j},\tennsor{}{x}{i})\},
&
(\tennsor{}{f}{ij})\circ(\tennsor{}{g}{ij})=(\tennsor{}{\sum}{k}\tennsor{}{f}{ik}\circ\tennsor{}{g}{kj}). 
\end{array}
\]

Recall $\tennsor{w}{\Pi}{}$ from \eqref{eqn-w-Pi}.
\begin{lem}
\label{lem-finitely-presented-trivial}
    Let $\fk{S}=(\tennsor{}{R}{x},\tennsor{y}{A}{x})$ be a species.  Then there is an additive equivalence 
    \[
    \lMod{\tennsor{}{\sr{T}}{\oplus}(\fk{S})}\to \Rep{\fk{S}}
    \]
    such that each $\tennsor{?\,}{\underline{A}}{\,\ell}$ is the image of $\Hom_{\tennsor{}{\sr{T}}{\oplus}(\fk{S})}((\ell),-)$, and  there is a natural isomorphism   \[
        \begin{array}{cc}
        \Hom_{\Rep{\fk{S}}}(\tennsor{?\,}{\underline{A}}{\,\ell},-)\to \tennsor{\ell}{\Pi}{},
        &
        \Hom_{\Rep{\fk{S}}}(\tennsor{?\,}{\underline{A}}{\,\ell},M)\ni f=(\tennsor{x}{f}{})\mapsto \tennsor{\ell}{f}{}(\tennsor{}{1}{\ell})\in\tennsor{\ell}{M}{},
         \end{array}
        \]  
\end{lem}

\begin{proof}
Any additive functor leaving $\sr{T}(\fk{S})$ factors uniquely (up to natural isomorphism) through the canonical functor $ \sr{T}(\fk{S})\to \tennsor{}{\sr{T}}{\oplus}(\fk{S})$ taking  $x\in\ff{C}$ to the length-$0$ sequence $(x)$, and taking $\underline{a}\in \Hom_{\tennsor{}{\sr{T}}{\oplus}(\fk{S})}((w),(x))$ to the $1\times 1$ matrix $(\underline{a})$. 
%
%
This gives an additive  equivalence $\lMod{\sr{T}(\fk{S})}\to \lMod{\tennsor{}{\sr{T}}{\oplus}(\fk{S})}$; see \cite[p.~198,~Exercise~6(a)]{MacLane-categories-for-the-working-mathematician}. 

Combined with Lemma~\ref{theorem-equivalence-of-reps-of-species-and-tensor-cat} this gives the required equivalence. 
By construction, each $\tennsor{?\,}{\underline{A}}{\,\ell}$ is the image of $\Hom_{\tennsor{}{\sr{T}}{\oplus}(\fk{S})}((\ell),-)$. 
The final statement is Yoneda's Lemma.   
\end{proof}

\section{Existence and uniqueness of strong decompositions}
\label{section-decomposition}

An object $M$ of an additive category $\cl{A}$ is \emph{strongly indecomposable} if $\End_{\cl{A}}(X)$ is local. 

\begin{lem}
    \label{lem-uniqueness-of-decompositions}
    Let $\fk{S}=(\tennsor{}{R}{x},\tennsor{y}{A}{x})$ be a species. 
    If $\tennsor{}{\bigoplus}{i\in I}\tennsor{}{M}{i}\cong \tennsor{}{\bigoplus}{j\in J}\tennsor{}{N}{i}$ for sets $I$ and $J$ where each $\tennsor{}{M}{i}$ and each $\tennsor{}{N}{j}$ are strongly indecomposable representations of $\fk{S}$, then there is a bijection $I\to J$ such that $\tennsor{}{M}{i}\cong \tennsor{}{N}{j}$ whenever $i\mapsto j$. 
\end{lem}

\begin{proof}
    The uniqueness of decompositions into strongly indecomposable objects, in an ambient  Grothendieck abelian category such as $\lMod{\sr{T}(\fk{S})}$, is well-known as the Krull--Remak--Schmidt--Azumaya theorem; see for example \cite[Theorem~6.44]{bucur-deleanu-book}. 
    Hence the result follows from Lemma~\ref{theorem-equivalence-of-reps-of-species-and-tensor-cat} and Lemma~\ref{lemma-co-limits-in-rep-of-species-cat}. 
\end{proof}


\begin{defn}
    
An object $M$ of an additive category $\cl{A}$ with all directed colimits is called \emph{finitely presented} if  $\Hom_{\cl{A}}(M,-)\colon \cl{A}\to \Ab$ preserves directed colimits. 
One calls $\cl{A}$ \emph{locally finitely presented} if the subcategory of finitely presented objects is skeletally small, and if every object is a directed colimit of finitely presented objects. 
The \emph{definable subgroup} of an object  $M$ of \emph{sort $F$, given by} a morphism $\varepsilon\colon F\to G$  with $F$ and $G$ finitely presented, is  
\[
M\varepsilon=\im(\Hom_{\cl{A}}(\varepsilon,M))=\{ \zeta\varepsilon \in \Hom_{\cl{A}}(F,M) \mid \zeta  \in \Hom_{\cl{A}}(G,M)\}.
\]
\end{defn}

\begin{lem}
\label{lem-prereq-for-sigma-pure-inj-for-local-fin-pres}
Let $F$, $G$ and $M=(\tennsor{x}{M}{},\tennsor{y}{\alpha}{x})$ be representations of a species $\fk{S}=(\tennsor{}{R}{x},\tennsor{y}{A}{x})$ with $F$ and $G$ finitely presented. 
For each $\varepsilon\in\Hom_{\Rep{\fk{S}}}(F,G)$ there exist $\End_{\Rep{\fk{S}}}(M)$-linear maps $\Gamma$ and $\Delta$ such that $M\varepsilon= \im(\tennsor{}{\Gamma\vert}{\ker(\Delta)})$.
Moreover, $\Gamma$ and $\Delta$ are of the form
\begin{align*}
    (\tennsor{y(j)\,}{{\underline{\alpha}}}{\,x(i)}(\tennsor{j\,}{\underline{a}}{\,i}\otimes - ))\colon \tennsor{x(1)}{M}{}\oplus \cdots \oplus \tennsor{x(\ff{m})}{M}{}\to \tennsor{y(1)}{M}{}\oplus \cdots \oplus \tennsor{y(\ff{n})}{M}{},
\end{align*}
where $x(i),y(j)\in\ff{C}$ and $\tennsor{j\,}{\underline{a}}{\,i}\in \tennsor{y(j)\,}{\underline{A}}{\,x(i)}$ for each $i$ and $j$. 
\end{lem}

\begin{proof}
By \cite[Theorem,~p.~1645]{CrawleyBoeveyLocallyFP} $\lMod{\tennsor{}{\sr{T}}{\oplus}(\fk{S})}$ is locally finitely presented, and finitely presented objects   are cokernels of morphisms between representables. 
By Lemma~\ref{theorem-equivalence-of-reps-of-species-and-tensor-cat} this means any finitely presented representation  is a cokernel of a  morphism of the form
        \[
        \begin{array}{cc}
        \bigoplus_{i=1}^{\ff{d}}\tennsor{?\,}{\underline{A}}{\,x(i)}\to \bigoplus_{j=1}^{\ff{e}}\tennsor{?\,}{\underline{A}}{\,y(j)},
        &
        (x(1),\dots,x(\ff{m}),y(1),\dots,y(\ff{n})\in\ff{C}).
        \end{array}
        \]
Applying this separately to $F$ and then to $G$, we obtain the exact rows in the diagram  
\[
\begin{tikzcd}
\bigoplus_{i=1}^{\ff{m}}\tennsor{?\,}{\underline{A}}{\,s(i)}\arrow[r]\arrow[d]
&
\bigoplus_{j=1}^{\ff{n}}\tennsor{?\,}{\underline{A}}{\,t(j)}\arrow[d, "{\gamma}"]\arrow[r]
&
F\arrow[d, "{\varepsilon}"]\arrow[r] & 0\arrow[d]
\\
\bigoplus_{k=1}^{\ff{p}}\tennsor{?\,}{\underline{A}}{\,u(k)}\arrow[r, "{\delta}"']
&
\bigoplus_{\ell=1}^{\ff{q}}\tennsor{?\,}{\underline{A}}{\,v(\ell)}\arrow[r]
&
G\arrow[r] & 0
\end{tikzcd}
\]
 in $\Rep{\fk{S}}$. 
The first statement in Lemma~\ref{lem-finitely-presented-trivial} says each $\tennsor{?\,}{\underline{A}}{\,\ell}$ is projective, giving the vertical maps such that this diagram commutes. 
Furthermore,  applying $\Hom_{\Rep{\fk{S}}}(-,M)$ gives a commutative diagram in $\lMod{\End_{\Rep{\fk{S}}}(M)}$ with exact rows of the form
    \[
\begin{tikzcd}
\tennsor{s(1)}{M}{}\oplus \cdots \oplus \tennsor{s(\ff{m})}{M}{}
&
\tennsor{t(1)}{M}{}\oplus \cdots \oplus \tennsor{t(\ff{n})}{M}{}\arrow[l]
&
\ker(\eta)\arrow[l]
&
0\arrow[l]
\\
\Hom_{}(\bigoplus_{i=1}^{\ff{m}}\tennsor{?\,}{\underline{A}}{\,s(i)},M)\arrow[u, "{\cong\,\,}"]
&
\Hom_{}(\bigoplus_{j=1}^{\ff{n}}\tennsor{?\,}{\underline{A}}{\,t(j)},M)\arrow[l]\arrow[u, "{\cong\,\,}"]
&
\Hom_{}(F,M)\arrow[l]\arrow[u, "{\cong\,\,}"]
&
0\arrow[l]\arrow[u]
\\
\Hom_{}(\bigoplus_{k=1}^{\ff{p}}\tennsor{?\,}{\underline{A}}{\,u(k)},M)\arrow[u]
&
\Hom_{}(\bigoplus_{\ell=1}^{\ff{q}}\tennsor{?\,}{\underline{A}}{\,v(\ell)},M)\arrow[l, "\,\Delta"']\arrow[u, "\Gamma\,"]
&
\Hom_{}(G,M)\arrow[l]\arrow[u]
&
0\arrow[l]\arrow[u]
\\
\tennsor{u(1)}{M}{}\oplus \cdots \oplus \tennsor{u(\ff{p})}{M}{}\arrow[u, "{\cong\,\,}"]
&
\tennsor{v(1)}{M}{}\oplus \cdots \oplus \tennsor{v(\ff{q})}{M}{}\arrow[l]\arrow[u, "{\cong\,\,}"]
&
\ker(\mu)\arrow[l]\arrow[u, "{\cong\,\,}"]
&
0\arrow[l]\arrow[u]
\end{tikzcd}
\]
%
%
%
By Lemma~\ref{lem-finitely-presented-trivial} any  morphism $\tennsor{?\,}{\underline{A}}{\,y}\to  \tennsor{?\,}{\underline{A}}{\,x}$ has the form $\underline{a}\otimes -\colon \underline{b}\mapsto \underline{a}\otimes \underline{b}$ for unique $\underline{a}\in \tennsor{y\,}{\underline{A}}{\,x}$, and so   the components of the maps $\Gamma$ and $\Delta$ are described by maps of the form 
\[
\begin{array}{cc}
\tennsor{t(j)\,}{{\underline{\alpha}}}{\,v(\ell)}(\tennsor{j\,}{\underline{c}}{\,\ell}\otimes - )\colon \tennsor{v(\ell)}{M}{}\to \tennsor{t(j)}{M}{}, 
&
\tennsor{u(k)\,}{{\underline{\alpha}}}{\,v(\ell)}(\tennsor{k\,}{\underline{d}}{\,\ell}\otimes - )\colon \tennsor{v(\ell)}{M}{}\to \tennsor{u(k)}{M}{}.  
\end{array}
\]
where $\gamma(\tennsor{}{1}{t(j)})=\tennsor{}{\sum}{\ell}\tennsor{j\,}{\underline{c}}{\,\ell}$ and $\delta(\tennsor{}{1}{u(k)})=\tennsor{}{\sum}{\ell}\tennsor{k\,}{\underline{d}}{\,\ell}$. 
\end{proof}

\begin{lem}
\label{lem-finiteness-cond-for-existence-strong-decompo}
    Let $M=(\tennsor{x}{M}{},\tennsor{y}{\alpha}{x})$ be a representation of a species  $\fk{S}$. 
    If each $\tennsor{x}{M}{}$ is artinian over $\End_{\Rep{\fk{S}}}(M)$  then $M$ is a direct sum of strongly indecomposable representations. 
\end{lem}

\begin{proof}
Any definable subgroup of $M$ is an $\End_{\Rep{\fk{S}}}(M)$-submodule, and by Lemma~\ref{lem-prereq-for-sigma-pure-inj-for-local-fin-pres}, is an  $\End_{\Rep{\fk{S}}}(M)$-submodule of a finite direct sum of $\End_{\Rep{\fk{S}}}(M)$-modules of the form $\tennsor{\ell}{M}{}$. 
Each $\tennsor{\ell}{M}{}$ is artinian over $\End_{\Rep{\fk{S}}}(M)$, so any descending chain of definable subgroups of $M$  stabilises. 
Now apply the equivalence of (3) and (4) in \cite[Theorem~2,~p.~1661]{CrawleyBoeveyLocallyFP}. 
\end{proof}

\begin{lem}
    \label{lem-finiteness-cond-for-existence-strong-decompo-application}
    Let $\fk{S}=(\tennsor{}{R}{x},\tennsor{y}{A}{x})$ be a species, $S$ a common central subring of each $\tennsor{}{R}{x}$  that acts centrally on each $\tennsor{y}{A}{x}$, and  $M=(\tennsor{x}{M}{},\tennsor{y}{\alpha}{x})$  a representation  of $\fk{S}$. 
    If each $\tennsor{x}{M}{}$ is artinian over $S$ then $M$ is a direct sum of strongly indecomposable representations. 
\end{lem}

\begin{proof}
By assumption, for each  $x,y\in\ff{C}$ and $s\in S$, we have $rs=sr$ for each  $r\in\tennsor{}{R}{x}$, and $sa=as$ for each $a\in\tennsor{y}{A}{x}$. 
In this notation,  if $m\in \tennsor{x}{M}{}$ it follows that $srm=rsm$ and 
\[
\tennsor{y}{\alpha}{x}(a\otimes sm)=
\tennsor{y}{\alpha}{x}(as\otimes m)=
\tennsor{y}{\alpha}{x}(sa\otimes m)=
s\tennsor{y}{\alpha}{x}(a\otimes m).
\]
Thus the map $\rho\colon S\to \End_{\Rep{\fk{S}}}(M)$, sending $s$ to  $[m\mapsto sm]$, defines a ring map. 
Furthermore, the action of $\rho(s)$ on $m$ is $sm$, 
so Lemma~\ref{lem-finiteness-cond-for-existence-strong-decompo} applies. 
\end{proof}

We now establish a different set of sufficient conditions that ensure a decomposition into strongly indecomposable representations. 
We closely follow the proof of \cite[Theorem~1.1]{botnan-crawley-boevey-persistence}, a result of Botnan and Crawley-Boevey, that says that pointwise finite-dimensional persistence modules have such decompositions. 
We repeat the argument for completeness.

Recall that a ring $R$ is \emph{strongly}-$\pi$-\emph{regular} if for each $a\in R$ the descending chain  $Ra \supseteq Ra^{2}\supseteq Ra^{3}\dots $ stabilises. 
It is equivalent, by a result of Dischinger \cite{Disching}, that each descending chain  $aR \supseteq a^{2}R\supseteq a^{3}R\dots $ stabilises. 
By a result of Bass \cite[Theorem~P]{bassfini} a ring is \emph{left}-\emph{perfect} if and only if every descending chain of cyclic right ideals stabilises. Semiprimary rings, and left (right) artinian rings, are left (right) perfect. 

\begin{lem}
\label{lem-strongly-indecomposable}      
Let $M=(\tennsor{x}{M}{},\tennsor{y}{\alpha}{x})$ be an indecomposable $\fk{S}$-representation such that each $\End_{\Rep{\fk{S}}}(\tennsor{x}{M}{})$ is strongly-$\pi$-regular for each $x\in\ff{C}$. 
Then $\End_{\Rep{\fk{S}}}(M)$ is local. 
\end{lem}

\begin{proof}
A result of Armendariz--Fisher--Snider  \cite[Proposition~2.3]{inj-surj-morphs-of-f-g} says that the endomorphism ring of a left module   is strongly-$\pi$-regular if and only if said module satisfies the conclusion of Fitting's Lemma. 
By assumption, this applies to each $\tennsor{x}{M}{}$. 

Let $f\in \End_{\Rep{\fk{S}}}(M)$ be a non-unit. 
We show $\tennsor{}{\iden}{M}-f$ is a unit. 
As above, for each $x$ there exists $n(x)>0$ such that $\tennsor{x}{M}{}=\im(\tennsor{x}{f}{}^{i})\oplus \ker(\tennsor{x}{f}{}^{i}) $  for each $i\geq n(x)$. 
    Let $\tennsor{x}{I}{}=\im(\tennsor{x}{f}{}^{n(x)})$ and $\tennsor{x}{K}{}=\ker(\tennsor{x}{f}{}^{n(x)})$, and write $\tennsor{x}{\iota}{I}\colon \tennsor{x}{I}{}\to \tennsor{x}{M}{}$, $\tennsor{x}{\iota}{K}\colon \tennsor{x}{K}{}\to \tennsor{x}{M}{}$, $\tennsor{x}{\rho}{I}\colon \tennsor{x}{M}{}\to \tennsor{x}{I}{}$ and $\tennsor{x}{\rho}{K}\colon \tennsor{x}{M}{}\to \tennsor{x}{K}{}$ 
    for the canonical embeddings 
     and projections. 
    Let 
    \[
    \begin{array}{cc}
    I=(\tennsor{x}{I}{},\tennsor{y}{\rho}{I}\circ \tennsor{y}{\alpha}{x}\circ (\tennsor{}{\iden}{\tennsor{y}{A}{x}}\otimes \tennsor{x}{\iota}{I})),
    &
    K=(\tennsor{x}{K}{},\tennsor{y}{\rho}{K}\circ \tennsor{y}{\alpha}{x}\circ (\tennsor{}{\iden}{\tennsor{y}{A}{x}}\otimes \tennsor{x}{\iota}{K})).
    \end{array}
    \] 
    If $x,y\in\ff{C}$ and $i\geq \max\{n(x),n(y)\}$ then 
    \[
    \tennsor{y}{\alpha}{x}\circ (\tennsor{}{\iden}{\tennsor{y}{A}{x}}\otimes  \tennsor{x}{f}{}^{i})
    = 
    \tennsor{y}{\alpha}{x}\circ (\tennsor{}{\iden}{\tennsor{y}{A}{x}}\otimes  \tennsor{x}{f}{})^{i}
    =
    \tennsor{y}{f}{} \circ \tennsor{y}{\alpha}{x}\circ (\tennsor{}{\iden}{\tennsor{y}{A}{x}}\otimes  \tennsor{x}{f}{})^{i-1}
    =
    \cdots
    =
    \tennsor{y}{f}{}^{i} \circ \tennsor{y}{\alpha}{x}. 
    \]
   It follows that, for each $a\in \tennsor{y}{A}{x}$ and $m\in\tennsor{x}{M}{}$, we have 
   \[
   \tennsor{y}{\iota}{I}(\tennsor{y}{\rho}{I}(\tennsor{y}{\alpha}{x}(a\otimes  \tennsor{x}{f}{}^{i}(m))))=\tennsor{y}{\iota}{I}(\tennsor{y}{\rho}{I}(\tennsor{y}{f}{}^{i} ( \tennsor{y}{\alpha}{x}(a\otimes  m))))=\tennsor{y}{f}{}^{i} ( \tennsor{y}{\alpha}{x}(a\otimes  m))=\tennsor{y}{\alpha}{x}(a\otimes  \tennsor{x}{f}{}^{i}(m)).
   \]
   By Lemma~\ref{lemma-co-limits-in-rep-of-species-cat} it follows that $M=I\oplus K$. 
   We claim $I=0$. 
   Since $f$ is not a unit, there is some $x\in\ff{C}$ with $\tennsor{x}{f}{}$ not a unit. 
   Since $\tennsor{x}{M}{}=\tennsor{x}{I}{}\oplus \tennsor{x}{K}{}$, we have $\tennsor{x}{K}{}\neq 0$, for otherwise we also have $\tennsor{x}{M}{}=\tennsor{x}{I}{}$. 
   This shows $K\neq 0$ and so the claim holds since $M$ is indecomposable. 

To see that $\tennsor{}{\iden}{M}-f$ is surjective, given $x\in \ff{C}$ and $m\in \tennsor{x}{M}{}=\tennsor{x}{K}{}$ we have  
\[
m=m-\tennsor{x}{f}{}^{n(x)}(m)=(\tennsor{}{\iden}{\tennsor{x}{M}{}}-\tennsor{x}{f}{})(\tennsor{}{\iden}{\tennsor{x}{M}{}}+\tennsor{x}{f}{}+\dots + \tennsor{x}{f}{}^{n(x)-1})(m).
\]
For injectivity, given   $m\in \ker(\tennsor{}{\iden}{\tennsor{x}{M}{}}-\tennsor{x}{f}{})$ we have $m=\tennsor{x}{f}{}(m)=\dots=\tennsor{x}{f}{}^{n(x)}(m)=0$. 
\end{proof}

To be clear, when we refer to a subrepresentation $N=(\tennsor{x}{N}{},\tennsor{y}{\beta}{x})$ of a representation  $M=(\tennsor{x}{M}{},\tennsor{y}{\alpha}{x})$, we mean that each $\tennsor{x}{N}{}$ is an $\tennsor{}{R}{x}$-submodule (and in particular, a subset) of $\tennsor{x}{M}{}$, and that each $\tennsor{y}{\beta}{x}\colon \tennsor{y}{A}{x}\otimes\tennsor{x}{N}{}\to \tennsor{y}{N}{}$ is the restriction of $\tennsor{y}{\alpha}{x}\colon \tennsor{y}{A}{x}\otimes\tennsor{x}{M}{}\to \tennsor{y}{M}{}$. 
In particular, the collection $\Sigma$ in the lemma below is a subset of the power set. 
\begin{lem}
\label{lem-technical-botnan--crawley-boevey-argument}
    Let $M=(\tennsor{x}{M}{},\tennsor{y}{\alpha}{x})$ be a representation of a species $\fk{S}=(\tennsor{}{R}{x},\tennsor{y}{A}{x})$,  $\Sigma$ be the set  of non-zero subrepresentations of $M$,  $\Delta=\{Z\subseteq \Sigma\mid M=\tennsor{}{\bigoplus }{L\in Z}L\}$, and 
    \[
    \preceq \,=\, \left\{(Z,Y)\in \Delta\times \Delta \left\vert\begin{array}{cc}
       \text{for each }L\in Z\text{ there exists a subset }
       \\
       Y(L)\subseteq Y\text{ such that }L=\tennsor{}{{\bigoplus}}{K\in Y(L)}K
    \end{array}\right.\right\}. 
    \]
    \begin{enumerate}
        \item If $Z \preceq Y $ then the following statements hold. 
        \begin{enumerate}
            \item If $L,N\in Z$ and $L\neq N$ then $Y(L)\cap Y(N)=\emptyset$. 
            \item If  $K\in Y$ then there exists $L\in Z$ unique such that $K\in Y(L)$. 
            \item There is a well-defined surjective function $\mu_{ZY}\colon Y\to Z$. 
        \end{enumerate}
        \item The relation $\preceq $ defines a partial order on $\Delta$, and $\Delta$ has unique minimum $\{M\}$. 
        \item If $X\in \Delta$ is maximal  with respect to $\preceq $, then each $L\in X$ is indecomposable. 
        \item If $\Theta$ is a chain in $\Delta$, and if $\Lambda=\{\underline{L}=(\tennsor{}{L}{Z})\in\tennsor{}{\prod}{Z\in\Theta}Z\mid \mu_{ZY}(\tennsor{}{L}{Y})=\tennsor{}{L}{Z}\} $, then the representations $M[\underline{L}]=\tennsor{}{\bigcap}{Z\in \Theta}\tennsor{}{L}{Z}$ are independent, meaning  $\tennsor{}{\sum}{\underline{L}\in\Lambda}M[\underline{L}]$ is direct. 
    \end{enumerate}
\end{lem}

\begin{proof}
(1a) 
If $K\in Y(L)$ then $K$ is a summand, so in particular a submodule, of $L$. 
Hence if $K\in Y(L)\cap Y(N)$ then $K\in \Sigma$ and $L,N\in Z$, giving the contradiction $0\neq K\subseteq L\cap N=0$. 

(1b) 
By (1a) the union $X=\tennsor{}{\bigcup}{L\in Z}Y(L)$ is disjoint. 
Assume $K\notin X$. 
Since $X\subseteq Y\ni K$ the sum $K+\tennsor{}{\sum}{J\in X}J$ is direct. 
Since each $L\in Z$ satisfies $L=\tennsor{}{{\bigoplus}}{H\in Y(L)}H$ we have $M=\tennsor{}{\bigoplus}{L\in Z}L=\tennsor{}{{\bigoplus}}{J\in X}J$. 
Since $0\neq K$ is a   subrepresentation of $M$, this is a contradiction.  
So $K\in X$, and so $K\in Y(L)$ for some $L\in Z$. By (1a), $L$ is unique  such that $K\in Y(L)$. 

(1c) 
If $K\in Y$ choose $\mu_{ZY}(K)=L\in Z$, unique with $K\in Y(L)$ by (1b). 
So $\mu_{ZY}$ is well-defined. 
If $L\in Z$ then $Y(L)\neq \emptyset$, as $0\neq L=\bigoplus_{K\in Y(L)}K$. 
So $\mu_{ZY}$ is surjective. 

(2) 
Any $L\in Z$ defines $Z(L)=\{L\}$ such that $L=\tennsor{}{\bigoplus}{K\in Z(L)}K$. 
So $\preceq$ is reflexive. 

Let $ Z\preceq Y\preceq X $ and $L\in Z$. 
For each $K\in Y(L)$ we have $K=\tennsor{}{\bigoplus}{J\in X(K)}J$ for some $X(K)\subseteq X$. 
By (1a) the union $X[L]=\tennsor{}{\bigcup}{K\in Y(L)} X(K)$ is disjoint. 
Thus, 
\[
L=\tennsor{}{{\bigoplus}}{K\in Y(L)}K=\tennsor{}{{\bigoplus}}{K\in Y(L)}\tennsor{}{{\bigoplus}}{J\in X(K)}J=\tennsor{}{{\bigoplus}}{J\in X[L]}J,
\]
and so $\preceq$ is transitive. 
To see $\preceq$ is antisymmetric, we suppose  $Z\preceq Y\preceq Z$. 
By symmetry it suffices to show $Z\subseteq Y$. 
Let $L\in Z$. 
We claim $L\in Y$. 
Take $Z=X$ in the above, so we have a disjoint union $Z[L]=\tennsor{}{\bigcup}{K\in Y(L)} Z(K)$ such that $L=\tennsor{}{\bigoplus}{J\in Z[L]}J$. 
For each $J\in Z[L]$ we have $0\neq J\in Z$ and $J\subseteq L$, and so $J=L$ as this  sum is direct. 
This means $Z[L]=\{L\}$.  
So there is a unique $K\in Y(L)$ with $Z(K)\neq \emptyset $, in which case $Z(K)=\{L\}$. 
This shows $Y(L)=\{K\}$ by (1a), and so $L=K\in Y$.  

(3)
Suppose otherwise, so $L=Q\oplus R$ for $Q\neq 0\neq R$ and let $Y=(X\setminus\{L\}) \cup \{Q,R\}$. 
Now if $N\in X$ either $N=L$ or $N\neq L$. If $N=L$ then let $Y(N)=\{Q,R\}$, and if $N\neq L$ then $N\in Y$ and let $Y(N)=\{N\}$. 
In either case $Y(N)$ is a subset  of $Y$  such that $N=\bigoplus_{K\in Y(N)}K$. 
This contradicts maximality. 

(4) 
Sums and intersections in cocomplete abelian categories  are defined by coproducts, kernels and cokernels. 
By Lemma~\ref{lemma-co-limits-in-rep-of-species-cat}, it suffices to let $x\in\ff{C}$ and show   $\tennsor{}{\sum}{\underline{L}\in\Lambda}\tennsor{x}{M[\underline{L}]}{}$ is direct. 
Let $\tennsor{1}{\underline{L}}{},\dots,\tennsor{\ff{n}}{\underline{L}}{}\in \Lambda$ be distinct. 
Suppose $\sum_{i=1}^{\ff{n}}\tennsor{}{m}{i}=0$ where $\tennsor{}{m}{i}\in \tennsor{x}{{M[\tennsor{i}{\underline{L}}{}]}}{}$. 

Given $i<j$ we have $\tennsor{i}{\underline{L}}{}\neq \tennsor{j}{\underline{L}}{}$, and so  there must exist $Y(i,j)\in \Theta$ such that $\tennsor{i}{L}{Y(i,j)}\neq \tennsor{j}{L}{Y(i,j)}$. 
Since $\Theta$ is chain the finite set $\{Y(i,j)\mid 1 \leq i<j \leq \ff{n}\}$ has a maximal element $Y$. 
It follows that $\tennsor{i}{L}{Y}\neq  \tennsor{j}{L}{Y}$ when $i<j$, for otherwise $\tennsor{i}{L}{Z}=\tennsor{}{\mu}{ZY}(\tennsor{i}{L}{Y})=\tennsor{}{\mu}{ZY}(\tennsor{j}{L}{Y})=\tennsor{j}{L}{Z}$ whenever $Z \preceq Y$, which is false for $Z=Y(i,j)$. 
Hence $\tennsor{i}{L}{Y}\neq \tennsor{j}{L}{Y}$ when $i\neq j$. 
So $\tennsor{i}{L}{Y}\in Y$ for each $i$, and $\tennsor{}{m}{i}\in \tennsor{}{\bigcap}{Y\in \Theta}\tennsor{i}{L}{Y}\subseteq \tennsor{i}{L}{Y}$. 
Since the sum $M=\tennsor{}{\sum}{L\in Y}L$ is direct, so is the sum $\sum_{i=1}^{\ff{n}}\tennsor{i}{L}{Y}$, and this ensures $\tennsor{}{m}{i}=0$ for each $i$. 
\end{proof}

\begin{lem}
\label{lem-finite-ness-conds-for-botnan--crawley-boevey-argument}
    Let $M=(\tennsor{x}{M}{},\tennsor{y}{\alpha}{x})$ be a representation of a species $\fk{S}=(\tennsor{}{R}{x},\tennsor{y}{A}{x})$. 
    Assume that, for each $x\in\ff{C}$, there is a bound $\ff{b}(x)$ on the length $\ff{m}$ of a sequence $(\tennsor{}{N}{1},\dots ,\tennsor{}{N}{\ff{m}})$ of non-zero $\tennsor{}{R}{x}$-submodules  $\tennsor{}{N}{i}$ of  $\tennsor{x}{M}{}$ such that the sum $\sum_{i=1}^{\ff{m}}\tennsor{}{N}{i}$ is direct. 
    Then $M$ is isomorphic to a direct sum of indecomposable representations of $\fk{S}$. 
\end{lem}

\begin{proof}
We assume the notation from the statement and proof of Lemma~\ref{lem-technical-botnan--crawley-boevey-argument}. 
By part (3) it suffices to find a maximal element of $\Delta$ with respect to the relation $\preceq$, a partial order by part (2). 
By Zorn's lemma it suffices to show any chain $\Theta$ in $\Delta$ has an upper bound. 
Recall $\Lambda$, and the direct sum $\tennsor{}{\bigoplus}{\underline{L}\in\Lambda}M[\underline{L}]$, from part (4). 
We assert that $M=\tennsor{}{\bigoplus}{\underline{L}\in\Lambda}M[\underline{L}]$. 
Given the assertion holds, the set $W=\{M[\underline{L}]\mid \underline{L}\in\Lambda,M[\underline{L}]\neq 0\}$ defines the required upper bound of $\Theta$. 
So it suffices to prove the assertion. 

Let $x\in\ff{C}$ and $0\neq m\in \tennsor{x}{M}{}$. 
By construction, and by Lemma~\ref{lemma-co-limits-in-rep-of-species-cat}, for each $Z\in\Delta$ we have $\tennsor{x}{M}{}=\tennsor{}{\bigoplus}{L\in Z}\tennsor{x}{L}{}$. 
For a contradiction, assume that for each $\ff{n}\geq1$ there exists some $Z(\ff{n})\in \Delta$ and pairwise distinct $\tennsor{}{L}{\ff{n}}(1),\dots,\tennsor{}{L}{\ff{n}}(\ff{n})\in Z(\ff{n})$ such that $m=\sum_{i=1}^{\ff{n}}\tennsor{\ff{n}}{\ell}{i}$ for some $0\neq \tennsor{\ff{n}}{\ell}{i}\in \tennsor{x}{L}{\ff{n}}(i)$ for each $i$. 
So $\tennsor{x}{L}{\ff{n}}(i)\neq 0$ for each $\ff{n}$ and $i$. 
For each $\ff{n}$ let $\tennsor{}{K}{\ff{n}}=\bigoplus_{i=1}^{\ff{n}}\tennsor{}{L}{\ff{n}}(i)$. 
Again using Lemma~\ref{lemma-co-limits-in-rep-of-species-cat}, note that $\tennsor{x}{K}{\ff{n}}=\bigoplus_{i=1}^{\ff{n}}\tennsor{x}{L}{\ff{n}}(i)$, a direct sum of $\ff{n}$ non-zero  modules, and a direct summand of $\tennsor{}{\bigoplus}{L\in Z(\ff{n})}\tennsor{x}{L}{}=\tennsor{x}{M}{}$. 
This contradicts the hypothesis when $\ff{n}=\ff{b}(x)+1$ by considering $\tennsor{}{N}{i}=\tennsor{x}{L}{\ff{n}}(i)$.

So, choose $Z\in \Theta$ and  $\ff{n}\geq 1$ appropriately: such that $m=\sum_{i=1}^{\ff{n}}\tennsor{ }{\ell}{i}$ where $0\neq \tennsor{ }{\ell}{i}\in \tennsor{x}{L}{ }(i)$ for some $\tennsor{}{L}{}(1),\dots,\tennsor{}{L}{}(\ff{n})\in Z$, pairwise distinct; and such that, if $Y\in \Theta$ and  $m=\sum_{i=1}^{\ff{m}}\tennsor{ }{k}{i}$ where $0\neq \tennsor{ }{k}{i}\in \tennsor{x}{K}{ }(i)$ for $\tennsor{}{K}{}(1),\dots,\tennsor{}{K}{}(\ff{m})\in Y$ pairwise distinct, then $\ff{m}\leq \ff{n}$. 

Consider the chain $\Omega=\{Y\in \Theta \mid Z \preceq Y\}$. 
For each $i$ we construct an element $\tennsor{}{\underline{L}}{}(i)=(\tennsor{}{L}{Y}(i))\in \tennsor{}{\prod}{Y\in\Omega}Y$. 
Suppose $Y\in\Theta$ and $Z\preceq Y$. 
It follows that there is a subset $Y(i)\subseteq Y$ such that $\tennsor{x}{L}{ }(i)=\tennsor{}{\bigoplus}{K\in X(i)}\tennsor{x}{K}{}$, meaning we can write $\tennsor{ }{\ell}{i}=\tennsor{}{\sum}{K\in Y(i)}\tennsor{K}{\ell}{i}$ where $\tennsor{K}{\ell}{i}\in \tennsor{x}{K}{}$. 
The maximality of $\ff{n}$ above forces that there exists $\tennsor{}{K}{Y}(i)\in Y(i)$ unique such that $\tennsor{K_{Y}(i)}{\ell}{i}\neq 0$, meaning $\tennsor{ }{\ell}{i}=\tennsor{K(i)}{\ell}{i}$. 
In particular $\tennsor{ }{\ell}{i}\in \tennsor{x}{K}{Y}(i)$ and $\tennsor{x}{K}{X}(i)\in X$.
Now let $\tennsor{}{L}{Y}(i)=\tennsor{}{K}{Y}(i)$. 
By construction we have $\tennsor{}{\mu}{YX}(\tennsor{}{L}{X}(i))=\tennsor{}{L}{Y}(i)$ for each $X,Y\in\Omega$ with $Y \preceq X$.  %
Also by construction, $\tennsor{}{\ell}{i}\in \tennsor{}{L}{X}(i)\cap\tennsor{}{\bigcap}{Z\neq Y\in\Omega}\tennsor{}{K}{Y}(i) = M[\tennsor{}{\underline{L}}{}(i)]$. 
\end{proof}

\begin{lem}\label{lem-finite-length-means-nice-decomposition}
    Let $M=(\tennsor{x}{M}{},\tennsor{y}{\alpha}{x})$ be a representation of a species $\fk{S}=(\tennsor{}{R}{x},\tennsor{y}{A}{x})$. 
    Assume that, for each $x\in\ff{C}$, the $\tennsor{}{R}{x}$-module $\tennsor{x}{M}{}$ has finite length. 
    Then $M$ is a direct sum of strongly indecomposable representations of $\fk{S}$. 
\end{lem}

\begin{proof}
Finite length modules satisfy the Krull--Remak--Schmidt property, and so they satisfy the hypothesis of Lemma~\ref{lem-finite-ness-conds-for-botnan--crawley-boevey-argument}. 
By \cite[Proposition~2.3]{inj-surj-morphs-of-f-g} and Fitting's lemma,   Lemma~\ref{lem-strongly-indecomposable} also applies. 
Now combine the conclusions of these results. 
\end{proof}

\section{Ideals, relations, and the  double category of bimodules}
\label{section-relations}
\label{ideals-relations-and-BIMOD-oh-my!}

\begin{defn}\label{defn-ideal}
An \emph{ideal} $\sr{I}=(\tennsor{h}{I}{t})$ of a species $\fk{S}=(\tennsor{}{R}{x},\tennsor{y}{A}{x})$ will mean an ideal of $\sr{T}(\fk{S})$, where $\tennsor{h}{I}{t}$ is the corresponding subgroup of $\Hom_{\sr{T}(\fk{S})}(t,h)=\tennsor{h\,}{\underline{A
}}{\,t}$ for each $t,h\in\ff{C}$.
\end{defn}
\begin{rem}
\label{rem-ideals}
    That $\sr{I}$ is an ideal of $\fk{S}$ says that each $\tennsor{h}{I}{t}$ is an $\tennsor{}{R}{h}$-$\tennsor{}{R}{t}$-subbimodule of $\tennsor{h\,}{\underline{A}}{\,t}$ such that $\tennsor{\ell}{A}{h}\otimes \tennsor{h}{I}{t}\subseteq \tennsor{\ell}{I}{t}$ and   $\tennsor{h}{I}{t}\otimes \tennsor{t}{A}{r}\subseteq \tennsor{h}{I}{r}$ for each $\ell,r\in\ff{C}$. 
To see this, recall that composition in $\sr{T}(\fk{S})$ is tensor concatenation, and that $\tennsor{}{R}{x}=\tennsor{}{A}{(x)}$ and $\tennsor{y}{A}{x}=\tennsor{}{A}{(y,x)}$ define  summands of $\tennsor{y\,}{\underline{A}}{\,x}$ for each $x,y\in\ff{C}$. 
\end{rem}

\begin{lem}
    \label{lem-ideals}
Let $\sr{I}$ be an ideal of a species $\fk{S}=(\tennsor{}{R}{x},\tennsor{y}{A}{x})$. 
An $\fk{S}$-representation $M=(\tennsor{x}{M}{},\tennsor{y}{\alpha}{x})$ lies in the essential image of the restriction of the equivalence from Lemma~\ref{theorem-equivalence-of-reps-of-species-and-tensor-cat} to $\lMod{\sr{T}(\fk{S})/\sr{I}}$, if and only if, $\tennsor{h}{I}{t}\otimes \tennsor{t}{M}{}\subseteq \ker(\tennsor{h\,}{\underline{\alpha}}{\,t})$ for each $t,h\in\ff{C}$. 
\end{lem}

\begin{proof}
The category $\sr{T}(\fk{S})/\sr{I}$ has the same objects as $\sr{T}(\fk{S})$, meaning elements of $\ff{C}$, where as the abelian group of morphisms $t\to h$ in $\sr{T}(\fk{S})/\sr{I}$ is defined by the quotient $\tennsor{h\,}{\underline{A}}{\,t}/\tennsor{h}{I}{t}$. 
The quotient functor $\rho\colon \sr{T}(\fk{S})\to \sr{T}(\fk{S})/\sr{I}$ fixes objects and is defined on morphisms by the projections $\tennsor{h}{\rho}{t}\colon \tennsor{h\,}{\underline{A}}{\,t}\to \tennsor{h\,}{\underline{A}}{\,t}/\tennsor{h}{I}{t}$. 
The functor  $\lMod{\sr{T}(\fk{S})/\sr{I}}\to \Rep{\fk{S}}$ is the composition of the equivalence $\lMod{\sr{T}(\fk{S})}\to \Rep{\fk{S}}$ from Lemma~\ref{theorem-equivalence-of-reps-of-species-and-tensor-cat} with the restriction $\lMod{\sr{T}(\fk{S})/\sr{I}}\to \lMod{\sr{T}(\fk{S})}$ along $\rho$. 

Using density of the equivalence, there is a functor $\sr{F}\colon \sr{T}(\fk{S})\to \Ab$ and isomorphisms  $\tennsor{x}{\theta}{}\colon \tennsor{x}{M}{}\to \sr{F}(x)$ ($x\in\ff{C}$) such that, defining $\tennsor{y}{\varepsilon}{x}\colon \tennsor{y}{A}{x}\otimes \sr{F}(x)\to \sr{F}(y)$ ($x,y\in\ff{C}$) by  $a\otimes m\mapsto\sr{F}(a)(m)$, we have $\tennsor{y}{\theta}{}\circ \tennsor{y}{\alpha}{x}=\tennsor{y}{\varepsilon}{x}\circ (\tennsor{}{\iden}{\tennsor{y}{A}{x}}\otimes \tennsor{x}{\theta}{})$. 

Following the proof of Lemma~\ref{lem-elementary-underline-maps}, for each $t,h\in\ff{C}$, $i\in\tennsor{h}{I}{t}$ and  $m\in\tennsor{t}{M}{}$, we have
\[
\tennsor{h\,}{\underline{\alpha}}{\,t}(i\otimes m)=\tennsor{y}{\theta}{}^{-1}(\tennsor{h\,}{\underline{\varepsilon}}{\,t}(i\otimes \tennsor{x}{\theta}{}(m)))=\tennsor{y}{\theta}{}^{-1}(\sr{F}(i)(\tennsor{x}{\theta}{}(m))). 
\]
This expression is $0$ for each $i$ and $m$ if and only if $\tennsor{y}{\theta}{}^{-1}\sr{F}(i)\tennsor{x}{\theta}{}=0$ for each $i$, if and only if $\sr{F}(i)=0$ for each $i$, if and only if $\tennsor{h}{I}{t}$ lies in the kernel of the map $\Hom_{\sr{T}(\fk{S})}(t,h)\to \Hom_{\Ab}(\sr{F}(t),\sr{F}(h))$. 
This holds for each $(t,h)$ if and  only if $\sr{F}$ factors through $\rho$. 
\end{proof}

Ideals, and representations of species annihilated by said ideals, have been studied before, in work of Assem \cite[\S 2]{Assem-B-and-C}  and Berg \cite[\S 2]{Berg2011}. 
\begin{nota}\label{nota-Rep(S,I)}
    Given an ideal $\sr{I}$ of $\fk{S}$ let $\Rep{\fk{S},\sr{I}}$ be the subcategory of representations  $M$ with  $\tennsor{h}{I}{t}\otimes\tennsor{t}{M}{}\subseteq \ker(\tennsor{h\,}{\underline{\alpha}}{\,t})$ for each $(h,t)\in\ff{C}^2$. 
\end{nota}
\begin{lem}
    \label{rem-decomposition-passes-to-relations}
    If $\sr{I}$ is an ideal of a species $\fk{S}$ then $\Rep{\fk{S},\sr{I}}$ is closed under set-indexed direct sums and summands. 
\end{lem}

\begin{proof}
     Suppose a representation $M=(\tennsor{x}{M}{},\tennsor{y}{\alpha}{x})$ of a species $\fk{S}=(\tennsor{}{R}{x},\tennsor{y}{A}{x})$ decomposes into a direct sum $\bigoplus_{j\in J} M(j)$, for some set $J$.
    Note that $\tennsor{y}{\underline{\alpha}}{x}=\bigoplus_{j\in J} \tennsor{y\,}{\underline{\alpha}}{\,x}(j)$. 

    On the one hand, suppose each $M(j)$ lies in $\Rep{\fk{S},\sr{I}}$. 
    Let $m\in \tennsor{t}{M}{}$. 
    By Lemma~\ref{lemma-co-limits-in-rep-of-species-cat} we have $m=\sum_{j\in J} m_j$ where $m_j\in \tennsor{t}{M}{}(j)$ for each $j$, and $m_j=0$ for all but finitely many $j$. 
    By Lemma~\ref{lem-ideals}, for any $z\in \tennsor{h}{I}{t}$  we have $\tennsor{h\,}{\underline{\alpha}}{\,t}(z\otimes m)=\sum_{j\in J}\tennsor{h\,}{\underline{\alpha}}{\,t}(j)(z\otimes m_j)=0$.  

    On the other hand, if $M$ lies in $\Rep{\fk{S},\sr{I}}$ then $\tennsor{h}{I}{t}\otimes \tennsor{t}{M}{}\subseteq \ker(\tennsor{h\,}{\underline{\alpha}}{\,t})$  for each $(h,t)\in\ff{C}^{2}$ giving, as above, that  $\tennsor{h}{I}{t}\otimes \tennsor{t}{M}{}(j) \subseteq \ker(\tennsor{h\,}{\underline{\alpha}}{\,t}(j))$, so $M(j)$ also lies in $\Rep{\fk{S},\sr{I}}$. 
\end{proof}

\begin{lem}
    \label{lem-when-functors-lift-to-functor-on-species-relations}
    Let  $\fk{S}=(\tennsor{}{R}{x},\tennsor{y}{A}{x})$  and $\fk{T}=(\tennsor{}{S}{x},\tennsor{y}{B}{x})$ be species,   $\sr{I}=(\tennsor{h}{I}{t})$ be an ideal of $\fk{S}$,  $\sr{J}=(\tennsor{h}{J}{t})$ be an ideal of $\fk{T}$, and  suppose there are natural transformations of the form
    \begin{equation}
\label{eqn-suff-nat-trans-for-ptwise-functors-to-lift}
    \begin{tikzcd}[column sep = large]
	{\lMod{\tennsor{}{R}{x}}} & {\lMod{\tennsor{}{S}{x}}} \\
	{\lMod{\tennsor{}{R}{y}}} & {\lMod{\tennsor{}{S}{y}}}
	\arrow["{{\tennsor{x}{\sr{G}}{}}}", from=1-1, to=1-2]
	\arrow["{{\tennsor{y}{A}{x}\otimes_{\tennsor{}{R}{x}}-}\,\,}"', from=1-1, to=2-1]
	\arrow["{{\tennsor{y}{\Omega}{x}}}"{description}, Rightarrow, from=1-2, to=2-1]
	\arrow["{\,\,{\tennsor{y}{B}{x}\otimes_{\tennsor{}{S}{x}}-}}", from=1-2, to=2-2]
	\arrow["{{\tennsor{y}{\sr{G}}{}}}"', from=2-1, to=2-2]
\end{tikzcd}
\quad\left(\forall(y,x)\in\ff{C}^{2}\right)
\end{equation}
as  in Lemma~\ref{lem-functors-between-reps-of-species}. 
If $ \tennsor{}{\left({\tennsor{h\,}{\underline{\Omega}}{\,t}} \right)}{L}
(\tennsor{h}{J}{t}\otimes \tennsor{t}{\sr{G}}{}
(L)
)
\subseteq 
\tennsor{h}{\sr{G}}{}
(\tennsor{h}{I}{t}\otimes L
) $ for each $(h,t)\in\ff{C}^{2}$ and each $\tennsor{}{R}{t}$-module $L$, then $\tennsor{}{\sr{G}}{\Omega}$ restricts to  $ \Rep{\fk{S},\sr{I}}\to \Rep{\fk{T},\sr{J}}$. 
\end{lem}

\begin{proof}
By assumption, and by Lemma~\ref{lem-functors-between-reps-of-species}, whenever $\tennsor{}{\sr{G}}{\Omega}(\tennsor{x}{M}{},\tennsor{y}{\alpha}{x})=(\tennsor{x}{N}{},\tennsor{y}{\omega}{x})$ we have
\[
    \tennsor{h\,}{\underline{\omega}}{\,t}\left(\tennsor{h}{J}{t}\otimes \tennsor{t}{\sr{G}}{}
\left(\tennsor{t}{M}{}\right)
\right)=
\tennsor{h}{\sr{G}}{}(\tennsor{h\,}{\underline{\alpha}}{\,t})\left(\tennsor{}{{(\tennsor{h\,}{\underline{\Omega}}{\,t})}}{{\tennsor{t}{M}{}}}\left(\tennsor{h}{J}{t}\otimes \tennsor{t}{\sr{G}}{}
\left(\tennsor{t}{M}{}\right)\right)
\right)
\subseteq 
\tennsor{h}{\sr{G}}{}\left(\tennsor{h\,}{\underline{\alpha}}{\,t}\left(\tennsor{h}{I}{t}\otimes \tennsor{t}{M}{}
\right)\right).
\]
Hence if $\tennsor{h}{I}{t}\otimes \tennsor{t}{M}{}\subseteq \ker\left(\tennsor{h\,}{\underline{\alpha}}{\,t}\right)$ then $\tennsor{h}{J}{t}\otimes \tennsor{t}{M}{}\subseteq \ker\left(\tennsor{h\,}{\underline{\omega}}{\,t}\right)$. 
Now apply   Lemma~\ref{lem-ideals}. 
\end{proof}

\begin{lem}
\label{lem-ring-homs-and-compatible-maps-between-bimodules-ideals}
    Let $\fk{S}=(\tennsor{}{R}{x},\tennsor{y}{A}{x})$ and $\fk{T}=(\tennsor{}{S}{x},\tennsor{y}{B}{x})$ be species,  $\sr{I}=(\tennsor{h}{I}{t})$ an ideal of $\fk{S}$,  $\sr{J}=(\tennsor{h}{J}{t})$ an ideal of $\fk{T}$, and suppose there are ring   and bilinear maps  as in Lemma~\ref{lem-ring-and-bilinear-maps-induce-functor-between-tensor-cat},
    \begin{equation}
        \begin{array}{cc}
    \tennsor{}{\rusEl}{x}\colon \tennsor{}{S}{x}\to \tennsor{}{R}{x}\quad \left(x\in\ff{C}\right),
    &
    \tennsor{y}{\RusEl}{x}\colon \tennsor{y}{B}{x}\to\tennsor{y}{A}{x}\quad \left((y,x)\in\ff{C}^{2}\right).
    \end{array}
    \label{eqn-ring-and-bilinear-maps}
    \end{equation}
    If $\tennsor{h\,}{\underline{\RusEl}}{\,t}(\tennsor{h}{J}{t})\subseteq \tennsor{h}{I}{t}$ for each  $(h,t)\in\ff{C}^{2}$  then  $(\rusEl,\RusEl) $  restricts to  $ \sr{T}(\fk{T})/\sr{J}\to \sr{T}(\fk{S})/\sr{I}$. 
\end{lem}

\begin{proof}
   Consider the composition $\sr{Q}\colon  \sr{T}(\fk{T}) \to \sr{T}(\fk{S})/\sr{I}$ of $(\rusEl,\RusEl)\colon \sr{T}(\fk{T})\to \sr{T}(\fk{S})$, from Lemma~\ref{lem-ring-and-bilinear-maps-induce-functor-between-tensor-cat}, with the canonical quotient   $\sr{T}(\fk{S})\to \sr{T}(\fk{S})/\sr{I}$. 

   The inclusion $\tennsor{h\,}{\underline{\RusEl}}{\,t}(\tennsor{h}{J}{t})\subseteq \tennsor{h}{I}{t}$ ensures that $\tennsor{h}{J}{t}$ lies in the kernel of the induced additive map $\Hom_{\sr{T}(\fk{T})}(t,h)\to \Hom_{\sr{T}(\fk{S})/\sr{I}}(t,h)$. 
   By the universal property of ideals in preadditive categories (see for example \cite[Proposition~1.1]{auslander-reiten-stable}), there exists a unique functor $ \sr{T}(\fk{T})/\sr{J}\to \sr{T}(\fk{S})/\sr{I}$ such that precomposing with the quotient  $\sr{T}(\fk{T})\to \sr{T}(\fk{T})/\sr{J}$ yields $\sr{Q}$. 
\end{proof}


\begin{defn}\label{defn-associativity}
\cite[\S3]{Simson1979}  An  $\ff{S}$-\emph{associativity condition} of a species $(\tennsor{}{R}{x},\tennsor{y}{A}{x})$ is a collection $\bff{c}$ of  $\tennsor{}{R}{z}$-$\tennsor{}{R}{x}$-bilinear maps  $\tennsor{}{c}{zyx}\colon \tennsor{z}{A}{y}\otimes \tennsor{y}{A}{x}\to \tennsor{z}{A}{x}$ indexed by a subset $\ff{S}\subseteq\ff{C}^{3}$   such that 
\begin{equation}
\label{eqn-associativity-condition}
    \tennsor{}{c}{zyw} (\tennsor{}{\iden}{\tennsor{z}{A}{y}}\otimes \tennsor{}{c}{yxw})=\tennsor{}{c}{zxw} (\tennsor{}{c}{zyx}\otimes \tennsor{}{\iden}{\tennsor{x}{A}{w}}),\quad((z,y,w),(y,x,w),(z,x,w),(z,y,x)\in\ff{S}).
\end{equation}
In case $\ff{S}=\ff{C}^{3}$ we call  $\bff{c}$ \emph{full}. 
Given also a collection $\bff{\theta}$ of additive maps $\tennsor{z}{\theta}{}\colon \tennsor{}{R}{z}\to \tennsor{z}{A}{z}$, one for each $z\in \ff{C}$, we call the pair $(\bff{c},\bff{\theta})$ an  $\ff{S}$-\emph{commutativity condition} if 
\begin{align}
\label{eqn-commutativity-conditions}
    \begin{array}{ccc}
\tennsor{}{c}{wwv}(\tennsor{w}{\theta}{}(r)\otimes a)=ra,
&
\tennsor{}{c}{vuu}(b\otimes \tennsor{w}{\theta}{}(s))=bs,
&
\tennsor{v}{\theta}{}(pq)=\tennsor{}{c}{vvv}(\tennsor{v}{\theta}{}(p)\otimes \tennsor{v}{\theta}{}(q))
\end{array}
\end{align}
whenever  $(w,w,v),(v,v,v),(v,u,u)\in\ff{S}$, $r\in\tennsor{}{R}{w}$, $a\in \tennsor{w}{A}{v}$, $b\in \tennsor{v}{A}{u}$, $s\in\tennsor{}{R}{w}$, and $p,q\in\tennsor{}{R}{v}$. 
     %

Note that, if each $\tennsor{z}{\theta}{}$ is $\tennsor{}{R}{z}$-bilinear, then the third equation in \eqref{eqn-commutativity-conditions} follows from either the first or second. 
For example, take $w=v$, $r=p$ and $a=\tennsor{v}{\theta}{}(q)$ in the first. 
\end{defn}

    By a \emph{rng} we mean a possibly non-unital ring.
    %
    Modules over rngs are abelian groups equipped with an associative action that distributes over addition. 

\begin{lem}
       \label{lem-justifying-name-ass-condition}
Let $\bff{c}=(\tennsor{}{c}{zyx})$ be an $\ff{S}$-associativity condition on a species  $\fk{S}=(\tennsor{}{R}{x},\tennsor{y}{A}{x})$. 
\begin{enumerate}
    \item If $(x,x,x)\in \ff{S}$ then $\tennsor{x}{A}{x}$ is a rng with multiplication $\tennsor{}{c}{xxx}$. 
%
    %
    %
%
%
%
%
%
\item If   $(y,y,y),(y,y,x)\in \ff{S}$ then    $\tennsor{y}{A}{x}$ is a left $\tennsor{y}{A}{y}$-module with action $\tennsor{}{c}{yyx}$. 
%
\item If    $(y,y,y),(y,y,x),(y,x,x),(x,x,x)\in \ff{S}$ then    $\tennsor{y}{A}{x}$ is an  $\tennsor{y}{A}{y}$-$\tennsor{x}{A}{x}$-bimodule. 
%
%
\end{enumerate}
Additionally, now assume $(\bff{c},\bff{\theta})$ is an   $\ff{S}$-commutativity condition. 
\begin{enumerate}
\setcounter{enumi}{3}
    \item If $(x,x,x)\in\ff{S}$ then $\tennsor{x}{A}{x}$ is a unital ring and $\tennsor{x}{\theta}{}$ is a ring map. 
    \item If   $(y,y,y),(y,y,x)\in \ff{S}$ then the left $\tennsor{y}{A}{y}$-action on $\tennsor{y}{A}{x}$ extends that of $\tennsor{}{R}{y}$. 
\end{enumerate}
In particular, if each $\tennsor{}{\theta}{x}$ is injective, then  $\bff{c}$ is a commutativity condition as in \cite{Simson1979}.  
\end{lem}

\begin{proof}
   The proof of (1)--(5) is straightforward. 
   Consider items (i)--(iii) of \cite[\S3]{Simson1979}, where Simson introduced the notion of a commutativity condition. 
Item (i)  is the same as \eqref{eqn-associativity-condition}. 
    Item (ii) is the combination of  (1) and (4) above. 
Likewise, (iii) combines (2), (3) and (5). 
The injectivity of each $\tennsor{}{\theta}{x}$ is the only remaining requirement from \cite{Simson1979}, which follows when each $\tennsor{}{R}{x}$ is a division ring, a running assumption throughout \cite{Simson1979}. 
%
    %
    %
\end{proof}
By parts (1)--(3) of Lemma~\ref{lem-justifying-name-ass-condition},  associativity conditions automatically ensure the associativity of various binary operations.
%
%
%
For a given species $\fk{S}$ equipped with a commutativity condition $\bff{c}$, 
 Simson considers a subcategory of representations of $\fk{S}$ that  \emph{satisfy} $\bff{c}$.  
 We recover this category in Lemma~\ref{lem-reps-satisfying-ass-conditions} using an appropriate ideal of 
 %
 $\fk{S}$. 

\begin{nota}
    \label{defn-commutativity-cond-ideal}
Let $\underline{x}\in\ff{C}^{\ff{n}}(h,t)$ be a path of length $\ff{n}\geq 2$. 
For $1\leq \ff{m}\leq \ff{n}-1$ define  
\[
\tennsor{}{\underline{x}}{}( \ff{m})=(h=x_{\ff{n}},\dots,x_{\ff{m}+1},x_{\ff{m}-1},\dots,x_{0}=t)\in\ff{C}^{\ff{n}-1}(h,t).
\]
%
%
For an $\ff{S}$-associativity condition $\bff{c}$  on a species  $\fk{S}=(\tennsor{}{R}{x},\tennsor{y}{A}{x})$, if $(x_{\ff{m}+1},x_{\ff{m}},x_{\ff{m}-1})\in\ff{S}$ let 
\begin{equation}
\label{eqn-c-maps-for-ideal}
    \tennsor{}{c}{\underline{x},\ff{m}}=
    \left\{
\begin{array}{cc}
    \tennsor{}{c}{(x_{\ff{n}},x_{\ff{n}-1},x_{\ff{n}-2})}
    \otimes  \dots \otimes 
    \tennsor{}{\iden}{\tennsor{x_{1}}{A}{t}}
    &
    (\ff{m}=\ff{n-1})
      \\
\tennsor{}{\iden}{\tennsor{h}{A}{x_{\ff{n}-1}}}
    \otimes \dots \otimes 
    \tennsor{}{c}{(x_{2},x_{1},x_{0})}
    &
    (\ff{m}=1)
    \\
\tennsor{}{\iden}{\tennsor{h}{A}{x_{\ff{n}-1}}}
    \otimes \dots \otimes 
    \tennsor{}{c}{(x_{\ff{m}+1},x_{\ff{m}},x_{\ff{m}-1})}
    \otimes  \dots \otimes 
    \tennsor{}{\iden}{\tennsor{x_{1}}{A}{t}}
    &
    (\text{otherwise})
\end{array}
    \right\}.
\end{equation}
So $\tennsor{}{c}{\underline{x},\ff{m}}\colon \tennsor{}{A}{\underline{x}}\to \tennsor{}{A}{\tennsor{}{\underline{x}}{}( \ff{m})}$ is $\tennsor{}{R}{h}$-$\tennsor{}{R}{t}$-bilinear. 
Considering the maps $
\tennsor{}{\iota}{\underline{x}}-\tennsor{}{\iota}{\tennsor{}{\underline{x}}{}( \ff{m})} \tennsor{}{c}{\underline{x},\ff{m}}\colon \tennsor{}{A}{\underline{x}}\to \tennsor{h\,}{\underline{A}}{\,t}$,  where $\tennsor{}{\iota}{\underline{x}}$ and $\tennsor{}{\iota}{\tennsor{}{\tennsor{}{\underline{x}}{}( \ff{m})}{\neg \ff{m}}}$  denote the  inclusions $\tennsor{}{A}{\underline{x}}\subseteq \tennsor{h\,}{\underline{A}}{\,t}$ and $\tennsor{}{A}{\tennsor{}{\underline{x}}{}( \ff{m})}\subseteq \tennsor{h\,}{\underline{A}}{\,t}$, we obtain a map
\begin{equation}
\label{eqn-comm-map}
\tennsor{}{{\bigoplus}}{\ff{n}\geq 2}
{\bigoplus}_{\ff{m}=1}^{\ff{n}-1}
\tennsor{}{{\bigoplus}}{\underline{x}\in\ff{C}^{\ff{n}}(h,t)\colon(x_{\ff{m}+1},x_{\ff{m}},x_{\ff{m}-1})\in\ff{S}}
\tennsor{}{A}{\underline{x}}\to \tennsor{h\,}{\underline{A}}{\,t},
\quad \underline{a}\mapsto  \underline{a}-\tennsor{}{c}{\underline{x},\ff{m}}(\underline{a}),
\end{equation}
for each $t,h\in \ff{C}$, and we let $\tennsor{h}{I}{t}(\bff{c})$ be the image of this $\tennsor{}{R}{h}$-$\tennsor{}{R}{t}$-bilinear map. 
\end{nota}

\begin{lem}
\label{lem-reps-satisfying-ass-conditions}
    Let $\fk{S}=(\tennsor{}{R}{x},\tennsor{y}{A}{x})$ be a species,   $\ff{S}\subseteq \ff{C}^{3}$ and $\bff{c}$   an  $\ff{S}$-associativity condition. 
    %
    %
    Then $\sr{I}(\bff{c})=(\tennsor{h}{I}{t}(\bff{c}))$ is an ideal of $\fk{S}$, and $(\tennsor{x}{M}{},\tennsor{y}{\alpha}{x})$ lies in  $\Rep{\fk{S},\sr{I}(\bff{c})}$ if and only if
    \begin{equation}
    \label{eqn-ass-con-reps-satisfy}
        \begin{array}{cc}
\tennsor{z}{\alpha}{y}\circ (\tennsor{}{\iden}{\tennsor{z}{A}{y}}\otimes \tennsor{y}{\alpha}{x})=\tennsor{z}{\alpha}{x}\circ (c_{zyx}\otimes \tennsor{}{\iden}{\tennsor{x}{M}{}}),&((z,y,x)\in \ff{S}).
\end{array}
    \end{equation} 
Furthermore, in case \eqref{eqn-ass-con-reps-satisfy} holds, $\tennsor{x}{M}{}$ is  a left module over the rng $\tennsor{x}{A}{x}$ for each $(x,x,x)\in\ff{S}$. 
    \end{lem}

\begin{proof}
Let $\underline{v}\in\ff{C}^{\ff{n}+1}(h,t)$, $\underline{v}(h)=(v_{\ff{n}},\dots,t)$ and $\underline{v}(t)=(h,\dots,v_{1})$. 
By construction, 
\[
\begin{array}{cc}
(\tennsor{}{v}{2},\tennsor{}{v}{\ff{1}},t)\in\ff{S}
\Rightarrow
\tennsor{}{\iden}{\tennsor{h}{A}{\tennsor{}{v}{\ff{n}}}}\otimes \tennsor{}{c}{\underline{v}(h),1} 
   =\tennsor{}{c}{\underline{v},1},
   &
   (h,\tennsor{}{v}{\ff{n}},\tennsor{}{v}{\ff{n}-1})\in\ff{S}
   \Rightarrow
   \tennsor{{\ff{n}-1}}{c}{\underline{v}(t)}\otimes 
    \tennsor{}{\iden}{\tennsor{v_{1}}{A}{t}}=\tennsor{{\ff{n}}}{c}{\underline{v}}.
\end{array}
\]
By Remark~\ref{rem-ideals}, it follows that $\sr{I}(\bff{c})$ is an ideal of $\fk{S}$. 
Suppose we are given $\underline{x}\in\ff{C}^{\ff{n}}(h,t)$ such that   $(x_{\ff{m}+1},x_{\ff{m}},x_{\ff{m}-1})\in\ff{S}$ for some  $\ff{n}\geq 2$ and some $\ff{m}=1,\dots,\ff{n}-1$. 
%
%
%
Let $\underline{r}=(h,\dots,x_{\ff{m}+1})$, $\underline{u}=(h,\dots,x_{\ff{m}})$, $\underline{v}=(h,\dots,x_{\ff{m}-1})$ and $\underline{w}=(x_{\ff{m}-1},\dots,t)$. 
If $\ff{m}=1$ then $\underline{w}=(t)$, and we ignore the symbols $\tennsor{}{\alpha}{\underline{w}}$ and $\tennsor{}{\iden}{\tennsor{}{A}{\underline{w}}}$. 
Likewise, if $\ff{m}=\ff{n}-1$, we ignore $\tennsor{}{\alpha}{\underline{r}}$ and $\tennsor{}{\iden}{\tennsor{}{A}{\underline{r}}}$. 
Thus, 
\begin{align*}
    \tennsor{h}{\underline{\alpha}}{t}(\tennsor{}{\iota}{\underline{x}}\otimes\tennsor{}{\iden}{\tennsor{t}{M}{}})
=
\tennsor{}{\alpha}{\underline{x}}
=
\tennsor{}{\alpha}{\underline{r}}(
\tennsor{}{\iden}{\tennsor{}{A}{\underline{r}}}
\otimes
\tennsor{x_{\ff{m}+1}}{\alpha}{x_{\ff{m}}})(
\tennsor{}{\iden}{\tennsor{}{A}{\underline{u}}}
\otimes 
\tennsor{x_{\ff{m}}}{\alpha}{x_{\ff{m}-1}}
)
(\tennsor{}{\iden}{\tennsor{}{A}{\underline{v}}}
\otimes \tennsor{}{\alpha}{\underline{w}}
)
\\
=
\tennsor{}{\alpha}{\underline{r}} 
(\tennsor{}{\iden}{\tennsor{}{A}{\underline{r}}}
\otimes(
\tennsor{x_{\ff{m}+1}}{\alpha}{x_{\ff{m}}}(\tennsor{}{\iden}{\tennsor{x_{\ff{m}+1}}{A}{x_{\ff{m}}}}
\otimes 
\tennsor{x_{\ff{m}}}{\alpha}{x_{\ff{m}-1}})))(
\tennsor{}{\iden}{\tennsor{}{A}{\underline{v}}}
\otimes 
\tennsor{}{\alpha}{\underline{w}}).
\end{align*}
Furthermore, we can rewrite the map $\tennsor{h}{\underline{\alpha}}{t}(
   \tennsor{}{\iota}{\underline{x}(\ff{m})}
   \tennsor{}{c}{\underline{x},\ff{m}}
   \otimes
   \tennsor{}{\iden}{\tennsor{t}{M}{}})$ as 
\begin{align*}
   \tennsor{h}{\underline{\alpha}}{t}(
   \tennsor{}{\iota}{\underline{x}(\ff{m})}
   \tennsor{}{c}{\underline{x},\ff{m}}
   \otimes
   \tennsor{}{\iden}{\tennsor{t}{M}{}})
   =
   \tennsor{}{\alpha}{\underline{x}(\ff{m})}(
   \tennsor{}{c}{\underline{x},\ff{m}}
   \otimes\tennsor{}{\iden}{\tennsor{t}{M}{}})
   =
   \tennsor{}{\alpha}{\underline{x}(\ff{m})}(\tennsor{}{\iden}{\tennsor{}{A}{\underline{r}}}
   \otimes
   \tennsor{}{c}{(x_{\ff{m}+1},x_{\ff{m}},x_{\ff{m}-1})}
   \otimes
   \tennsor{}{\iden}{\tennsor{}{A}{\underline{w}}}
   \otimes 
   \tennsor{}{\iden}{\tennsor{t}{M}{}})
   \\
   =
   \tennsor{}{\alpha}{\underline{r}} 
(\tennsor{}{\iden}{\tennsor{}{A}{\underline{r}}}
\otimes
\tennsor{x_{\ff{m}+1}}{\alpha}{x_{\ff{m}-1}})(\tennsor{}{\iden}{\tennsor{}{A}{\underline{r}}}
\otimes\tennsor{}{\iden}{\tennsor{x_{\ff{m}+1}}{A}{x_{\ff{m}-1}}}
\otimes\tennsor{}{\alpha}{\underline{w}})(\tennsor{}{\iden}{\tennsor{}{A}{\underline{r}}}
   \otimes
   \tennsor{}{c}{(x_{\ff{m}+1},x_{\ff{m}},x_{\ff{m}-1})}
   \otimes
   \tennsor{}{\iden}{\tennsor{}{A}{\underline{w}}}
   \otimes 
   \tennsor{}{\iden}{\tennsor{t}{M}{}})
   \\
   =
   \tennsor{}{\alpha}{\underline{r}} 
(\tennsor{}{\iden}{\tennsor{}{A}{\underline{r}}}
\otimes
\tennsor{x_{\ff{m}+1}}{\alpha}{x_{\ff{m}-1}})(\tennsor{}{\iden}{\tennsor{}{A}{\underline{r}}}
   \otimes
   \tennsor{}{c}{(x_{\ff{m}+1},x_{\ff{m}},x_{\ff{m}-1})}
   \otimes
   \tennsor{}{\iden}{\tennsor{x_{\ff{m}-1}}{M}{}})(\tennsor{}{\iden}{\tennsor{}{A}{\underline{v}}}
\otimes\tennsor{}{\alpha}{\underline{w}})
 \\
   =
   \tennsor{}{\alpha}{\underline{r}} 
(\tennsor{}{\iden}{\tennsor{}{A}{\underline{r}}}
\otimes(
\tennsor{x_{\ff{m}+1}}{\alpha}{x_{\ff{m}-1}}(
   \tennsor{}{c}{(x_{\ff{m}+1},x_{\ff{m}},x_{\ff{m}-1})}
   \otimes
   \tennsor{}{\iden}{\tennsor{x_{\ff{m}-1}}{M}{}})))(\tennsor{}{\iden}{\tennsor{}{A}{\underline{v}}}
\otimes\tennsor{}{\alpha}{\underline{w}}).
\end{align*}
Computing $\tennsor{h}{\underline{\alpha}}{t}((\tennsor{}{\iota}{\underline{x}}\otimes\tennsor{}{\iden}{\tennsor{t}{M}{}}-\tennsor{}{\iota}{\underline{x}(\ff{m})}
   \tennsor{}{c}{\underline{x},\ff{m}})
   \otimes
   \tennsor{}{\iden}{\tennsor{t}{M}{}})$ by taking the difference gives the expression
\[
   \tennsor{}{\alpha}{\underline{r}} 
(\tennsor{}{\iden}{\tennsor{}{A}{\underline{r}}}
\otimes(\tennsor{x_{\ff{m}+1}}{\alpha}{x_{\ff{m}}}(\tennsor{}{\iden}{\tennsor{x_{\ff{m}+1}}{A}{x_{\ff{m}}}}
\otimes 
\tennsor{x_{\ff{m}}}{\alpha}{x_{\ff{m}-1}})-
\tennsor{x_{\ff{m}+1}}{\alpha}{x_{\ff{m}-1}}(
   \tennsor{}{c}{(x_{\ff{m}+1},x_{\ff{m}},x_{\ff{m}-1})}
   \otimes
   \tennsor{}{\iden}{\tennsor{x_{\ff{m}-1}}{M}{}})))(\tennsor{}{\iden}{\tennsor{}{A}{\underline{v}}}
\otimes\tennsor{}{\alpha}{\underline{w}}). 
\]
On the one hand, if the inclusions from Lemma~\ref{lem-ideals} hold, then the expression above is $0$  in case $\ff{n}=2$, $x=x_{0}$, $y=x_{1}$ and $z=x_{2}$, and so \eqref{eqn-ass-con-reps-satisfy} holds. 
On the other hand, if \eqref{eqn-ass-con-reps-satisfy}  holds, then taking $x=x_{\ff{m}-1}$, $y=x_{\ff{m}}$ and $z=x_{\ff{m}+1}$ ensures the expression is $0$. 

For the final statement, note that, by Lemma~\ref{lem-justifying-name-ass-condition}, we have that $\tennsor{}{c}{xxx}$ makes $\tennsor{x}{A}{x}$ into a rng. 
The left action of $\tennsor{x}{A}{x}$ on $\tennsor{x}{M}{}$ is defined by the map $\tennsor{x}{\alpha}{x}$. 
The associativity of this action follows from \eqref{eqn-ass-con-reps-satisfy}.  
Distributivity  over addition is automatic. 
\end{proof}

\begin{lem}
    Let $\bff{c}$ be  a full  commutativity condition   on a species $\fk{S}=(\tennsor{}{R}{x},\tennsor{y}{A}{x})$. 
    Then $\Rep{\fk{S},\sr{I}(\bff{c})}$ is equivalent to $\lMod{\cl{A}}$ where $\cl{A}$ is the preadditive category with object set $\ff{C}$, $\End_{\cl{A}}(x)=\tennsor{}{R}{x}$ and $\Hom_{\cl{A}}(x,y)=\tennsor{y}{A}{x}$ when $x\neq y$. 
\end{lem}

\begin{proof}
    This is a result of Simson \cite[Theorem~2.1]{Simson1979}, who assumes  each $\tennsor{}{R}{x}$ is a division ring, however   the proof  generalises to our setting with no complications. 
\end{proof}

\begin{lem}
\label{lem-full-comm-condition}
Let $\cl{A}$ be a small preadditive category with object set $\ff{C}$, and let 
 \[
    \fk{S}(\cl{A})=(\End_{\cl{A}}(x),\Hom_{\cl{A}}(x,y)).
    \]
Then composition in $\cl{A}$ defines an $\ff{S}$-commutativity condition $\bff{c}(\cl{A})$ on $\fk{S}(\cl{A})$ where $\ff{S}=\ff{C}^{3}$ and such that $\lMod{\cl{A}}$ is equivalent to $\Rep{\fk{S}(\cl{A}),\sr{I}(\bff{c}(\cl{A}))}$. 
\end{lem}

\begin{proof}
Let $\fk{S}(\cl{A})=(\tennsor{}{R}{x},\tennsor{y}{A}{x})$. 
Let $\tennsor{}{a}{yxw}\colon \tennsor{y}{A}{x}\times \tennsor{x}{A}{w}\to \tennsor{y}{A}{w}$ denote the composition maps in $\cl{A}$. 
%
%
Using that composition is associative and bilinear, we ensure $\tennsor{y}{A}{x}$ becomes an $\tennsor{}{R}{y}$-$\tennsor{}{R}{x}$-bimodule, and it follows that $\tennsor{}{a}{yxw}$ is balanced over $\tennsor{}{R}{x}$. 
Thus the universal property of the tensor product defines the $\tennsor{}{R}{y}$-$\tennsor{}{R}{w}$-bilinear maps  $\tennsor{}{c}{yxw}\colon \tennsor{y}{A}{x}\otimes \tennsor{x}{A}{w}\to \tennsor{y}{A}{w}$.

The associativity of composition gives  \eqref{eqn-associativity-condition}. 
To apply Lemma~\ref{lem-justifying-name-ass-condition}, we take $\tennsor{}{\theta}{z}$ to be the identity on $\tennsor{}{R}{z}=\End_{\cl{A}}(z)=\Hom_{\cl{A}}(z,z)=\tennsor{z}{A}{z}$, and so $\bff{c}$ is an $\ff{S}$-commutativity condition. 
Let $\sr{M}\colon \cl{A}\to \Ab$ be an additive functor. 
For each $x\in\ff{C}$ the abelian group $\tennsor{x}{M}{}=\sr{M}(x)$ is an $\tennsor{}{R}{x}$-module, 
and for each $(x,y)\in\ff{C}^{2}$ the map $\tennsor{y}{A}{x}\times \tennsor{x}{M}{}\to \tennsor{y}{M}{}$ given by $(a,m)\mapsto \sr{M}(a)(m)$ is $\tennsor{}{R}{x}$-balanced and left $\tennsor{}{R}{y}$-linear. 
%
%
Conversley, any representation $M=(\tennsor{x}{M}{},\tennsor{y}{\alpha}{x})$ of $\fk{S}(\cl{A})$ defines an assignment $\sr{M}\colon \cl{A}\to \Ab$ on objects and morphisms, preserving identity morphisms.
To say that $\sr{M}$ preserves composition is to say  that \eqref{eqn-ass-con-reps-satisfy} holds, or equivalently by Lemma~\ref{lem-reps-satisfying-ass-conditions}, that $M$ lies in $\Rep{\fk{S}(\cl{A}),\sr{I}(\bff{c})}$. 

It  is straightforward to define $\lMod{\cl{A}}\to\Rep{\fk{S}(\cl{A}),\sr{I}(\bff{c})}$ fully-faithfully on morphisms, which we have seen is dense on objects. 
\end{proof}

Several examples we consider, in Section~\ref{sec-examples}, follow the same pattern. 
To each object in a given small category one associates some ring,  and to each morphism  one associates a bimodule  over the rings given by the domain and codomain, in a functorial way. 
Definition~\ref{defn-byemod} and Lemma~\ref{small-cat-becomes-species} make this precise. 

 We refer the reader to Definition~\ref{defn-double-category}, Example~\ref{example-cats-are-trivially-double-cats}, Example~\ref{example-double-cat-bimodules} and Definition~\ref{defn-lax-double-func}.  

\begin{defn}
\label{defn-byemod}
        Let $\Rings$ be the category of rings, and let $\Byemod$ be the  double category defined as follows. 
       The category of objects and vertical morphisms is $\Rings$. 
       The horizontal morphisms from $R$ to $S$ are the $S$-$R$-bimodules. 
        A square morphism from a $Q$-$R$-bimodule $A$ to an $S$-$T$-bimodule $ B$ is a pair of ring maps $Q\to S$ and $R\to T$, together with a $Q$-$R$-bilinear map $A\to B$. 
    Horizontal composition is given by tensoring bimodules. 
        %
        %
\end{defn}

\begin{lem}\label{small-cat-becomes-species}
    Let $\cl{D}=(\cl{V},\cl{H},\Delta,\sr{L},\sr{R},\odot,\bff{\alpha},\bff{\lambda},\bff{\rho})$ be a double category where
    %
 $\cl{V}$ and $\cl{H}$ are small, and let $\ff{S}\subseteq \ff{C}^{3}$ where $\ff{C}$ is the object set of $\cl{V}$. 
For each  double functor 
\[
(\tennsor{}{\RusI}{0},\tennsor{}{\RusI}{1},\tennsor{}{\RusI}{\odot},\tennsor{}{\RusI}{\Delta})\colon \cl{D}\to \Byemod
\]
the pair $(\tennsor{}{\RusI}{\odot},\tennsor{}{\RusI}{\Delta})$ defines an $\ff{S}$-commutativity condition $\bff{c}(\RusI)$ on the  species
    \[
    \fk{S}(\RusI)
    =
    (\tennsor{}{\RusI}{0}(x),
    \tennsor{}{{\bigoplus}}{\fk{f}\in \cl{H}\,\colon \sr{L}(\fk{f})=y,\,\sr{R}(\fk{f})=x}
    \tennsor{}{\RusI}{1}(\fk{f}
    )).
    \]
    %
    %
    %
\end{lem}

\begin{proof}
    That  $\tennsor{}{\RusI}{0}$ is an object assignment,  to each  $x\in \ff{C}$ we consider a ring  
$\tennsor{}{\RusI}{0}(x)$.  
    %
    
    That  $\tennsor{}{\RusI}{1}$ is an object assignment says, on a horizontal morphism, so a morphism $\fk{f}\in\Hom_{\cl{C}}(x,y)$, we have a horizontal morphism in $\Byemod$, meaning a bimodule $\tennsor{}{\RusI}{1}(\fk{f}
    )$. 
    The compatibility equations between $\tennsor{}{\RusI}{1}$ and $\tennsor{}{\RusI}{0}$ say that $\tennsor{}{\RusI}{1}(\fk{f}
    )$ is an $\tennsor{}{\RusI}{0}(y)$-$\tennsor{}{\RusI}{0}(x)$-bimodule.  
    Hence so is the direct sum over  $\fk{f}\in\cl{H}$ with $\sr{L}(\fk{f})=y$ and $\sr{R}(\fk{f})=x$, so $\fk{S}(\RusI)$ is a species. 

    For each $(u,v)\in\ff{C}^{2}$ let $\tennsor{v}{A}{u}=\tennsor{}{{\bigoplus}}{{\fk{d}\in \cl{H}\,\colon \sr{L}(\fk{d})=v,\,\sr{R}(\fk{d})=u}}\tennsor{}{\RusI}{1}(\fk{d}
    )$, and for each  $\fk{e}\in\cl{H}$ with $\sr{L}(\fk{e})=v$ and $\sr{R}(\fk{e})=u$ write $\tennsor{}{\iota}{\fk{e}}\colon \tennsor{}{\RusI}{1}(\fk{e})\to \tennsor{v}{A}{u}$ for the canonical inclusion. 

    For each $x\in\ff{C}$ and each pair $(\fk{g},\fk{f})$ in $\cl{H}$ with $\sr{R}(\fk{g})=\sr{L}(\fk{f})$, we have bilinear maps 
    \[
    \begin{array}{cc}
    \tennsor{}{{(\tennsor{}{\RusI}{\odot})}}{{\fk{g},\fk{f}}}\colon \tennsor{}{\RusI}{1}(\fk{g}
    )\tennsor{}{\otimes}{{\tennsor{}{\RusI}{0}(y)}}\tennsor{}{\RusI}{1}(\fk{f})\to \tennsor{}{\RusI}{1}(\fk{g}\odot\fk{f}
    ),
         & 
         \tennsor{}{{(\tennsor{}{\RusI}{\Delta})}}{{x}}\colon \tennsor{}{\RusI}{0}(x)\to \tennsor{}{\RusI}{1}(\Delta(x)).
    \end{array}
    \]
    %
    %
    
    %
    Applying universal properties to the bilinear maps $\tennsor{}{\iota}{\fk{gf}}\circ \tennsor{}{{(\tennsor{}{\RusI}{\odot})}}{{\fk{g},\fk{f}}}\colon  \tennsor{}{\RusI}{1}(\fk{g}
    )\tennsor{}{\otimes}{{\tennsor{}{\RusI}{0}(y)}}\tennsor{}{\RusI}{1}(\fk{f})\to\tennsor{z}{A}{x}$,  and using the distributivity of tensor products over direct sums,   there exists a unique bilinear map $\tennsor{}{{(\tennsor{}{\RusI}{\odot})}}{zyx}\colon \tennsor{z}{A}{y}\tennsor{}{\otimes}{\tennsor{}{\RusI}{0}(y)}\tennsor{y}{A}{x}\to \tennsor{z}{A}{x}$ for each $(z,y,x)\in\ff{C}^{3}$ such that
    \begin{equation}
     \tennsor{}{{(\tennsor{}{\RusI}{\odot})}}{zyx}\circ (\tennsor{}{\iota}{\fk{g}}\otimes \tennsor{}{\iota}{\fk{f}})=\tennsor{}{\iota}{\fk{gf}}\circ \tennsor{}{{(\tennsor{}{\RusI}{\odot})}}{{\fk{g},\fk{f}}},\quad(\sr{L}(\fk{g})=z, \,\sr{R}(\fk{g})=y=\sr{L}(\fk{f}),\,\sr{R}(\fk{f})=x).
        \label{eqn-good-compatable-cheers-mate}
    \end{equation}
We assert that $(\tennsor{}{{(\tennsor{}{\RusI}{\odot})}}{zyx})$ is an $\ff{S}$-associativity condition. 
Fix elements $(z,y,w)$, $(y,x,w)$, $(z,x,w)$, $(z,y,x)$ of $\ff{S}$, and fix a triple of horizontal morphisms $(\fk{h},\fk{g},\fk{f})$ such that $\sr{L}(\fk{h})=z$, $\sr{R}(\fk{h})=y=\sr{L}(\fk{g})$, $\sr{R}(\fk{g})=x=\sr{L}(\fk{f})$, and  $\sr{R}(\fk{f})=w$. 

The commutative squares \eqref{eqn-lax-double-functor-comm-conditions-1} each define the outer quadrilateral in 
\[
\begin{tikzcd}[column sep = 0.75cm, row sep = 0.3cm]
 {\tennsor{}{\RusI}{1}(\fk{h})}\otimes {\tennsor{}{\RusI}{1}(\fk{g})}\otimes {\tennsor{}{\RusI}{1}(\fk{f})}
\ar[dd, "\tennsor{}{\iota}{\fk{h}}\,\otimes\, \tennsor{}{\iota}{\fk{g}}\,\otimes\, \tennsor{}{\iota}{\fk{f}}"]
\ar[rr, "\tennsor{}{\iden}{ {\tennsor{}{\RusI}{1}(\fk{h})}}\,\otimes\, {\tennsor{}{{(\tennsor{}{\RusI}{\odot})}}{{\fk{g},\fk{f}}}}"]
\ar[dddddd, bend right = 80, "{\tennsor{}{{(\tennsor{}{\RusI}{\odot})}}{{\fk{h},\fk{g}}}}\,\otimes\,\tennsor{}{\iden}{ {\tennsor{}{\RusI}{1}(\fk{f})}}"', pos = 0.7]
&&
 {\tennsor{}{\RusI}{1}(\fk{h})}\otimes {\tennsor{}{\RusI}{1}(\fk{gf})}
\ar[dd, "\tennsor{}{\iota}{\fk{h}}\,\otimes\, \tennsor{}{\iota}{\fk{g}\fk{f}}"']
\ar[dddddd, bend left = 80, "{\tennsor{}{{(\tennsor{}{\RusI}{\odot})}}{{\fk{h},\fk{gf}}}}", pos = 0.7]
\\
\\
\tennsor{z}{A}{y} \otimes  \tennsor{y}{A}{x} \otimes  \tennsor{x}{A}{w} 
\ar[rr, "\tennsor{}{\iden}{\tennsor{z}{A}{y} }\,\otimes\,  \tennsor{}{{(\tennsor{}{\RusI}{\odot})}}{yxw} "]
\ar[dd, "{\tennsor{}{{(\tennsor{}{\RusI}{\odot})}}{zyx} \,\otimes\, \tennsor{}{\iden}{\tennsor{x}{A}{w} }}"]
&&
\tennsor{z}{A}{y} \otimes  \tennsor{y}{A}{w} 
\ar[dd, "{\tennsor{}{{(\tennsor{}{\RusI}{\odot})}}{zyw}}"']
\\
&(*)&
\\
\tennsor{z}{A}{x} \otimes  \tennsor{x}{A}{w} \ar[rr, "{\tennsor{}{{(\tennsor{}{\RusI}{\odot})}}{zxw} }"']
&&
\tennsor{z}{A}{w} 
\\
\\
 {\tennsor{}{\RusI}{1}(\fk{hg})}\otimes {\tennsor{}{\RusI}{1}(\fk{f})}
\ar[rr, "{\tennsor{}{{(\tennsor{}{\RusI}{\odot})}}{{\fk{hg},\fk{f}}}}"']
\ar[uu, "\tennsor{}{\iota}{\fk{h}\fk{g}}\,\otimes\, \tennsor{}{\iota}{\fk{f}}"']
&&
 {\tennsor{}{\RusI}{1}(\fk{hgf})}
\ar[uu, "\tennsor{}{\iota}{\fk{h}\fk{g}\fk{f}}"]
\end{tikzcd}
\]
To check that \eqref{eqn-associativity-condition} holds, we require that square $(*)$ commutes. 
Firstly, 
\begin{align*}
      (\tennsor{}{\iden}{\tennsor{z}{A}{y} }\otimes \tennsor{}{{(\tennsor{}{\RusI}{\odot})}}{{yxw}})
    \circ 
    (\tennsor{}{\iota}{\fk{h}}\otimes \tennsor{}{\iota}{\fk{g}}\otimes\tennsor{}{\iota}{\fk{f}})
    =
    (\tennsor{}{\iota}{\fk{h}}\otimes \tennsor{}{\iota}{\fk{g}\fk{f}})
    \circ 
    (\tennsor{}{\iden}{ {\tennsor{}{\RusI}{1}(\fk{h})}}\otimes{\tennsor{}{{(\tennsor{}{\RusI}{\odot})}}{{\fk{g},\fk{f}}}}),
     \\
    ( \tennsor{}{{(\tennsor{}{\RusI}{\odot})}}{{zyx}}\otimes \tennsor{}{\iden}{\tennsor{x}{A}{w} })
    \circ 
    (\tennsor{}{\iota}{\fk{h}}\otimes \tennsor{}{\iota}{\fk{g}}\otimes\tennsor{}{\iota}{\fk{f}})
    =
    (\tennsor{}{\iota}{\fk{h}\fk{g}}\otimes \tennsor{}{\iota}{\fk{f}})
    \circ
    ({\tennsor{}{{(\tennsor{}{\RusI}{\odot})}}{{\fk{h},\fk{g}}}}\otimes\tennsor{}{\iden}{ {\tennsor{}{\RusI}{1}(\fk{f})}}),
     \\
     \begin{array}{cc}
     \tennsor{}{{(\tennsor{}{\RusI}{\odot})}}{{zyw}} 
     \circ
     (\tennsor{}{\iota}{\fk{h}}\otimes \tennsor{}{\iota}{\fk{g}\fk{f}})
     =
     \tennsor{}{\iota}{\fk{h}\fk{g}\fk{f}}\circ {\tennsor{}{{(\tennsor{}{\RusI}{\odot})}}{{\fk{h},\fk{gf}}}}
     , 
     &
      \tennsor{}{\iota}{\fk{h}\fk{g}\fk{f}}\circ {\tennsor{}{{(\tennsor{}{\RusI}{\odot})}}{{\fk{hg},\fk{f}}}}
      =
      \tennsor{}{{(\tennsor{}{\RusI}{\odot})}}{{zxw}}
      \circ
      (\tennsor{}{\iota}{\fk{h}\fk{g}}\otimes \tennsor{}{\iota}{\fk{f}}). 
     \end{array}
\end{align*}
Here: each equation is found by considering various specifications of \eqref{eqn-good-compatable-cheers-mate}; and for the first (respectively, second) equation one tensors with $\tennsor{}{\iota}{\fk{h}}$ (respectively, $\tennsor{}{\iota}{\fk{f}}$).   
These equations say that, apart possibly from $(*)$, the internal squares  must commute. 
Altogether,
\begin{align*}
    \tennsor{}{{(\tennsor{}{\RusI}{\odot})}}{{zxw}}
    \circ 
    (\tennsor{}{{(\tennsor{}{\RusI}{\odot})}}{{zyx}} \otimes \tennsor{}{\iden}{\tennsor{x}{A}{w} })
    \circ 
    (\tennsor{}{\iota}{\fk{h}}\otimes \tennsor{}{\iota}{\fk{g}}\otimes\tennsor{}{\iota}{\fk{f}})
    =
    \tennsor{}{{(\tennsor{}{\RusI}{\odot})}}{{zxw}} 
    \circ 
    (\tennsor{}{\iota}{\fk{h}\fk{g}}\otimes \tennsor{}{\iota}{\fk{f}})
    \circ
    ({\tennsor{}{{(\tennsor{}{\RusI}{\odot})}}{{\fk{h},\fk{g}}}\otimes\tennsor{}{\iden}{ {\tennsor{}{\RusI}{1}(\fk{f})}}})
    \\
    =
    \tennsor{}{\iota}{\fk{h}\fk{g}\fk{f}}\circ {\tennsor{}{{(\tennsor{}{\RusI}{\odot})}}{{\fk{hg},\fk{f}}}}
    \circ
    ({\tennsor{}{{(\tennsor{}{\RusI}{\odot})}}{{\fk{h},\fk{g}}}\otimes\tennsor{}{\iden}{ {\tennsor{}{\RusI}{1}(\fk{f})}}})
    =
    \tennsor{}{\iota}{\fk{h}\fk{g}\fk{f}}\circ {\tennsor{}{{(\tennsor{}{\RusI}{\odot})}}{{\fk{h},\fk{gf}}}}
    \circ
    (\tennsor{}{\iden}{ {\tennsor{}{\RusI}{1}(\fk{h})}}\otimes {\tennsor{}{{(\tennsor{}{\RusI}{\odot})}}{{\fk{g},\fk{f}}}})
    \\
    =
    \tennsor{}{{(\tennsor{}{\RusI}{\odot})}}{{zyw}} 
     \circ
     (\tennsor{}{\iota}{\fk{h}}\otimes \tennsor{}{\iota}{\fk{g}\fk{f}})
    \circ
    (\tennsor{}{\iden}{ {\tennsor{}{\RusI}{1}(\fk{h})}}\otimes {\tennsor{}{{(\tennsor{}{\RusI}{\odot})}}{{\fk{g},\fk{f}}}}
    =
    \tennsor{}{{(\tennsor{}{\RusI}{\odot})}}{{zyw}} 
     \circ
    (\tennsor{}{\iden}{\tennsor{z}{A}{y} }\otimes \tennsor{}{{(\tennsor{}{\RusI}{\odot})}}{{yxw}} )
    \circ 
    (\tennsor{}{\iota}{\fk{h}}\otimes \tennsor{}{\iota}{\fk{g}}\otimes\tennsor{}{\iota}{\fk{f}}).
\end{align*}
By the universal property of the coproduct $\tennsor{z}{A}{y}\otimes  \tennsor{y}{A}{x}\otimes  \tennsor{x}{A}{w}$, considering this final equation for each $\fk{f}$, $\fk{g}$ and $\fk{h}$, we have that the central internal square commutes.

    Consider the collection $\tennsor{}{{(\tennsor{}{\RusI}{\Delta})}}{}$ of $\tennsor{}{\RusI}{0}(x)$-bilinear maps $\tennsor{}{{(\tennsor{}{\RusI}{\Delta})}}{x}$.
    As discussed in Definition~\ref{defn-associativity}, since each  is $\tennsor{}{\RusI}{0}(x)$-bilinear, it suffices to observe that the first two equations in \eqref{eqn-commutativity-conditions} hold: and they express the commutativity of \eqref{eqn-lax-double-functor-comm-conditions-2} and \eqref{eqn-lax-double-functor-comm-conditions-3}.
\end{proof}


\begin{lem}
    \label{small-cat-becomes-species-morphisms}
    Let $\cl{D}$ be a double category whose objects form a set $\ff{C}$, let $\ff{S}\subseteq \ff{C}^{3}$ and let $\tennsor{}{\RusI}{}$ and $\tennsorp{}{\RusI}{}$ be  double functors of the form $\cl{D}\to \Byemod$. 
    If $(\tennsor{}{\RusEl}{0},\tennsor{}{\RusEl}{1})\colon \tennsor{}{\RusI}{}\Rightarrow\tennsorp{}{\RusI}{}$ is a double transformation   then  it induces a commutative square of additive functors
    \[\begin{tikzcd}
	{\sr{T}(\fk{S}(\RusI))} & {\sr{T}(\fk{S}(\RusI'))} \\
	{\sr{T}(\fk{S}(\RusI))/\sr{I}(\bff{c}(\RusI))} & {\sr{T}(\fk{S}(\RusI'))/\sr{I}(\bff{c}(\RusI'))}
	\arrow["{(\rusEl,\RusEl)}", from=1-1, to=1-2]
	\arrow[from=1-1, to=2-1, two heads]
	\arrow[from=1-2, to=2-2, two heads]
	\arrow[from=2-1, to=2-2]
\end{tikzcd}\]
where the vertical functors are the canonical quotients. 
\end{lem}

\begin{proof}
    Let $\tennsor{}{\RusI}{}=(\tennsor{}{\RusI}{0},\tennsor{}{\RusI}{1},\tennsor{}{\RusI}{\odot},\tennsor{}{\RusI}{\Delta})$ and $\tennsorp{}{\RusI}{}=(\tennsorp{}{\RusI}{0},\tennsorp{}{\RusI}{1},\tennsorp{}{\RusI}{\odot},\tennsorp{}{\RusI}{\Delta})$. 
    To each object $x\in\ff{C}$ we have a ring map $\tennsor{}{{(\tennsor{}{\RusEl}{0})}}{x}\colon \tennsor{}{\RusI}{0}(x)\to \tennsorp{}{\RusI}{0}(x)$. 
    To each horizontal morphism  $\fk{f}$ we have an additive map  $\tennsor{}{{(\tennsor{}{\RusEl}{1})}}{\fk{f}}\colon \tennsor{}{\RusI}{1}(\fk{f})\to \tennsorp{}{\RusI}{1}(\fk{f})$, bilinear with respect to the restricted action along a pair of ring maps. 
    The compatibility equations in the first condition of double transformations, between \eqref{eqn-lax-double-functor-comm-conditions-3} and \eqref{eqn-lax-double-functor-comm-conditions-4}, say that these ring maps must be $(\tennsor{}{{(\tennsor{}{\RusEl}{0})}}{\sr{L}(\fk{f})},\tennsor{}{{(\tennsor{}{\RusEl}{0})}}{\sr{R}(\fk{f})})$, and that the map $\tennsor{}{{(\tennsor{}{\RusEl}{1})}}{\fk{f}}$ is $\tennsor{}{\RusI}{0}(\sr{L}(\fk{f}))$-$\tennsor{}{\RusI}{0}(\sr{R}(\fk{f}))$-bilinear. 
    %

    By Lemma~\ref{lem-ring-and-bilinear-maps-induce-functor-between-tensor-cat}  there is an additive functor $(\rusEl,\RusEl)\colon \sr{T}(\fk{S}(\RusI))\to \sr{T}(\fk{S}(\RusI'))$ which is the identity on objects. 
    By Lemma~\ref{lem-ring-homs-and-compatible-maps-between-bimodules-ideals} it suffices to show $\tennsor{h\,}{\underline{\RusEl}}{\,t}(\tennsor{h}{I}{t})\subseteq \tennsorp{h}{I}{t}$ for each  $(h,t)\in\ff{C}^{2}$ where $\sr{I}(\bff{c}(\RusI))=(\tennsor{h}{I}{t})$  and  $\sr{I}'(\bff{c}(\RusI'))=(\tennsorp{h}{I}{t})$. 
    If $(z,y,x)\in\ff{S}$, $\fk{g},\fk{f}\in\cl{H}$, $\sr{L}(\fk{g})=z$, $\sr{R}(\fk{g})=y=\sr{L}(\fk{f})$, $\sr{R}(\fk{f})=x$,   $b\in\tennsor{}{\RusI}{1}(\fk{g})$, $a\in \tennsor{}{\RusI}{1}(\fk{f})$ and $c=b\otimes a$, then 
    we have 
   \begin{align*}
         \tennsor{h\,}{\underline{\RusEl}}{\,t}
    (c-\tennsor{}{{(\tennsor{}{\RusI}{\odot})}}{{zyx}}(c))
    =
    \tennsor{}{\RusEl}{(z,y,x)}(c)-\tennsor{}{\RusEl}{(z,x)}(\tennsor{}{{(\tennsor{}{\RusI}{\odot})}}{{zyx}}(c))
    \\
    =
    \tennsor{z}{\RusEl}{y}(b)\otimes \tennsor{y}{\RusEl}{x}( a)-\tennsor{z}{\RusEl}{x}(\tennsor{}{{(\tennsor{}{\RusI}{\odot})}}{{zyx}}(c))
    =
    \tennsor{}{{(\tennsor{}{\RusEl}{1})}}{\fk{g}}(b)\otimes \tennsor{}{{(\tennsor{}{\RusEl}{1})}}{\fk{f}}(a)-\tennsor{}{{(\tennsor{}{\RusEl}{1})}}{\fk{g}\odot\fk{f}}(\tennsor{}{{(\tennsor{}{\RusI}{\odot})}}{{\fk{g},\fk{f}}}(c))
    \\
    =
    \tennsor{}{{(\tennsor{}{\RusEl}{1})}}{\fk{g}}(b)\otimes \tennsor{}{{(\tennsor{}{\RusEl}{1})}}{\fk{f}}(a)-\tennsor{}{{(\tennsorp{}{\RusI}{\odot})}}{{\fk{g},\fk{f}}}(\tennsor{}{{(\tennsor{}{\RusEl}{1})}}{\fk{g}}(b)\otimes \tennsor{}{{(\tennsor{}{\RusEl}{1})}}{\fk{f}}(a))\in \tennsorp{h}{I}{t}.
   \end{align*}
    See Definition~\ref{defn-commutativity-cond-ideal}.  
    Note that the final equality is given by the commutativity of \eqref{eqn-lax-double-functor-comm-conditions-4}. 
\end{proof}

\begin{cor}
\label{cor-cool-nice-cheers-ta-nice1}
    For any double category $\cl{D}$ there is a functor $\sr{T}(\fk{S}(-))/\sr{I}(\boldsymbol{c}(-))$ from the category of double functors $\cl{D}\to \Byemod$ to the category of small preadditive categories. 
\end{cor}

\begin{proof}
    To any object, so double functor $\RusI\colon \cl{D}\to \Byemod$, we associate the species $\fk{S}(\RusI)$ and the commutativity condition $\boldsymbol{c}(\RusI)$ from Lemma~\ref{small-cat-becomes-species}. 
    By Lemma~\ref{lem-reps-satisfying-ass-conditions}, the bimodules given in Definition~\ref{defn-commutativity-cond-ideal} form an ideal $\sr{I}(\boldsymbol{c}(\RusI))$ of the tensor category $\sr{T}(\fk{S}(\RusI))$. 
    By Lemma~\ref{small-cat-becomes-species-morphisms}, and its proof, the assignment $\RusI\mapsto\sr{T}(\fk{S}(\RusI))/\sr{I}(\boldsymbol{c}(\RusI))$ is functorial. 
\end{proof}

\section{Examples and applications}\label{sec-examples}

As discussed above Definition~\ref{defn-byemod}, several situations we consider are constructed via a  double bifunctor into $\Byemod$. 
Of those examples, and in the notation of  Lemma~\ref{small-cat-becomes-species}, many involve defining the $\tennsor{}{\RusI}{0}(y)$-$\tennsor{}{\RusI}{0}(x)$ bimodule $\tennsor{}{\RusI}{1}(\fk{f}
    )$ using a diagram of rings of the form $\tennsor{}{\RusI}{0}(x)\rightarrow \tennsor{}{\RusI}{1}(\fk{f}
    )\leftarrow\tennsor{}{\RusI}{0}(y)$. 
We recall results from work of Davydov--Kong--Runkel \cite{field-theories}.   

\begin{lem}
\label{lem-bilinear-maps}
If the non-dashed arrows below exist and define a  diagram of rings
\begin{equation}
\label{eqn-HI-diagram}
    \begin{tikzcd}[cramped,row sep=tiny]
	&&& J &&& \\
	&& H && I \\
	& E && F && G \\
	A && B && C && D
	\arrow["\lambda", dashed, from=2-3, to=1-4]
	\arrow["\mu"', dashed, from=2-5, to=1-4]
	\arrow["\eta", from=3-2, to=2-3]
	\arrow["\theta"', from=3-4, to=2-3]
	\arrow["\iota", dashed, from=3-4, to=2-5]
	\arrow["\kappa"', dashed, from=3-6, to=2-5]
	\arrow["\alpha", from=4-1, to=3-2]
	\arrow["\beta"', from=4-3, to=3-2]
	\arrow["\gamma", from=4-3, to=3-4]
	\arrow["\delta"', from=4-5, to=3-4]
	\arrow["\varepsilon", dashed, from=4-5, to=3-6]
	\arrow["\zeta"', dashed, from=4-7, to=3-6]
\end{tikzcd}
\end{equation}
that commutes, then the following statements hold. 
\begin{enumerate}
    \item There is a   $ A $-$ C $-bilinear map  $\tennsor{\eta,\alpha}{\nabla}{\theta,\delta}\colon 
    E\tennsor{}{\otimes}{B}F\to H
    $ given by   $e\otimes f\mapsto\eta(e)\theta(f)$. 
    \item If the images of $\beta$ and $\gamma$ are  central then $E\tennsor{}{\otimes}{B}F$ is a ring. 
    %
    If, in addition, the images of $\alpha$ and $\delta$ are central, then so are the images of the ring maps defined by 
    \[
    \begin{array}{cc}
    A\to E\tennsor{}{\otimes}{B}F,\quad a\mapsto\alpha(a)\otimes \tennsor{}{1}{F}
    &
    C\to E\tennsor{}{\otimes}{B}F,\quad a\mapsto\tennsor{}{1}{F}\otimes\delta(c).
    \end{array}
    \]
    \item If the dashed and  non-dashed arrows all exist, and define a commutative diagram of rings, then there is a commutative square of $A$-$D$-bimodules 
\begin{equation}
\label{eqn-comm.square}
     \begin{tikzcd}
        E\tennsor{}{\otimes}{B}F\tennsor{}{\otimes}{C}G
        \ar[rr,"\iden\,\otimes\, \tennsor{\iota,\gamma}{\nabla}{\kappa,\zeta}"]
        \ar[d,"\tennsor{\eta,\alpha}{\nabla}{\theta,\delta}\,\otimes\,\iden "']
        & 
        &
        E\tennsor{}{\otimes }{B}I
        \ar[d,"\tennsor{\lambda\eta,\alpha }{\nabla}{\mu,\kappa\zeta}"]
        \\
        H\tennsor{}{\otimes}{C}G\ar[rr,"\tennsor{\lambda,\eta\alpha }{\nabla}{\mu\kappa,\zeta}"']
        & 
        &
        J
    \end{tikzcd}
\end{equation}
    
\end{enumerate} 
\end{lem}

\begin{proof}
To emphasize when we are regarding a given ring as a bimodule by restricting along rings maps into it, we label the bimodule with the ring maps as subscripts. 
For example, $E$ is a ring, and we write $\tennsor{\alpha}{E}{\beta}$ for the $A$-$B$-bimodule it defines.  

      (1) The map $\Delta\colon\tennsor{\alpha}{E}{\beta} \times \tennsor{\gamma}{F}{\delta}\to \tennsor{\eta\alpha}{H}{\theta\delta}$ given by $\Delta(e, f)=\eta(e)\theta(f)$ is $B$-balanced, since 
\[
      \Delta(e\cdot b ,f)
    =\eta(e\beta( b ))\theta(f)
    =\eta(e)\eta(\beta( b ))\theta(f)
    =\eta(e)\theta(\gamma( b ))\theta(f)
    =\Delta(e, b \cdot f),
\]
Since $\Delta( a \cdot e,f)
    =\eta(\alpha( a )e)\theta(f)
    =\eta(\alpha( a ))\eta(e)\theta(f)
    = a \cdot\Delta(e,f)$, it is also left $A$-linear, and similarly, it is right $C$-linear. 

(2)  
By assuming the images of $\beta$ and $\gamma$ are central, $(eb,f,e',f')-(e,bf,e',f')$ and $(e,f,e'b,f')-(e,f,e',bf')$ lie in  the kernel of the assignment  $(e,f,e',f')\mapsto ee'\otimes ff'$ for each $b\in B$.   
One obtains multiplication on $E\tennsor{}{\otimes}{B}F$ that, by construction, distributes over addition. 
It is straightforward to check $\tennsor{}{1}{E}\otimes \tennsor{}{1}{F}$ serves as the multiplicative unit, where $\tennsor{}{1}{E}$ and $\tennsor{}{1}{F}$ are the units of $E$ and $F$, respectively. 
So, $E\tennsor{}{\otimes}{B}F$ is a ring.

Furthermore, the maps $\tennsor{}{\iden}{E}\otimes \tennsor{}{1}{F}\colon E\to E\tennsor{}{\otimes}{B}F$ and $ \tennsor{}{1}{E}\otimes \tennsor{}{\iden}{F}\colon F\to E\tennsor{}{\otimes}{B}F$, 
    given by $e\mapsto e\otimes \tennsor{}{1}{F}$ and $f\mapsto \tennsor{}{1}{E}\otimes f$, are ring maps, and $\beta(b)\otimes \tennsor{}{1}{F}=\tennsor{}{1}{E}\otimes \gamma(b)$ for each $b\in B$. 
For more details, see for example  \cite[pp.~110--111]{field-theories}.

(3) For the second statement, assuming all arrows exist and the extended diagram commutes, let $\omega=\tennsor{\iota,\gamma}{\nabla}{\kappa,\zeta}$, $\psi=\tennsor{\eta,\alpha}{\nabla}{\theta,\delta}$, $\chi=\tennsor{\lambda\eta,\alpha }{\nabla}{\mu,\kappa\zeta}$ and $\varphi=\tennsor{\lambda,\eta\alpha }{\nabla}{\mu\kappa,\zeta}$, so that
    \begin{align*}
                   \chi(e\otimes \omega(f\otimes g))=
    \chi(e\otimes \iota(f)\kappa(g))=
    \lambda(\eta(e))\mu(\iota(f)\kappa(g))
    =\lambda(\eta(e))\mu(\iota(f))\mu(\kappa(g))
    \\
      =\lambda(\eta(e))\lambda(\theta(f))\mu(\kappa(g))
      =\lambda(\eta(e)\theta(f))\mu(\kappa(g))
      =\varphi(\eta(e)\theta(f)\otimes g )
      =\varphi(\psi(e\otimes f)\otimes g )
    \end{align*}
    for each $e\in \tennsor{\alpha }{E}{\beta}$, $f\in \tennsor{\gamma}{F}{\delta}$ and  $g\in\tennsor{\varepsilon}{G}{\zeta}$. 
\end{proof}

\begin{lem}
\label{lem-lax-doub-func-rings-to-bimod-direct}
     There is a  double functor $(\tennsor{}{\RusZhe}{0},\tennsor{}{\RusZhe}{1},\tennsor{}{\RusZhe}{\odot},\tennsor{}{\RusZhe}{\Delta})\colon \Rings\to \Byemod$ that takes a ring map $\sigma\colon R\to S$ to the $S$-$R$-bimodule $\tennsor{S}{S}{\sigma}$.  
\end{lem}

\begin{proof}
Define $\tennsor{}{\RusZhe}{0}$ on objects (so rings) and vertical morphisms (so identities) by $R\to R$ and $\tennsor{}{\iden}{R}\mapsto \tennsor{}{\iden}{R}$. 
The functor $\tennsor{}{\RusZhe}{1}$ is already defined on horizontal morphisms (so ring maps), and any square morphism in $\Rings$ is an identity. 
Similarly, $\tennsor{}{\RusZhe}{\Delta}$ must $\tennsor{}{\iden}{R}$

For $\tennsor{}{\RusZhe}{\odot}$, let $\zeta\colon D\to C$, $\delta\colon C\to B$ and $\beta\colon B\to A$ be ring maps. 
Let  $G=C$, $I=F=B$, $J=H=E=A$, $\varepsilon=\tennsor{}{\iden}{C}$, $\gamma=\iota=\tennsor{}{\iden}{B}$, $\alpha=\eta=\lambda=\tennsor{}{\iden}{A}$, $\kappa=\delta$  and $\mu=\theta=\beta$. 
It is immediate that \eqref{eqn-HI-diagram} commutes. 
So, the square  \eqref{eqn-comm.square}  becomes 
\[
     \begin{tikzcd}
        \tennsor{}{A}{\beta}\tennsor{}{\otimes }{B} \tennsor{}{B}{\delta}\tennsor{}{\otimes }{C} \tennsor{}{C}{\zeta}
        \ar[rr,"\iden\,\otimes\, \tennsor{}{\nabla}{\delta,\zeta}"]
        \ar[d,"\tennsor{}{\nabla}{\beta,\delta}\,\otimes\,\iden "']
        & 
        &
        \tennsor{}{A}{\beta}\tennsor{}{\otimes }{B} \tennsor{}{C}{\delta\zeta}
        \ar[d,"\tennsor{}{\nabla}{\beta,\delta\zeta}"]
        \\
        \tennsor{}{A}{\beta\delta}\tennsor{}{\otimes}{C}\tennsor{}{C}{\zeta}\ar[rr,"\tennsor{}{\nabla}{\beta\delta,\zeta}"']
        & 
        &
        \tennsor{}{A}{\beta\delta\zeta}
    \end{tikzcd}
\]
This square commutes by Lemma~\ref{lem-bilinear-maps}(3). 
Note that we omitted labels for identity ring maps.  
This ensures \eqref{eqn-lax-double-functor-comm-conditions-1} commutes. 
The remaining details of the proof are straightforward, and we refer the reader to the discussion in work of  B{\'e}nabou \cite[\S2.5]{intro-to-bicats}. 
\end{proof}

\begin{rem}
\label{rem-not-pushout}
    In the category of commutative rings, the tensor product from Lemma~\ref{lem-bilinear-maps}(2) is the pushout of $\beta$ and $\gamma$. 
The key point is that, from the proof, we have 
\[
\tennsor{\eta,\alpha}{\nabla}{\theta,\delta}((e\otimes f)(e'\otimes f'))=\tennsor{\eta,\alpha}{\nabla}{\theta,\delta}(ee'\otimes ff')=\eta(ee')\theta(ff')=\eta(ee')\theta(ff')
\]
which equals $\eta(e)\theta(f)\eta(e')\theta(f')$ when $H$ is commutative, making this bilinear map multiplicative. 
In our generality, however, the tensor product is not the pushout.  
\end{rem}

Before we get on with examples, we define a  double functor $\Rings\to \Byemod$, which factors through a bicategory constructed in \cite[\S4]{field-theories}, which we expose in Lemma~\ref{lem-central-is-double}. 

The aforementioned  bicategory is reminiscent of \emph{cospans} in the category of commutative rings. 
A key difference is that, in our considerations, the ring in the middle of such a cospan need not be commutative. 
Instead we only ask the hypothesis of Lemma~\ref{lem-bilinear-maps}(2) holds, ensuring that we can still consider tensor products of algebras of these noncommutative rings. 
Horizontal composition in a bicategory of cospans is defined via pushouts. 
This, together with Remark~\ref{rem-not-pushout}, warns us to avoid using the language of cospans.  
\begin{nota}
\label{not-rings-cospan-etc}
\cite[p.~114,~(4.24)]{field-theories} The \emph{centre} of a ring map $\sigma\colon R\to S$ is the subring 
    \[
    Z(\sigma)=\{s\in S\mid s\sigma(r)=\sigma(r)s\text{ for each } r\in R\}
    \]
    of $S$. 
    We say $\sigma$ has a \emph{central image} if $Z(\sigma)=S$. 
    Denote the centre of a ring by $Z(R)=\{z\in R\mid rz=zr\text{ for each } r\in R\}$. 
    It follows that $Z(\tennsor{}{\iden}{R})=Z(R)$; see  \cite[Corollary~3.4]{laxfuncs}. 
    %

%
%
Let $\Central_{\Comms}^{\Rings}$ denote a   double category, defined as follows.  

Objects are commutative rings, and the only vertical morphisms   are identity ring maps. 

The horizontal morphisms are quintuples $(A,\alpha,E,\beta,B)$ where  $\alpha\colon A\to E$ and $\beta\colon B\to E$ are ring maps with central images, and $A$ and $B$ are commutative. 
Square  morphisms $(A,\alpha,E,\beta,B)\to (A,\sigma,E',\tau,B)$ are ring maps $\varepsilon\colon E\to E'$  such that $\sigma=\varepsilon\alpha$ and $\tau=\varepsilon\beta$. 
%

The boundary, left, and right functors are defined, in the notation above, on objects by $A\mapsto (\tennsor{}{\iden}{A},\tennsor{}{\iden}{A})$, $(\alpha,\beta)\mapsto A$, and $(\alpha,\beta)\mapsto B$. 
Necessarily,  morphisms are sent to identities.    
The horizontal composition of $(A,\alpha,E,\beta,B)$ and $(B,\gamma,F,\delta,C)$ is 
\begin{equation}
\label{eq-horcompcospans}
    (A,(\tennsor{}{\iden}{E}\otimes \tennsor{}{1}{F})\circ \alpha,E\tennsor{}{\otimes}{B}F,(\tennsor{}{1}{E}\otimes \tennsor{}{\iden}{F})\circ \delta,C).
\end{equation}
%
\end{nota}

\begin{lem}
\label{lem-central-is-double}
    $\Central_{\Comms}^{\Rings}$ is a double category whose vertical morphisms are identities. 
\end{lem} 

\begin{proof}
In \cite{field-theories} a bicategory $\textbf{CAlg}(K)$ is introduced, where $K$ is a field. 
Of the constructions and observations from \cite{field-theories} that we use, none  require that $K$ be a field. 
We explain why the data given by $\Central_{\Comms}^{\Rings}$  is exactly that of ${\textbf{CAlg}(\bb{Z})}$.

Objects of ${\textbf{CAlg}(\bb{Z})}$ are commutative $\bb{Z}$-algebras, so commutative rings; see just after \cite[p.~110,~(4.11)]{field-theories}. 
Objects in the category of morphisms of ${\textbf{CAlg}(\bb{Z})}$, or rather, its horizontal morphisms, are cospans in the category of  $\bb{Z}$-algebras where the domains of the involved rings maps are commutative, and where their images are central; see the first item of \cite[Definition~4.3]{field-theories}, together with \cite[Definition~4.2]{field-theories}. 

By \cite[Lemma~4.4]{field-theories} horizontal composition is defined just as in \eqref{eq-horcompcospans}, and this composition is functorial. 
To each commutative ring one associates the identity cospan, so the diagram formed by $2$ copies of the identity; see \cite[p.~110,~(4.12)]{field-theories}, just after which the authors give natural transformations such that \eqref{eqn-pentagon} and \eqref{eqn-triangle} commute. 
\end{proof} 

\begin{lem}
\label{lem-coolguy}
    There is a  double functor $(\tennsor{}{\RusSha}{0},\tennsor{}{\RusSha}{1},\tennsor{}{\RusSha}{\odot},\tennsor{}{\RusSha}{\Delta})\colon \Central_{\Comms}^{\Rings}\to\Byemod$ that takes a horizontal morphism $(A,\alpha,E,\beta,B)$ to the $A$-$B$-bimodule $\tennsor{\alpha}{E}{\beta}$.  
\end{lem}

\begin{proof}
The category of objects and vertical morphisms in $\Central_{\Comms}^{\Rings}$ is a subcategory of $\Rings$, and we let  $\tennsor{}{\RusSha}{0}$ be the inclusion of this subcategory. 
Define $\tennsor{}{\RusSha}{1}$ on objects as in the statement. 
A square morphism $(A,\alpha,E,\beta,B)\to (A,\sigma,E',\tau,B)$ is a  ring map $E\to E'$  which $\tennsor{}{\RusSha}{1}$ takes to the same function, which is an $A$-$B$-bilinear map $\tennsor{\alpha}{E}{\beta}\to \tennsorp{\sigma}{E}{\tau}$,  since the diagram formed using the involved rings and maps commutes. 

Define $\tennsor{}{\RusSha}{\odot}$ on a horizontal composition \eqref{eq-horcompcospans} by  $A$-$C$-bimodule $\tennsor{\alpha}{E}{\beta} \otimes \tennsor{\gamma}{F}{\delta}$. 
To each  ring $R$  define  $\tennsor{}{(\tennsor{}{\RusSha}{\Delta})}{R}\colon \tennsor{R}{R}{R}\to \tennsor{\tennsor{}{\iden}{R}}{R}{\tennsor{}{\iden}{R}}$  by the identity on $R$.   
It is straightforward to check that these functors and natural transformations indeed define a  double functor. 
In particular, to see that \ref{eqn-lax-double-functor-comm-conditions-1} commutes, apply Lemma~\ref{lem-bilinear-maps}. 
%
%
\end{proof}


\begin{lem}
\label{lem-field-theory-stuff}
      \cite[Theorem~4.12]{field-theories} There is a  double functor $ (\tennsor{}{\RusDe}{0},\tennsor{}{\RusDe}{1},\tennsor{}{\RusDe}{\odot},\tennsor{}{\RusDe}{\Delta})\colon \Rings\to \Central_{\Comms}^{\Rings}$ taking a ring map $\sigma\colon R\to S$ to  the horizontal morphism
     \[
     (Z(S),\tennsor{}{i}{S},Z(\sigma),\tennsor{\sigma}{i}{R},Z(R)).
     \]   
     where $\tennsor{}{i}{S}(s)=s$ for each $s\in Z(S)$ and $\tennsor{\sigma}{i}{R}(r)=\sigma(r)$ for each $r\in Z(R)$. 
\end{lem}

Note that \cite[Theorem~4.4]{laxfuncs} follows as a corollary of Lemma~\ref{lem-coolguy} and Lemma~\ref{lem-field-theory-stuff}. 
\begin{cor}
\label{cor-double-functors-giving-species}
    There exists a (possibly noncommutative) diagram of double functors 
\[
\begin{tikzcd}
\Rings\arrow[dr, "\RusDe"']
\arrow[rr, "\RusZhe"]
&
&
\Byemod
\\
&
\Central_{\Comms}^{\Rings}\arrow[ur, "\RusSha"']
&
\end{tikzcd}
\]
\end{cor}

\begin{proof}
 Combine Lemma~\ref{lem-lax-doub-func-rings-to-bimod-direct}, Lemma~\ref{lem-coolguy} and Lemma~\ref{lem-field-theory-stuff}. 
\end{proof}

In examples we consider, we construct a double functor, either to $\Rings$ or to $\Central_{\Comms}^{\Rings}$, compose  it with a functor  from Corollary~\ref{cor-double-functors-giving-species}, and then apply Corollary~\ref{cor-cool-nice-cheers-ta-nice1}.  
For each such construction we then apply Lemma~\ref{small-cat-becomes-species}.  
Before the examples we add a definition. 
\begin{defn}
\label{defn-acyclic-connected}
    We call a species $\fk{S}=(\tennsor{}{R}{x},\tennsor{y}{A}{x})$ \emph{connected} if for each    $(x,z)\in \ff{C}^{2}$ there exists $\tennsor{}{y}{0},\dots,\tennsor{}{y}{n}\in\ff{C}$ such that  $\tennsor{}{y}{0}=x$, $\tennsor{}{y}{n}=z$, and ($\tennsor{y_{i}}{A}{y_{i+1}}\neq 0$ or $\tennsor{y_{i+1}}{A}{y_{i}}\neq 0$) when $i<n$.  
    We call $\fk{S}$ \emph{acylic} if $\tennsor{x}{\underline{A}}{x}=\tennsor{}{R}{x}$ for each $x\in\ff{C}$, i.e.,  $\tennsor{}{A}{\underline{x}}=0$ for each  $\ff{n}>0$ and $\underline{x}\in\ff{C}^{\ff{n}}(x,x)$.  
\end{defn}

\subsection{Pointwise artinian functors}
\label{subsec-ptwise-artinian}
We start with something familiar.
Let $S$ be a commutative ring, $\cl{C}$ a small category and $[\cl{C},\lMod{S}]$ the category of functors $\cl{C}\to \lMod{S}$. %

\textbf{Realisation}. 
Let $S\cl{C}$ be the $S$-\emph{linearisation} of $\cl{C}$. 
It follows that $[\cl{C},\lMod{S}]$ is equivalent to the category of $S$-linear functors $S\cl{C}\to \lMod{S}$; see for example   \cite[p.~198,~Exercise~6(a)]{MacLane-categories-for-the-working-mathematician}. 
As discussed by Auslander--Reiten \cite[p.~310]{auslander-reiten-stable}, this category is equivalent to   $\lMod{S\cl{C}}$. 

\textbf{Application}. 
We recover a decomposition result for  \emph{pointwise-artinian functors}. 

\begin{cor}\label{cor-actually-understandable-decomposition}
     If $\sr{F}$ is an object in $[\cl{C},\lMod{S}]$ such that each $S$-module $\sr{F}(x)$ is artinian,  then $\sr{F}$ decomposes into a direct sum of strongly indecomposable functors. 
\end{cor}

\begin{proof}                 Consider the functor $\cl{C}\to\Rings$ given by sending all objects to $S$ and all morphisms to $\tennsor{}{\iden}{S}$. 
Composing  with $\RusZhe$  from Corollary~\ref{cor-cool-nice-cheers-ta-nice1}, one obtains a double functor $\RusI$, and hence a preadditive category $\sr{T}(\fk{S}(\RusI))/\sr{I}(\boldsymbol{c}(\RusI))$ by Corollary~\ref{cor-double-functors-giving-species}. 
By construction,  in this case we have $\fk{S}(\RusI)=\fk{S}(S\cl{C})$ and $\boldsymbol{c}(\RusI)=\boldsymbol{c}(S\cl{C})$ in the notation of Lemma~\ref{lem-full-comm-condition}, and so, by the discussion above the statement of the Corollary, there is an equivalence $\Rep{\fk{S}(\RusI),\sr{I}(\boldsymbol{c}(\RusI))}\simeq [\cl{C},\lMod{S}]$.           %
Now apply Lemma~\ref{lem-finiteness-cond-for-existence-strong-decompo-application} and Remark~\ref{rem-decomposition-passes-to-relations}. 
\end{proof}


\subsection{Abelian length categories and distributive lattices}
\label{subsec:abelian-length-krause} 
We exhibit  examples  of representations of species with commutativity relations, defined using embeddings of (often strictly noncommutative) division rings. 
These examples arise in the  setting of recent work of Krause \cite{HENNING-ORDERS},  namely abelian \emph{length} categories, meaning each non-zero object admits a finite chain of subobjects whose composition factors are simple.

\textbf{Motivation}. 
An important invariant of abelian length categories is the \emph{ext-quiver}, with isoclasses of simples as the vertices, and  an arrow $[S]\to [T]$ if and only if $\Ext_{\cl{A}}(T,S)\neq 0$.

The set up for \cite[Theorem~1.2]{HENNING-ORDERS} is an essentially small abelian length category $\cl{A}$, whose ext-quiver has no oriented cycles, and an  object $M$ with distributive subobject lattice,  where $\End_{\cl{A}}(N)\to \End_{\cl{A}}(N/\rad(N))$ is onto when $N$ is a join-irreducible subobject of $M$, and where any object is a subquotient of $M^{n}$ for some $n>0$. 
The conclusion of the cited result says that the set of join-irreducible subobjects $\tennsor{}{M}{x}$ ($x\in \ff{P}$) of $M$ forms a complete list of representatives of the isoclasses of indecomposable projectives in $\cl{A}$. 

\textbf{Realisation}. 
As noted in \cite[Lemma~2.4]{HENNING-ORDERS}, the assumptions above ensure that each ring $\tennsor{}{K}{x}=\End_{\cl{A}}(\tennsor{}{M}{})$ is a division ring, and the corresponding simples are, up to isomorphism, of the form $\tennsor{}{S}{x}=\tennsor{}{M}{x}/\rad(\tennsor{}{M}{x})$. 
Furthermore, there is  a partial order on $\ff{P}$, where $x\leq y$ if and only if  $\tennsor{}{M}{x}$ is a subobject of  $\tennsor{}{M}{y}$, in which case we denote the corresponding monomorphism by $\tennsor{}{\iota}{xy}\colon \tennsor{}{M}{x}\to \tennsor{}{M}{y}$. 
When $x\leq y$ there is a ring map $\tennsor{}{\kappa}{xy}\colon \tennsor{}{K}{y} \to \tennsor{}{K}{x} $ such that  
\[
\begin{array}{cc}
((\tennsor{}{\iota}{xy}\circ \tennsor{}{\kappa}{xy}(\lambda))(m)=\lambda\tennsor{}{\iota}{xy}(m),&
(\lambda\in \tennsor{}{K}{y},\,m\in \tennsor{}{M}{x}). 
\end{array}
\]
We take $\tennsor{}{\iota}{xx}=\tennsor{}{\iden}{{\tennsor{}{M}{x}}}$ in case $x=y$, which ensures $\tennsor{}{\kappa}{xx}=\tennsor{}{\iden}{{\tennsor{}{K}{x}}}$. 
Likewise,  $\tennsor{}{\kappa}{xy}\tennsor{}{\kappa}{yz}=\tennsor{}{\kappa}{xz}$ whenever $x\leq y\leq z$. 
In other words, we have a functor
\begin{equation}
    \begin{array}{ccc}
\boldsymbol{\kappa}\colon \cl{P}^{\op}\to \Rings,
&
x\mapsto \tennsor{}{K}{x},
&
([x\leq y]\colon y\to x)\mapsto (\tennsor{}{\kappa}{xy}\colon\tennsor{}{K}{y} \to \tennsor{}{K}{x}).
\end{array}
\label{eqn-functor-henning}
\end{equation}
\textbf{Application}. 
\emph{Representations} of $\boldsymbol{\kappa}$ are defined  by a left  $\tennsor{}{K}{x}$-module $M(x)$ for each $x\in\ff{P}$,  together with a left $\tennsor{}{K}{y}$-linear map $\tennsor{}{\mu}{xy}\colon \tennsor{*}{\kappa}{xy}M(x)
\to M(y)$ whenever $x\leq y$,  where $\tennsor{*}{\kappa}{xy}\colon \lMod{\tennsor{}{K}{x}}\to \lMod{\tennsor{}{K}{y}}$ is the functor given by restriction along $\tennsor{}{\kappa}{xy}$. 
One also requires compatibility equations to hold. 
Namely, whenever $x\leq y\leq z$, meaning $\tennsor{*}{\kappa}{yz}(\tennsor{*}{\kappa}{xy}M(x))=\tennsor{*}{\kappa}{xz}M(x)$ since $\boldsymbol{\kappa}$ is functorial, one requires  $\tennsor{}{\mu}{xz}=\tennsor{}{\mu}{yz}\circ \tennsor{*}{\kappa}{yz}(\tennsor{}{\mu}{xy})$. 
See \cite[p.~2]{HENNING-ORDERS}. 

A representation  $(M(x),\tennsor{}{\mu}{xy})$ is said to be \emph{pointwise finite-dimensional} provided $M(x)$ is finite-dimensional as a $\tennsor{}{K}{x}$-vector space for each $x\in \ff{P}$. 

\begin{cor}
\label{cor-pwfd-decomps}
    Any pointwise finite-dimensional representation of $\boldsymbol{\kappa}=(\tennsor{}{K}{x},\tennsor{}{\kappa}{xy})$ decomposes uniquely into a direct sum of strongly indecomposable representations.  
\end{cor}

\begin{proof}
Composing $\boldsymbol{\kappa}$ from \eqref{eqn-functor-henning} with $\RusZhe$ from Corollary~\ref{cor-cool-nice-cheers-ta-nice1}, one obtains a double functor $\underline{\boldsymbol{\kappa}}$, and hence a preadditive category of the form $\sr{T}(\fk{S}(\underline{\boldsymbol{\kappa}}))/\sr{I}(\boldsymbol{c}(\underline{\boldsymbol{\kappa}}))$ by Corollary~\ref{cor-cool-nice-cheers-ta-nice1}. 
By definition, and by Lemma~\ref{lem-reps-satisfying-ass-conditions}, the category of  representations of $\boldsymbol{\kappa}$ is equivalent to $\Rep{\fk{S}(\underline{\boldsymbol{\kappa}}),\sr{I}(\boldsymbol{c}(\underline{\boldsymbol{\kappa}}))}$. 
We can now apply Lemma~\ref{lem-finiteness-cond-for-existence-strong-decompo-application} and Remark~\ref{rem-decomposition-passes-to-relations}. 
\end{proof}

For another link to previous work, one may recall \emph{structures} in the sense of Dlab--Ringel \cite{Dlabringstructures}, which are recovered from the discussion above by taking $\ff{P}$ to be a finite set, but where representations must have a projective socle, or equivalently, the maps $\tennsor{}{\mu}{xy}$ must be injective.  
The main result of \cite{Dlabringstructures} is a generalisation of a theorem of Kleiner, whose work is recovered when all the division rings $\tennsor{}{K}{x}$ coincide with a single field $K$.   

\subsection{Persistent homology with various coefficients and the field choice problem}
\label{sec-field-choice}

We exhibit  examples  of species with commutativity relations, defined using quotient maps between finite rings of the form $\bb{Z}/m\bb{Z}$ for $0\neq m\in\bb{Z}$. %
These examples come from \emph{topological data analysis} (TDA).  
We recall Sections 3.1, 4.1 and 4.2 of work Boissonnat--Maria \cite{Boissonnatmariacomputing}, who studied \emph{multi-field persistence diagrams}; see for example \cite[p.~64,~Figure~1]{Boissonnatmariacomputing}. 
 For a poset $\cl{P}$ the (usual) category of \emph{persistence modules} over a (single) field $K$ is  $[\cl{P},\lMod{K}]$. 

\textbf{Background}. 
Roughly, the classical story of  persistence modules in TDA  is as follows. 

At each point $\texttt{p}$ of a point cloud $\texttt{C}$, centre a ball $B_{\texttt{p}}(\varepsilon)$ of radius $\varepsilon\geq 0$. 
Next, consider the \emph{{\v C}ech  complex} $\tennsor{}{\bb{X}}{\varepsilon}$, the singular nerve complex given by $(B_{\texttt{p}}(\varepsilon) \mid \texttt{p}\in \texttt{C})$. 
So, $\tennsor{}{\bb{X}}{\varepsilon}$ consists of subsets $\Sigma\subseteq \texttt{C}$ such that $\tennsor{}{\bigcap}{\texttt{p}\in\Sigma}B_{\texttt{p}}(\varepsilon)\neq \emptyset$, and  $(\tennsor{}{\bb{X}}{\varepsilon}\mid\varepsilon\in[0,\infty))$ defines a filtered simplicial complex. 
The homology is said to \emph{persist} from $\delta$ to $\zeta$ if $H_{n}(\tennsor{}{\bb{X}}{\varepsilon})\neq 0$ for $\delta\leq \varepsilon\leq \zeta$. 
To see how homology may persist as one varies the radius $\varepsilon$, it can help to consider a picture, say
\vspace{1mm}
\[
\begin{array}{cccc}
\resizebox{0.19\textwidth}{!}{%
\begin{circuitikz}
\tikzstyle{every node}=[font=\normalsize]
\draw [ fill={rgb,255:red,186; green,186; blue,186} , line width=0.2pt ] (3.25,12) circle (0.295cm);
\draw [ fill={rgb,255:red,186; green,186; blue,186} , line width=0.2pt ] (2.5,11.5) circle (0.295cm);
\draw [ fill={rgb,255:red,186; green,186; blue,186} , line width=0.2pt ] (2.5,10.75) circle (0.295cm);
\draw [ fill={rgb,255:red,186; green,186; blue,186} , line width=0.2pt ] (3,10) circle (0.295cm);
\draw [ fill={rgb,255:red,186; green,186; blue,186} , line width=0.2pt ] (4,12) circle (0.295cm);
\draw [ fill={rgb,255:red,186; green,186; blue,186} , line width=0.2pt ] (4.75,12.5) circle (0.295cm);
\draw [ fill={rgb,255:red,186; green,186; blue,186} , line width=0.2pt ] (5.5,12) circle (0.295cm);
\draw [ fill={rgb,255:red,186; green,186; blue,186} , line width=0.2pt ] (5.5,11.25) circle (0.295cm);
\draw [ fill={rgb,255:red,186; green,186; blue,186} , line width=0.2pt ] (6.25,12.25) circle (0.295cm);
\draw [ fill={rgb,255:red,186; green,186; blue,186} , line width=0.2pt ] (7,12) circle (0.295cm);
\draw [ fill={rgb,255:red,186; green,186; blue,186} , line width=0.2pt ] (4,10.75) circle (0.295cm);
\draw [ fill={rgb,255:red,186; green,186; blue,186} , line width=0.2pt ] (3.5,9.5) circle (0.295cm);
\draw [ fill={rgb,255:red,186; green,186; blue,186} , line width=0.2pt ] (4.25,9.5) circle (0.295cm);
\draw [ fill={rgb,255:red,186; green,186; blue,186} , line width=0.2pt ] (5,9.75) circle (0.295cm);
\draw [ fill={rgb,255:red,186; green,186; blue,186} , line width=0.2pt ] (5.5,10.5) circle (0.295cm);
\draw [ fill={rgb,255:red,186; green,186; blue,186} , line width=0.2pt ] (6.25,10.25) circle (0.295cm);
\draw [ fill={rgb,255:red,186; green,186; blue,186} , line width=0.2pt ] (7,10.25) circle (0.295cm);
\draw [ fill={rgb,255:red,186; green,186; blue,186} , line width=0.2pt ] (7.75,10.75) circle (0.295cm);
\draw [ fill={rgb,255:red,186; green,186; blue,186} , line width=0.2pt ] (7.75,11.5) circle (0.295cm);
\draw [ fill={rgb,255:red,186; green,186; blue,186} , line width=0.2pt ] (0.75,12.75) circle (0.295cm);
\draw [ fill={rgb,255:red,186; green,186; blue,186} , line width=0.2pt ] (1,10.25) circle (0.295cm);
\draw [ fill={rgb,255:red,186; green,186; blue,186} , line width=0.2pt ] (5.75,9.75) circle (0.295cm);
\draw [ fill={rgb,255:red,186; green,186; blue,186} , line width=0.2pt ] (7.25,9.25) circle (0.295cm);
\draw [ fill={rgb,255:red,186; green,186; blue,186} , line width=0.2pt ] (1.5,12.75) circle (0.295cm);
\draw [ fill={rgb,255:red,186; green,186; blue,186} , line width=0.2pt ] (4.25,13.75) circle (0.295cm);
\draw [ fill={rgb,255:red,186; green,186; blue,186} , line width=0.2pt ] (6,13.5) circle (0.295cm);
\draw [ fill={rgb,255:red,186; green,186; blue,186} , line width=0.2pt ] (2.25,9.5) circle (0.295cm);
\draw [ fill={rgb,255:red,186; green,186; blue,186} , line width=0.2pt ] (1.5,9.75) circle (0.295cm);
\draw [ fill={rgb,255:red,186; green,186; blue,186} , line width=0.2pt ] (1.25,11) circle (0.295cm);
\draw [ fill={rgb,255:red,186; green,186; blue,186} , line width=0.2pt ] (6.5,9.5) circle (0.295cm);
\draw [ fill={rgb,255:red,186; green,186; blue,186} , line width=0.2pt ] (7.75,12.75) circle (0.295cm);
\draw [ fill={rgb,255:red,186; green,186; blue,186} , line width=0.2pt ] (6.75,13.75) circle (0.295cm);
\draw [ fill={rgb,255:red,186; green,186; blue,186} , line width=0.2pt ] (8.75,10.75) circle (0.295cm);
\end{circuitikz}
}
&
\resizebox{0.19\textwidth}{!}{%
\begin{circuitikz}
\tikzstyle{every node}=[font=\normalsize]
\draw [ fill={rgb,255:red,186; green,186; blue,186} , line width=0.2pt ] (3.25,12) circle (0.5cm);
\draw [ fill={rgb,255:red,186; green,186; blue,186} , line width=0.2pt ] (2.5,11.5) circle (0.5cm);
\draw [ fill={rgb,255:red,186; green,186; blue,186} , line width=0.2pt ] (2.5,10.75) circle (0.5cm);
\draw [ fill={rgb,255:red,186; green,186; blue,186} , line width=0.2pt ] (3,10) circle (0.5cm);
\draw [ fill={rgb,255:red,186; green,186; blue,186} , line width=0.2pt ] (4,12) circle (0.5cm);
\draw [ fill={rgb,255:red,186; green,186; blue,186} , line width=0.2pt ] (4.75,12.5) circle (0.5cm);
\draw [ fill={rgb,255:red,186; green,186; blue,186} , line width=0.2pt ] (5.5,12) circle (0.5cm);
\draw [ fill={rgb,255:red,186; green,186; blue,186} , line width=0.2pt ] (5.5,11.25) circle (0.5cm);
\draw [ fill={rgb,255:red,186; green,186; blue,186} , line width=0.2pt ] (6.25,12.25) circle (0.5cm);
\draw [ fill={rgb,255:red,186; green,186; blue,186} , line width=0.2pt ] (7,12) circle (0.5cm);
\draw [ fill={rgb,255:red,186; green,186; blue,186} , line width=0.2pt ] (4,10.75) circle (0.5cm);
\draw [ fill={rgb,255:red,186; green,186; blue,186} , line width=0.2pt ] (3.5,9.5) circle (0.5cm);
\draw [ fill={rgb,255:red,186; green,186; blue,186} , line width=0.2pt ] (4.25,9.5) circle (0.5cm);
\draw [ fill={rgb,255:red,186; green,186; blue,186} , line width=0.2pt ] (5,9.75) circle (0.5cm);
\draw [ fill={rgb,255:red,186; green,186; blue,186} , line width=0.2pt ] (5.5,10.5) circle (0.5cm);
\draw [ fill={rgb,255:red,186; green,186; blue,186} , line width=0.2pt ] (6.25,10.25) circle (0.5cm);
\draw [ fill={rgb,255:red,186; green,186; blue,186} , line width=0.2pt ] (7,10.25) circle (0.5cm);
\draw [ fill={rgb,255:red,186; green,186; blue,186} , line width=0.2pt ] (7.75,10.75) circle (0.5cm);
\draw [ fill={rgb,255:red,186; green,186; blue,186} , line width=0.2pt ] (7.75,11.5) circle (0.5cm);
\draw [ fill={rgb,255:red,186; green,186; blue,186} , line width=0.2pt ] (0.75,12.75) circle (0.5cm);
\draw [ fill={rgb,255:red,186; green,186; blue,186} , line width=0.2pt ] (1,10.25) circle (0.5cm);
\draw [ fill={rgb,255:red,186; green,186; blue,186} , line width=0.2pt ] (5.75,9.75) circle (0.5cm);
\draw [ fill={rgb,255:red,186; green,186; blue,186} , line width=0.2pt ] (7.25,9.25) circle (0.5cm);
\draw [ fill={rgb,255:red,186; green,186; blue,186} , line width=0.2pt ] (1.5,12.75) circle (0.5cm);
\draw [ fill={rgb,255:red,186; green,186; blue,186} , line width=0.2pt ] (4.25,13.75) circle (0.5cm);
\draw [ fill={rgb,255:red,186; green,186; blue,186} , line width=0.2pt ] (6,13.5) circle (0.5cm);
\draw [ fill={rgb,255:red,186; green,186; blue,186} , line width=0.2pt ] (2.25,9.5) circle (0.5cm);
\draw [ fill={rgb,255:red,186; green,186; blue,186} , line width=0.2pt ] (1.5,9.75) circle (0.5cm);
\draw [ fill={rgb,255:red,186; green,186; blue,186} , line width=0.2pt ] (1.25,11) circle (0.5cm);
\draw [ fill={rgb,255:red,186; green,186; blue,186} , line width=0.2pt ] (6.5,9.5) circle (0.5cm);
\draw [ fill={rgb,255:red,186; green,186; blue,186} , line width=0.2pt ] (7.75,12.75) circle (0.5cm);
\draw [ fill={rgb,255:red,186; green,186; blue,186} , line width=0.2pt ] (6.75,13.75) circle (0.5cm);
\draw [ fill={rgb,255:red,186; green,186; blue,186} , line width=0.2pt ] (8.75,10.75) circle (0.5cm);
\end{circuitikz}
}
&
\resizebox{0.19\textwidth}{!}{%
\begin{circuitikz}
\tikzstyle{every node}=[font=\normalsize]
\draw [ fill={rgb,255:red,186; green,186; blue,186} , line width=0.2pt ] (3.25,12) circle (0.675cm);
\draw [ fill={rgb,255:red,186; green,186; blue,186} , line width=0.2pt ] (2.5,11.5) circle (0.675cm);
\draw [ fill={rgb,255:red,186; green,186; blue,186} , line width=0.2pt ] (2.5,10.75) circle (0.675cm);
\draw [ fill={rgb,255:red,186; green,186; blue,186} , line width=0.2pt ] (3,10) circle (0.675cm);
\draw [ fill={rgb,255:red,186; green,186; blue,186} , line width=0.2pt ] (4,12) circle (0.675cm);
\draw [ fill={rgb,255:red,186; green,186; blue,186} , line width=0.2pt ] (4.75,12.5) circle (0.675cm);
\draw [ fill={rgb,255:red,186; green,186; blue,186} , line width=0.2pt ] (5.5,12) circle (0.675cm);
\draw [ fill={rgb,255:red,186; green,186; blue,186} , line width=0.2pt ] (5.5,11.25) circle (0.675cm);
\draw [ fill={rgb,255:red,186; green,186; blue,186} , line width=0.2pt ] (6.25,12.25) circle (0.675cm);
\draw [ fill={rgb,255:red,186; green,186; blue,186} , line width=0.2pt ] (7,12) circle (0.675cm);
\draw [ fill={rgb,255:red,186; green,186; blue,186} , line width=0.2pt ] (4,10.75) circle (0.675cm);
\draw [ fill={rgb,255:red,186; green,186; blue,186} , line width=0.2pt ] (3.5,9.5) circle (0.675cm);
\draw [ fill={rgb,255:red,186; green,186; blue,186} , line width=0.2pt ] (4.25,9.5) circle (0.675cm);
\draw [ fill={rgb,255:red,186; green,186; blue,186} , line width=0.2pt ] (5,9.75) circle (0.675cm);
\draw [ fill={rgb,255:red,186; green,186; blue,186} , line width=0.2pt ] (5.5,10.5) circle (0.675cm);
\draw [ fill={rgb,255:red,186; green,186; blue,186} , line width=0.2pt ] (6.25,10.25) circle (0.675cm);
\draw [ fill={rgb,255:red,186; green,186; blue,186} , line width=0.2pt ] (7,10.25) circle (0.675cm);
\draw [ fill={rgb,255:red,186; green,186; blue,186} , line width=0.2pt ] (7.75,10.75) circle (0.675cm);
\draw [ fill={rgb,255:red,186; green,186; blue,186} , line width=0.2pt ] (7.75,11.5) circle (0.675cm);
\draw [ fill={rgb,255:red,186; green,186; blue,186} , line width=0.2pt ] (0.75,12.75) circle (0.675cm);
\draw [ fill={rgb,255:red,186; green,186; blue,186} , line width=0.2pt ] (1,10.25) circle (0.675cm);
\draw [ fill={rgb,255:red,186; green,186; blue,186} , line width=0.2pt ] (5.75,9.75) circle (0.675cm);
\draw [ fill={rgb,255:red,186; green,186; blue,186} , line width=0.2pt ] (7.25,9.25) circle (0.675cm);
\draw [ fill={rgb,255:red,186; green,186; blue,186} , line width=0.2pt ] (1.5,12.75) circle (0.675cm);
\draw [ fill={rgb,255:red,186; green,186; blue,186} , line width=0.2pt ] (4.25,13.75) circle (0.675cm);
\draw [ fill={rgb,255:red,186; green,186; blue,186} , line width=0.2pt ] (6,13.5) circle (0.675cm);
\draw [ fill={rgb,255:red,186; green,186; blue,186} , line width=0.2pt ] (2.25,9.5) circle (0.675cm);
\draw [ fill={rgb,255:red,186; green,186; blue,186} , line width=0.2pt ] (1.5,9.75) circle (0.675cm);
\draw [ fill={rgb,255:red,186; green,186; blue,186} , line width=0.2pt ] (1.25,11) circle (0.675cm);
\draw [ fill={rgb,255:red,186; green,186; blue,186} , line width=0.2pt ] (6.5,9.5) circle (0.675cm);
\draw [ fill={rgb,255:red,186; green,186; blue,186} , line width=0.2pt ] (7.75,12.75) circle (0.675cm);
\draw [ fill={rgb,255:red,186; green,186; blue,186} , line width=0.2pt ] (6.75,13.75) circle (0.675cm);
\draw [ fill={rgb,255:red,186; green,186; blue,186} , line width=0.2pt ] (8.75,10.75) circle (0.675cm);
\end{circuitikz}
}
&
\resizebox{0.19\textwidth}{!}{%
\begin{circuitikz}
\tikzstyle{every node}=[font=\normalsize]
\draw [ fill={rgb,255:red,186; green,186; blue,186} , line width=0.2pt ] (3.25,12) circle (1.05cm);
\draw [ fill={rgb,255:red,186; green,186; blue,186} , line width=0.2pt ] (2.5,11.5) circle (1.05cm);
\draw [ fill={rgb,255:red,186; green,186; blue,186} , line width=0.2pt ] (2.5,10.75) circle (1.05cm);
\draw [ fill={rgb,255:red,186; green,186; blue,186} , line width=0.2pt ] (3,10) circle (1.05cm);
\draw [ fill={rgb,255:red,186; green,186; blue,186} , line width=0.2pt ] (4,12) circle (1.05cm);
\draw [ fill={rgb,255:red,186; green,186; blue,186} , line width=0.2pt ] (4.75,12.5) circle (1.05cm);
\draw [ fill={rgb,255:red,186; green,186; blue,186} , line width=0.2pt ] (5.5,12) circle (1.05cm);
\draw [ fill={rgb,255:red,186; green,186; blue,186} , line width=0.2pt ] (5.5,11.25) circle (1.05cm);
\draw [ fill={rgb,255:red,186; green,186; blue,186} , line width=0.2pt ] (6.25,12.25) circle (1.05cm);
\draw [ fill={rgb,255:red,186; green,186; blue,186} , line width=0.2pt ] (7,12) circle (1.05cm);
\draw [ fill={rgb,255:red,186; green,186; blue,186} , line width=0.2pt ] (4,10.75) circle (1.05cm);
\draw [ fill={rgb,255:red,186; green,186; blue,186} , line width=0.2pt ] (3.5,9.5) circle (1.05cm);
\draw [ fill={rgb,255:red,186; green,186; blue,186} , line width=0.2pt ] (4.25,9.5) circle (1.05cm);
\draw [ fill={rgb,255:red,186; green,186; blue,186} , line width=0.2pt ] (5,9.75) circle (1.05cm);
\draw [ fill={rgb,255:red,186; green,186; blue,186} , line width=0.2pt ] (5.5,10.5) circle (1.05cm);
\draw [ fill={rgb,255:red,186; green,186; blue,186} , line width=0.2pt ] (6.25,10.25) circle (1.05cm);
\draw [ fill={rgb,255:red,186; green,186; blue,186} , line width=0.2pt ] (7,10.25) circle (1.05cm);
\draw [ fill={rgb,255:red,186; green,186; blue,186} , line width=0.2pt ] (7.75,10.75) circle (1.05cm);
\draw [ fill={rgb,255:red,186; green,186; blue,186} , line width=0.2pt ] (7.75,11.5) circle (1.05cm);
\draw [ fill={rgb,255:red,186; green,186; blue,186} , line width=0.2pt ] (0.75,12.75) circle (1.05cm);
\draw [ fill={rgb,255:red,186; green,186; blue,186} , line width=0.2pt ] (1,10.25) circle (1.05cm);
\draw [ fill={rgb,255:red,186; green,186; blue,186} , line width=0.2pt ] (5.75,9.75) circle (1.05cm);
\draw [ fill={rgb,255:red,186; green,186; blue,186} , line width=0.2pt ] (7.25,9.25) circle (1.05cm);
\draw [ fill={rgb,255:red,186; green,186; blue,186} , line width=0.2pt ] (1.5,12.75) circle (1.05cm);
\draw [ fill={rgb,255:red,186; green,186; blue,186} , line width=0.2pt ] (4.25,13.75) circle (1.05cm);
\draw [ fill={rgb,255:red,186; green,186; blue,186} , line width=0.2pt ] (6,13.5) circle (1.05cm);
\draw [ fill={rgb,255:red,186; green,186; blue,186} , line width=0.2pt ] (2.25,9.5) circle (1.05cm);
\draw [ fill={rgb,255:red,186; green,186; blue,186} , line width=0.2pt ] (1.5,9.75) circle (1.05cm);
\draw [ fill={rgb,255:red,186; green,186; blue,186} , line width=0.2pt ] (1.25,11) circle (1.05cm);
\draw [ fill={rgb,255:red,186; green,186; blue,186} , line width=0.2pt ] (6.5,9.5) circle (1.05cm);
\draw [ fill={rgb,255:red,186; green,186; blue,186} , line width=0.2pt ] (7.75,12.75) circle (1.05cm);
\draw [ fill={rgb,255:red,186; green,186; blue,186} , line width=0.2pt ] (6.75,13.75) circle (1.05cm);
\draw [ fill={rgb,255:red,186; green,186; blue,186} , line width=0.2pt ] (8.75,10.75) circle (1.05cm);
\end{circuitikz}
}
\\
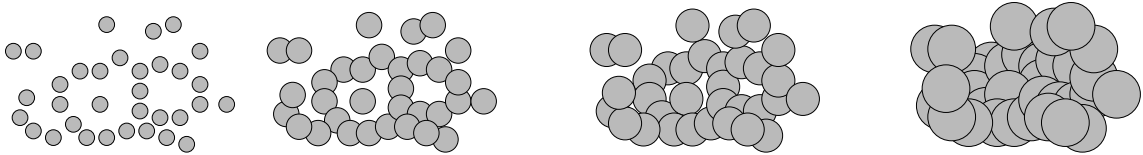
\text{\footnotesize{No homology.}}
&
\text{\footnotesize{Homology appears...}}
&
\text{\footnotesize{...persists, and more appears...}}
&
\text{\footnotesize{...and disappears.}}\\
\vspace{-6mm}
\end{array}
\]
One considers homology $H_{n}(\tennsor{}{\bb{X}}{\varepsilon};S)$ with coefficients in a principal ideal domain $S$, and doing so defines a functor $H_{n}(\bb{X}_{-};S)\colon [0,\infty)\to \lMod{S}$. 
Often, but not always, one takes $S=K$ a field, in which case the functor $H_{n}(\bb{X}_{-};K)\colon [0,\infty)\to \lMod{K}$ is pointwise finite-dimensional, since each $\bb{Z}$-module $H_{n}(\tennsor{}{\bb{X}}{\varepsilon})$ is finitely generated. 
Note that $[\cl{P},\lMod{K}]$ was seen in Sections \ref{subsec-ptwise-artinian} and \ref{subsec:abelian-length-krause}. 
Corollary~\ref{cor-actually-understandable-decomposition} and Corollary~\ref{cor-pwfd-decomps} each recover the  decomposition theorem for persistence modules of Botnan--Crawley-Boevey \cite[Theorem~1.1]{botnan-crawley-boevey-persistence}. 
The indecomposable persistence modules occurring in such a decomposition define the \emph{persistence barcode}. 

In case $\cl{P}=[0,\infty)$, by work of Crawley-Boevey \cite{crawley-boevey-intervals},  indecomposable pointwise finite-dimensional persistence modules have the form $\tennsor{}{K}{I}$ where $\tennsor{}{K}{I}(\delta\leq \varepsilon)=\tennsor{}{\iden}{K}$ if $\delta, \varepsilon\in I$ and $\tennsor{}{K}{I}( \zeta)=0$ if $\zeta\notin I$. 
The \emph{persistence diagram} of a persistence module $M$ is the set of pairs $(\inf(I),\sup(I))$ such that $\tennsor{}{K}{I}$ is a direct summand of $M$.  

\textbf{Motivation}. 
A subtlety in the construction above is that the structure of  $H_{n}(\tennsor{}{\bb{X}}{-};K)$ depends on $K$. 
Explicit examples are found in work of Obayashi--Yoshiwaki  \cite[Examples~1.5,~3.1,~3.2,~3.3]{ObayashiYoshiwakiFieldChoice}, who studied  the \emph{field choice problem} in TDA. 
For simplicity, assume the potential radii run through a discrete set, indexed by $\bb{N}$, and so we have functors of the form $H_{n}(\bb{X}_{-};\bb{Z}/m\bb{Z})\colon\bb{N}\to \lMod{\bb{Z}/m\bb{Z}}$ for each $m\in\bb{Z}$. 
Given distinct primes $\tennsor{}{q}{1},\dots,\tennsor{}{q}{\ell}$, consider the groups $H_{n}(\bb{X}_{i};\bb{Z}/\tennsor{}{q}{S}\bb{Z})$, where $S$ is an \emph{indexing} set, meaning a  non-empty  subset  of $ \{1,\dots,\ell\}$, and $\tennsor{}{q}{S}$ is the product over $i\in S$ of the $\tennsor{}{q}{i}$.

One  problem looked at in \cite{Boissonnatmariacomputing} is, roughly speaking, to relate the functors $H_{n}(\bb{X}_{-};\bb{Z}/\tennsor{}{q}{S}\bb{Z})$ as $S$ varies among indexing sets. 
To do this, these authors combine  the Chinese remainder theorem  with the universal coefficient theorem for simplicial homology. 
For $m\in\bb{Z}$ let $\tennsor{}{\bb{M}}{rc}(\bb{Z}/m\bb{Z})$ be the set of matrices in $\bb{Z}/m\bb{Z}$ with $r$ rows and $c$ columns. 
%
%
Let  $q=\tennsor{}{q}{1}\dots \tennsor{}{q}{\ell}$, suppose we are given the output of the functor  $H_{n}(\bb{X}_{-};\bb{Z}/q\bb{Z})\colon \bb{N}\to \bb{Z}/q\bb{Z}$, and assume for each $i\in\bb{N}$ that the $\bb{Z}$-module $H_{n}(\bb{X}_{i})$ is free, say of rank $\tennsor{}{d}{i}\geq 0$. 
Due to the negation of  \cite[Conjecture~1.4]{ObayashiYoshiwakiFieldChoice}, the structure of $H_{n}(\bb{X}_{i};\bb{Z}/m\bb{Z})$  may still depend on $m$, even though we assume each $H_{n}(\bb{X}_{i})$ is free. 
This assumption implies 
\[
\Hom_{}(H_{n}(\bb{X}_{i};\bb{Z}/m\bb{Z}),H_{n}(\bb{X}_{i+1};\bb{Z}/m\bb{Z}))
\cong
\Hom_{}((\bb{Z}/m\bb{Z})^{d_{i}},(\bb{Z}/m\bb{Z})^{d_{i+1}})\
\cong \tennsor{}{\bb{M}}{d_{i+1}d_{i}}(\bb{Z}/m\bb{Z}).
\]
Let $\tennsor{i}{M}{q}$ be the image of the  map $H_{n}(\bb{X}_{i};\bb{Z}/\tennsor{}{q}{}\bb{Z})\to H_{n}(\bb{X}_{i+1};\bb{Z}/\tennsor{}{q}{}\bb{Z})$ under this isomorphism. 
For an indexing set $S$ we define  $\tennsor{i}{M}{S}\in \tennsor{}{\bb{M}}{d_{i+1}d_{i}}(\bb{Z}/\tennsor{}{q}{S}\bb{Z})$ to be the image of $\tennsor{i}{M}{q}$ under 
\[
\begin{array}{cc}
\tennsor{}{\bb{M}}{d_{i+1}d_{i}}(\bb{Z}/\tennsor{}{q}{}\bb{Z})\to \tennsor{}{\bb{M}}{d_{i+1}d_{i}}(\bb{Z}/\tennsor{}{q}{S}\bb{Z}),
&
(\tennsor{}{m}{hj}+\tennsor{}{q}{}\bb{Z})\mapsto (\tennsor{}{m}{hj}+\tennsor{}{q}{S}\bb{Z}).
\end{array}
\] 
%
%
To understand the isoclass of  $H_{n}(\bb{X}_{-};\bb{Z}/\tennsor{}{q}{S}\bb{Z})$ as a persistence module, one uses row and column operations, to perform on  the $\tennsor{i}{M}{S}$ as $i$ varies. 
One can perform any column operations on $\tennsor{0}{M}{S}$, but, crucially, the row operations performed on $\tennsor{i}{M}{S}$ must be opposite to the column operations performed on $\tennsor{i+1}{M}{S}$. 
As in \cite{Boissonnatmariacomputing}, when $\tennsor{i}{M}{S}$ has finite support over $i$, one can instead consider a block matrix built from combining the matrices $\tennsor{i}{M}{S}$. 

To relate this to $H_{n}(\bb{X}_{-};\bb{Z}/\tennsor{}{q}{T}\bb{Z})$ for $T\subseteq S$, one carefully chooses operations to perform. 
One requires that, if $\tennsor{i}{L}{S}\in \tennsor{}{\bb{M}}{d_{i+1}d_{i}}(\bb{Z}/\tennsor{}{q}{S}\bb{Z})$  is the result of operating on $\tennsor{i}{M}{S}$, then there exist invertible $\tennsor{i}{E}{S}\in \tennsor{}{\bb{M}}{d_{i}d_{i}}(\bb{Z}/\tennsor{}{q}{S}\bb{Z})$  such that  the diagram below, whose curved arrows are products of the projections $\tennsor{T}{\pi}{S}\colon \bb{Z}/\tennsor{}{q}{S}\bb{Z}\to \bb{Z}/\tennsor{}{q}{T}\bb{Z}$, commutes
\[
    \begin{tikzcd}[column sep = 2.1em]	{(\bb{Z}/\tennsor{}{q}{S}\bb{Z})^{d_{0}}} & {(\bb{Z}/\tennsor{}{q}{S}\bb{Z})^{d_{1}}} & {(\bb{Z}/\tennsor{}{q}{S}\bb{Z})^{d_{2}}} & \cdots && \\
	& {(\bb{Z}/\tennsor{}{q}{S}\bb{Z})^{d_{0}}} & {(\bb{Z}/\tennsor{}{q}{S}\bb{Z})^{d_{1}}} & {(\bb{Z}/\tennsor{}{q}{S}\bb{Z})^{d_{2}}} & \cdots \\
	& {(\bb{Z}/\tennsor{}{q}{T}\bb{Z})^{d_{0}}} & {(\bb{Z}/\tennsor{}{q}{T}\bb{Z})^{d_{1}}} & {(\bb{Z}/\tennsor{}{q}{T}\bb{Z})^{d_{2}}} & \cdots \\
	&& {(\bb{Z}/\tennsor{}{q}{T}\bb{Z})^{d_{0}}} & {(\bb{Z}/\tennsor{}{q}{T}\bb{Z})^{d_{1}}} & {(\bb{Z}/\tennsor{}{q}{T}\bb{Z})^{d_{2}}} & \cdots
	\arrow["{\tennsor{0}{M}{S}}", from=1-1, to=1-2]
	\arrow["{\tennsor{0}{E}{S}}"', from=2-2, to=1-1]
	\arrow[two heads, bend right = 10, from=1-1, to=3-2]
	\arrow["{\tennsor{1}{M}{S}}", from=1-2, to=1-3]
	\arrow["{\tennsor{1}{E}{S}}"', from=2-3, to=1-2]
	\arrow[two heads, bend right = 10, from=1-2, to=3-3]
	\arrow[from=1-3, to=1-4]
	\arrow["{\tennsor{2}{E}{S}}"', from=2-4, to=1-3]
	\arrow[two heads, bend right = 10, from=1-3, to=3-4]
	\arrow["{\tennsor{0}{L}{S}}"{pos=0.6}, from=2-2, to=2-3]
	\arrow[two heads, bend right = 10, from=2-2, to=4-3]
	\arrow["{\tennsor{1}{L}{S}}"{pos=0.6}, from=2-3, to=2-4]
	\arrow[two heads, bend right = 10, from=2-3, to=4-4]
	\arrow[from=2-4, to=2-5]
	\arrow[two heads, bend right = 10, from=2-4, to=4-5]
	\arrow["{\tennsor{0}{M}{T}}", from=3-2, to=3-3]
	\arrow["{\tennsor{0}{E}{T}}"{pos=0.4}, from=4-3, to=3-2]
	\arrow["{\tennsor{1}{M}{T}}", from=3-3, to=3-4]
	\arrow["{\tennsor{1}{E}{T}}", from=4-4, to=3-3]
	\arrow[from=3-4, to=3-5]
	\arrow["{\tennsor{2}{E}{T}}", from=4-5, to=3-4]
	\arrow["{\tennsor{0}{L}{T}}"'{pos=0.6}, from=4-3, to=4-4]
	\arrow["{\tennsor{1}{L}{T}}"'{pos=0.6}, from=4-4, to=4-5]
	\arrow[from=4-5, to=4-6]
\end{tikzcd}
\]
%
The previously discussed row and column  operations, and their compatibilities, are detailed in  \cite[p.~67]{Boissonnatmariacomputing} and \cite[\S4.3]{Boissonnatmariacomputing}, using terminology such as \emph{partial column operations}. 

For example, suppose  one adds, to each column $c$ of $\tennsor{i}{M}{\{q_{t}\}}$,  a multiple $\tennsor{}{\lambda}{t}\in\bb{Z}/\tennsor{}{q}{t}\bb{Z}$ of column $d\neq c$. 
Write $\tennsor{}{\lambda}{T}\in \bb{Z}/\tennsor{}{q}{T}\bb{Z}$ for the image of $(\tennsor{}{\lambda}{t})\in\tennsor{}{\prod}{t\in T}\bb{Z}/\tennsor{}{q}{t}\bb{Z}$ under the ring isomorphism coming from the Chinese remainder theorem. 
The \emph{partial inverse} operation adds, to column $c$ of $\tennsor{i}{M}{T}$, the multiple   $\tennsor{}{\lambda}{T}$ of column $d$. 
See \cite[Definition~1,~Proposition~1]{Boissonnatmariacomputing}. 
\textbf{Realisation}. 
Let $\cl{P}$ be the product of $\bb{N}$ with the poset of indexing sets related via reverse inclusion. 
So, $\cl{P}$ is the poset consisting of pairs $(i,S)$ where $(i,S)\leq (j,T)$ if and only if $i\leq j$ and $T\subseteq S$.  
If $U\subseteq T\subseteq S$ then $\tennsor{U}{\pi}{T}\circ \tennsor{T}{\pi}{S}=\tennsor{U}{\pi}{S}$. 
So there is a functor
\begin{equation}
\label{eqn-field-choice-functor}
    \begin{array}{ccc}
\boldsymbol{z}\colon \cl{P}\to \Rings,
&
(i,S)\mapsto \bb{Z}/\tennsor{}{q}{S}\bb{Z},
&
[(i,S)\leq (j,T)]\mapsto \tennsor{T}{\pi}{S}.
\end{array}
\end{equation}
Composing $\boldsymbol{z}$ with $\RusZhe$  from Corollary~\ref{cor-cool-nice-cheers-ta-nice1}, one obtains a double functor $\underline{\boldsymbol{z}}=\RusZhe\circ \boldsymbol{z}$, and hence a preadditive category of the form $\sr{T}(\fk{S}(\underline{\boldsymbol{z}}))/\sr{I}(\boldsymbol{c}(\underline{\boldsymbol{z}}))$ by Corollary~\ref{cor-cool-nice-cheers-ta-nice1}. 
One can observe that the commutative diagram above defines an isomorphism in $\Rep{\fk{S}(\underline{\boldsymbol{z}}),\sr{I}(\underline{\boldsymbol{z}})}$.

\subsection{Precursor to new Dynkin diagrams: thread quivers}\label{sec-thread-quivers}

Thread quivers were introduced in work of Berg--Van Roosmalen \cite{berg-vanroosmalen-thread1},  and then generalized in work of Paquette--Rock--Yildrim \cite{paquette-rock-yildirim-thread2}.
Intuitively, one replaces each arrow  in a quiver  with a totally-ordered set. 
We warn the reader that  Section~\ref{sec-thread-quivers} serves mostly as a precursor to Section~\ref{sec-infinite-Dynkin}.

A \emph{thread quiver} $(Q,T)$ is given by a quiver $Q$ and a disjoint union $T=\tennsor{}{\bigsqcup}{\alpha\in Q_{1}}\tennsor{}{T}{\alpha}$ of possibly empty totally-ordered sets $\tennsor{}{T}{\alpha}$. 
Assume, for simplicity here, that $Q$ is connected and acyclic. 
Let $\tennsor{}{\overline{T}}{\alpha}=\{h(\alpha),t(\alpha)\}\sqcup \tennsor{}{T}{\alpha}$ and $t(\alpha)<x<h(\alpha)$ for   $x\in \tennsor{}{T}{\alpha}$. 
The \emph{path category} $\cl{C}(Q,T)$   has object set  $\tennsor{}{\bigsqcup}{\alpha} \tennsor{}{\overline{T}}{\alpha}$. 
Define $\Hom_{\cl{C}(Q,T)}(c,d)$ in cases, by  
\[
\begin{array}{cc}
 \{\tennsor{}{\eta}{dc} \}
    &
    (c,d\in\tennsor{}{\overline{T}}{\alpha},(c,d)\neq (t(\alpha),h(\alpha)),
    \\
    \{p\mid p\in \tennsor{}{Q}{\geq 0},h(p)=d,t(p)=c\}
    &
    (c,d\in\tennsor{}{Q}{0}),
    \\
    \{p\tennsor{}{\eta}{t(p)c}\mid p\in \tennsor{}{Q}{>0},h(p)=d,t(p)= h(\alpha)\}
    &
    (c\in\tennsor{}{T}{\alpha},d\in\tennsor{}{Q}{0}), 
    \\
    \{\tennsor{}{\eta}{dh(p)}p\mid p\in \tennsor{}{Q}{\geq 0},t(p)=c,h(p)=t(\beta) \}
    &
    (d\in\tennsor{}{T}{\beta},c\in\tennsor{}{Q}{0}),
    \\
    \{\tennsor{}{\eta}{dh(p)}p\tennsor{}{\eta}{t(p)c} \mid p\in \tennsor{}{Q}{>0},h(p)=t(\beta),t(p)=h(\alpha)\}
    &
    (c\in\tennsor{}{T}{\alpha},\,d\in\tennsor{}{T}{\beta},), 
\end{array}
\]
where $\tennsor{}{Q}{\geq0}$ (respectively, $\tennsor{}{Q}{>0}$) is the set of all (respectively, non-trivial) paths in $Q$. 
For composition, one extends   the  rules of  path composition,   by setting
\[
\begin{array}{cccc}
\tennsor{}{\eta}{cb}\tennsor{}{\eta}{ba}=\tennsor{}{\eta}{ca},
&
\tennsor{}{\eta}{h(\alpha)c}\tennsor{}{\eta}{ct(\alpha)}=\alpha,
&
\tennsor{}{e}{h(\alpha)}\tennsor{}{\eta}{h(\alpha)c}=\tennsor{}{\eta}{h(\alpha)c},
&
\tennsor{}{\eta}{ct(\alpha)}\tennsor{}{e}{t(\alpha)}=\tennsor{}{\eta}{ct(\alpha)},
\end{array}
\] for any $\alpha\in \tennsor{}{Q}{1}$ and $a\leq b\leq c$ in $\tennsor{}{T}{\alpha}$. 
So,   $\tennsor{}{\iden}{c}=\tennsor{}{\eta}{cc}$ if $c\in \tennsor{}{T}{\alpha}$, and $\tennsor{}{\iden}{c}=\tennsor{}{e}{c}$ if $c\in \tennsor{}{Q}{0}$.

\textbf{Background}. 
For a field $K$, a $K$-category $\cl{A}$ is a \emph{finite $K$-variety} provided it is hom-finite (over $K$), additive and idempotent complete. 
Writing $\tennsor{}{\lMod{\cl{A}}}{\operatorname{fd}}$ for the subcategory of $\sr{F}$ in $\lMod{\cl{A}}$ with each $\sr{F}(a)$ finite-dimensional, in this case, one obtains a duality
\[
\begin{array}{cc}
\tennsor{}{\lMod{\cl{A}}}{\operatorname{fd}}\to \tennsor{}{\lMod{\cl{A}^{\op}}}{\operatorname{fd}},
&
\sr{F}\mapsto \sr{F}^{*}=\Hom_{\lmod{K}}(\sr{F}(-),K). 
\end{array}
\]
 Auslander--Reiten \cite{auslander-reiten-stable} define a finite $K$-variety  $\cl{A}$ to be  \emph{dualizing} if this duality restricts to a duality  $\lmod{\cl{A}}\to\lmod{\cl{A}^{\text{op}}}$ where $\lmod{\cl{A}}$ is the full subcategory  of  $\lMod{\cl{A}}$  consisting of finitely presented functors. 
 For an  abelian, hom-finite, Ext-finite $K$-category $\cl{C}$, a  \emph{Serre functor}  is an auto-equivalence $\bb{S}$ of the bounded derived category $\cl{D}^{b}(\cl{C})$ together with a natural isomorphism  $\Hom_{\cl{D}^{b}(\cl{C})}(-,?)\to (\Hom_{\cl{D}^{b}(\cl{C})}(?,\bb{S}(-)))^{*}$. 
 %
 %
 Van Roosmalen \cite{vanroosmalenhereditarydirected} defines a preadditive category $\cl{A}$ to be \emph{semihereditary} provided $\lmod{\cl{A}}$ is abelian and hereditary. 
By \cite[Proposition~4.2]{vanroosmalenhereditarydirected}, this is equivalent to asking  each full subcategory  of  $\cl{A}$, with finitely many objects, is semihereditary. 
By \cite[Corollary~4.9]{berg-vanroosmalen-thread1}, a finite $K$-variety $\cl{A}$ is dualizing and semihereditary if and only if $\lmod{\cl{A}}$ is  abelian, hereditary, as has a Serre functor. 

For a finite $K$-variety $\cl{A}$, recall a morphism $A\to B$   is \emph{left almost split} if it is not split, and any nonsplit $A\to C$ factors through it. 
By \cite[Proposition~4.30]{berg-vanroosmalen-thread1}, it is equivalent that the simple  $\Hom_{\cl{A}}(A,-)/\rad(\Hom_{\cl{A}}(A,-))$ in $\lmod{\cl{A}}$ has a finite projective presentation.

\textbf{Motivation}.
Reiten--Van den Bergh \cite[Theorem~B]{reitenvandenberghnoetherian} classified  connected, noetherian, hereditary categories with a Serre functor. 
\cite[Theorem~1]{berg-vanroosmalen-thread1} says any  hereditary category  with enough projectives and a Serre functor is equivalent to $\lmod{K\cl{C}(Q,T)}$. 
One defines the thread quiver $(Q,T)$ from a 
semihereditary dualizing $K$-variety $\cl{A}$, as follows.

For objects $A$ and $B$  let 
$[A,C]$ be the additive closure of the indecomposables $B$ with $\Hom_{\cl{A}}(A,B)\neq 0 \neq \Hom_{\cl{A}}(B,C)$. 
By \cite[Proposition~4.30]{berg-vanroosmalen-thread1}, each  indecomposable object $A$ in $\cl{A}$ admits morphisms $A^{-}\to A\to A^{+}$ with  $A^{-}\to A$ left almost split and $A\to A^{+}$ right almost split. 
One calls $A$  \emph{thread} if one can choose $A^{-}$ and $A^{+}$ both indecomposable. 
One calls $[A,C]$ \emph{thread} if each indecomposable that it contains is thread. 
In this case, by  \cite[Corollaries~6.2,~6.3]{berg-vanroosmalen-thread1}, we have an equivalence $[A,C]\simeq K\cl{T}$ for $\cl{T}$ a locally discrete totally-ordered set with a maximum and a minimum, and if $B\in \operatorname{Ind}\cl{A}\setminus \operatorname{Ind}[A,C]$, then any morphism $A\to B$  (respectively, $B\to C$) factors through $A\to C^{+}$ (respectively, $B\to A^{-}$).

Let  $\{[\tennsor{}{A}{i},\tennsor{}{C}{i}]\mid i\in I\}$ denote the set of maximal infinite  threads in  $\cl{A}$. 
Define the full subcategory $\cl{Q}$ of $\cl{A}$ such that no indecomposable of $\cl{Q}$ lies in some $[\tennsor{}{A}{i},\tennsor{}{C}{i}]$.  
By \cite[Propositions~6.7,~6.8,~7.2]{berg-vanroosmalen-thread1}, $\cl{Q}$ is again a semihereditary dualizing $K$-variety, such that any thread $[X,Z]$ in $\cl{Q}$ must contain only finitely many isoclasses of indecomposables. 
By \cite[Proposition~7.1]{berg-vanroosmalen-thread1}, it follows that there is an equivalence  $\cl{Q}\simeq KQ$ where the (possibly infinite) quiver $Q$ has isoclasses $[X]$ of indecomposables $X$ in $\cl{Q}$ as vertices, and 
\[
\tennsor{}{\dim}{K}(\rad(\Hom_{\cl{Q}}(X,Y))/\rad^{2}(\Hom_{\cl{Q}}(X,Y)))
\]
distinct arrows $[X]\to [Y]$. 
One defines a totally ordered set $\tennsor{}{T}{\alpha}$, for each arrow $\alpha$ in $Q$, to be empty  unless $\alpha\in\Hom_{\cl{A}}(A_{i}^{-},C_{i}^{+})$, in which case one takes  the set of $[X]$ with $X$ indecomposable in $[\tennsor{}{A}{i},\tennsor{}{C}{i}]$, totally-ordered by $[X]\leq [Y]$ if and only  $\Hom_{\cl{A}}(X,Y)\neq 0$.

\textbf{More motivation}.
The path category $\cl{C}(Q,T)$ of a thread quiver $(Q,T)$ is a $2$-pushout,
\[\begin{tikzcd}[cramped,column sep = small, row sep=0.4em]
	\tennsor{}{\bigsqcup}{\alpha}(t\to h) &&&&&&&&& Q \\
	& t &&&&& {t(\alpha)} \\
	&&& h &&&&& {h(\alpha)} \\
	\\
	\\
	& {t(\alpha)} &&&&& {t(\alpha)} \\
	&&& {h(\alpha)} &&&&& {h(\alpha)} \\
	\\
	&& c &&&&& c \\
	{\tennsor{}{\bigsqcup}{\alpha}\tennsor{}{\overline{T}}{\alpha}} &&&&&&&&& {\cl{C}(Q,T)}
	\arrow["\sr{Q}", dashed, from=1-1, to=1-10]
	\arrow["\sr{T}"', dashed, from=1-1, to=10-1]
	\arrow[dashed, from=1-10, to=10-10]
	\arrow["{\sr{Q}^{\alpha}}", dashed, maps to, from=2-2, to=2-7]
	\arrow[from=2-2, to=3-4]
	\arrow["{\sr{T}^{\alpha}}"', dashed, maps to, from=2-2, to=6-2]
	\arrow["\alpha", from=2-7, to=3-9]
	\arrow[dashed, from=2-7, to=6-7]
	\arrow[dashed, maps to, from=3-4, to=3-9]
	\arrow[dashed, maps to, from=3-4, to=7-4]
	\arrow[dashed, from=3-9, to=7-9]
	\arrow[dashed, maps to, from=6-2, to=6-7]
	\arrow[from=6-2, to=7-4]
	\arrow[from=6-2, to=9-3]
	\arrow["{\alpha}", from=6-7, to=7-9]
	\arrow["{\tennsor{}{\eta}{ct(\alpha)}}"', from=6-7, to=9-8]
	\arrow[dashed, maps to, from=7-4, to=7-9]
	\arrow[from=9-3, to=7-4]
	\arrow[dashed, maps to, from=9-3, to=9-8]
	\arrow["{\tennsor{}{\eta}{h(\alpha)c}}"', from=9-8, to=7-9]
	\arrow[dashed, from=10-1, to=10-10]
\end{tikzcd}\]
See \cite[p.~255,~Example~2.11]{berg-vanroosmalen-thread1} for details. 
We now recall \cite[Theorem~A]{paquette-rock-yildirim-thread2}. 
Let $M$ in $\tennsor{}{\lMod{K\cl{C}(Q,T)}}{\operatorname{fd}}$ be indecomposable. 
Either there exists some $v\in \tennsor{}{Q}{0}$ with $M(v)\neq 0$, or there exists a unique  interval  $I\subseteq \tennsor{}{T}{\alpha}$ such that $M(c)=0$ for $c\notin I$, and such that the restriction of $M$ to $\tennsor{}{T}{\alpha}$ is $K_{I}$.  
Assuming the former, for each $\alpha\in\tennsor{}{Q}{1}$ we have 
\[
\tennsor{}{T}{\alpha}=\tennsor{1}{T}{\alpha}\sqcup \cdots \sqcup \tennsor{n(\alpha)}{T}{\alpha} \text{ with $M(\tennsor{}{\eta}{dc})$ an isomorphism whenever $c,d\in \tennsor{i}{T}{\alpha}$ and  $c\leq d$.}
\]
For each set $C$, either a  cell in the partition above, or of the form $C=\{v\}$  with $v\in\tennsor{}{Q}{0}$, choose $\tennsor{}{z}{C}\in C$. 
Consider the finite   totally ordered subset $\tennsor{}{T}{\alpha}(M)$ consisting of the  $\tennsor{}{z}{C}$. 

It follows that there exists  indecomposable $N$ in $\tennsor{}{\lMod{K(Q,\tennsor{}{T}{}(M))}}{\operatorname{fd}}$ such that
\[
\begin{array}{c}
        \begin{array}{ccc}
M(\tennsor{}{\eta}{dc})
=N(\tennsor{}{\eta}{z_{D} z_{C}} ),
&
M(p)
=N(p),
&
M(p\tennsor{}{\eta}{t(p)c})
=N(p\tennsor{}{\eta}{t(p)z_{C}}),
\end{array}
\\
\begin{array}{cc}
M(\tennsor{}{\eta}{dh(p)}p)
=N(\tennsor{}{\eta}{z_{D}h(p)}p),
&
M(\tennsor{}{\eta}{dh(p)}p\tennsor{}{\eta}{t(p)c})
=N(\tennsor{}{\eta}{z_{D}h(p)}p\tennsor{}{\eta}{t(p)z_{C}}),
\end{array}
\end{array}
\]
whenever $c\in C$ and $d\in D$. 
Thus, to classify indecomposables in  $\tennsor{}{\lMod{K\cl{C}(Q,T)}}{\operatorname{fd}}$, it suffices to classify indecomposables $N$ in  $\tennsor{}{\lMod{K\cl{C}(Q,T')}}{\operatorname{fd}}$ where each  $\tennsorp{}{T}{\alpha}$ is finite. 

\textbf{Realisation and application}. 
Let $(Q,T)$ be a thread quiver. 
As in Section~\ref{subsec-ptwise-artinian}, the category $K\tennsor{}{\lMod{\cl{C}(Q,T)}}{\operatorname{fd}}$ is equivalent to the category of representations of a species, satisfying a commutativity relation. 
An immediate consequence of Corollary~\ref{cor-actually-understandable-decomposition}, in this context, is \cite[Theorem~A(1)]{paquette-rock-yildirim-thread2}. 
Note that the proof of Corollary~\ref{cor-actually-understandable-decomposition} involved a (straightforward) functor of the form $\cl{C}(Q,T)\to \Rings$. 
To introduce new examples, in Section~\ref{sec-infinite-Dynkin}  below, we allow any choice of such a functor.



In \cite{paquette-rock-yildirim-thread2} one also allow ideals $\cl{I}$ in $\cl{C}(Q,T)$ satisfying  mild admissibility conditions.
We can replace $\cl{C}(Q,T)$ with $\cl{C}(Q,T)/\cl{I}$ in the above construction, and obtain a similar realisation and result. 
Another interesting example of a thread quiver is the circle, seen by thread quiver  quiver $(Q,T)$ where $Q$ is (possibly cyclic, and) of type $\tennsor{}{\widetilde{\bb{A}}}{n}$. 
In this case, \cite[Theorem~1.2(2)]{repsofcircle} follows by combining \cite[Theorem~1.2(1)]{repsofcircle} and Corollary~\ref{cor-actually-understandable-decomposition}. 

 Igusa--Rock--Todorov \cite{Igusa-Rock-Todorov-continuous-i} defined  \emph{continuous quivers of type $\bb{A}$}, and studied their representations, meaning functors valued in vector spaces.  
    As we have seen in Section~\ref{subsec-ptwise-artinian}, any such a functor category may be realised using a  species with a commutativity condition. 
    By Corollary~\ref{cor-yay-hereditary}, the category of representations of such a species is hereditary. 
    The presence of relations voids this property: to see that one loses the hereditary property for continuous type $\bb{A}$ quivers, see \cite[Section 3.5]{Igusa-Rock-Todorov-continuous-i}.

\begin{defn}
Let $\RusI\colon \cl{C}(Q,T)\to\Byemod$ be a double functor for a thread quiver $(Q,T)$.
We call $(\fk{S}(\RusI),\boldsymbol{c}(\RusI))$, the pair  coming from Lemma~\ref{small-cat-becomes-species}, the \emph{threaded species}  of $\RusI$. 
\end{defn}

Recall Definition~\ref{defn-acyclic-connected}. 
We leave the proof of the next result as an exercise for the reader. 

\begin{lem}
    Let $\fk{S}(\RusI)$ be the threaded species of a double functor $\RusI\colon \cl{C}(Q,T)\to\Byemod$. 
    \begin{enumerate}
        \item 
    If $Q$ is acyclic  then $\fk{S}(\RusI)$ is acylic. 
    \item If $Q$ is connected and $\RusI(\fk{f})\neq 0$ for each morphism $\fk{f}$, then $\fk{S}(\RusI)$ is connected. 
    \end{enumerate}
    
\end{lem}

\subsection{Infinite versions of Dynkin diagrams of type $\bb{B}$ and $\bb{C}$}
\label{sec-infinite-Dynkin} 
The usual Dynkin diagrams,  of type $\bb{A}$, $\bb{B}$, $\bb{C}$, $\bb{D}$,  $\bb{E}$, $\bb{F}$ and $\bb{G}$, are finite, and 
characterise  species of finite representation type. 
Their extended versions characterise those of tame representation type (in a suitable sense). 
This is the Dlab--Ringel \cite{Dlab-Ringel-graphs-1976}  generalisation of Gabriels theorem. 

Infinite versions of Dynkin diagrams of type $\bb{A}$ and $\bb{D}$ have already appeared; see for example \cite{strongly-locally-finite-quivers,crawley-boevey-intervals,GallupSawin,repsofcircle,Igusa-Rock-Todorov-continuous-i}. 
Each family of type $\bb{E}$, $\bb{F}$ or $\bb{G}$ is finite. 
We suggest a definition for those of type $\bb{B}$ and $\bb{C}$, and their extended versions, $\widetilde{\bb{B}}$, $\widetilde{\bb{C}}$ and $\widetilde{\bb{BC}}$. 

For brevity, we consider the following condition on a given species $\fk{S}=(\tennsor{}{R}{x},\tennsor{y}{A}{x})$. 
\[
(*)\quad\text{Each $\tennsor{}{R}{x}$ is a division ring, and each $\tennsor{y}{A}{x}$ is finite-dimensional over  $\tennsor{}{R}{y}$ and over $\tennsor{}{R}{x}$.}
\]
\textbf{Motivation}. 
Following \cite{Dlab-Ringel-graphs-1976}, we recall how the species that we want to generalise arise.  
By a \emph{valuation} on a graph with vertex set $\ff{V}$ and edge set $\ff{E}\subseteq \{\{x,y\}\mid x,y\in\ff{V}\}$, we mean maps $\ff{d}\colon\ff{V}^{2}\to \tennsor{}{\bb{N}}{\geq 0}\times \tennsor{}{\bb{N}}{\geq 0}$  and $\ff{f}\colon\ff{V}\to \tennsor{}{\bb{N}}{>0}$ such that $\ff{d}(x,y)\ff{f}(y)=\ff{d}(y,x)\ff{f}(x)$ for each $\{x,y\}\in\ff{E}$. 
By a \emph{modulation} of a valuation $(\ff{d},\ff{f})$ on the graph of an acyclic connected quiver $Q$, we mean a species $\fk{S}=(\tennsor{}{R}{x},\tennsor{y}{A}{x})$ satisfying $(*)$, and 
\[
\begin{array}{ccc}
\tennsor{}{\dim}{\tennsor{}{R}{x}}(\tennsor{y}{A}{x})=\ff{d}(y,x),
&
\Hom_{\lMod{\tennsor{}{R}{y}}}(\tennsor{y}{A}{x},\tennsor{}{R}{x})\cong \tennsor{x}{A}{y}\cong\Hom_{\rMod{\tennsor{}{R}{x}}}(\tennsor{y}{A}{x},\tennsor{}{R}{x}) 
\end{array}
\] 
for each $(x,y)\in\ff{V}^{2}$,  where the isomorphisms are $\tennsor{}{R}{x}$-$\tennsor{}{R}{y}$-bilinear. 
\cite[p.~2,~Theorem]{Dlab-Ringel-graphs-1976} says that the category $\Rep{\fk{S}}$ has finitely many isoclasses of indecomposables if and only if the valued graph is  of type $\bb{A}$, $\bb{B}$, $\bb{C}$, $\bb{D}$,  $\bb{E}$, $\bb{F}$ and $\bb{G}$. 

\begin{defn}
\label{defn-type-A}
Let $Q$ be a connected subquiver of a quiver with underlying graph 
\[\begin{tikzcd}
	\triangle & \boldsymbol{\ocircle} & \boldsymbol{\ocircle} & \boldsymbol{\ocircle} & \boldsymbol{\ocircle} & \cdots & (\square)
	\arrow[no head, from=1-2, to=1-1]
	\arrow[no head, from=1-3, to=1-2]
	\arrow[no head, from=1-4, to=1-3]
	\arrow[no head, from=1-5, to=1-4]
	\arrow[no head, from=1-6, to=1-5]
\end{tikzcd}\]
The square on the right in brackets denotes, in case $Q$ is finite, the right-most vertex of $Q$. 

%
%
%
Let $\tennsor{}{T}{a}$ be a possibly empty totally ordered set for each arrow $a$ in $Q$. 
As in Section~\ref{sec-thread-quivers}, consider the path category $\cl{C}(Q,T)$  of the thread quiver $(Q,T)$, with object set $\ff{C}$. 

Let $\fk{S}(\RusI)=(\tennsor{}{R}{x},\tennsor{y}{A}{x})$ be the threaded  species of a a  double functor $\RusI\colon\cl{C}\to \Byemod$. 
Assume $\fk{S}(\RusI)$ satisfies $(*)$ and $(**)$. 
For each $y,x\in \ff{C}^{2}$ let $\tennsor{y}{d}{x}=(\dim_{\tennsor{}{R}{y}}(\tennsor{y}{A}{x}), \dim_{\tennsor{}{R}{x}}(\tennsor{y}{A}{x}))$. 
%

    %
    %
$\RusI$ \emph{has type} $\bb{B}$ if $\tennsor{y}{A}{x}\neq 0$ implies  ($y=\triangle$ and $\tennsor{\triangle}{d}{x}=(1,2)$) or ($y\neq \triangle$ and $\tennsor{y}{d}{x}=(1,1)$).  

$\RusI$ \emph{has type} $\bb{C}$ if $\tennsor{y}{A}{x}\neq 0$ implies  ($y=\triangle$ and $\tennsor{\triangle}{d}{x}=(2,1)$) or ($y\neq \triangle$ and $\tennsor{y}{d}{x}=(1,1)$). 

Now assume in addition that $Q$ is finite (note $T$ may still be infinite). 

$\RusI$ \emph{has type} $\widetilde{\bb{B}}$ if  $\tennsor{y}{A}{x}\neq 0$ implies  ($y=\triangle$, $x\neq \square$ and $\tennsor{\boldsymbol{\ocircle}}{d}{x}=(1,2)$) or ($y\neq \triangle$, $x= \square$ and $\tennsor{y}{d}{\square}=(2,1)$) or ($(y,x)=(\triangle,\square),(\boldsymbol{\ocircle},\boldsymbol{\ocircle})$ and $\tennsor{y}{d}{x}=(1,1)$). 

$\RusI$ \emph{has type} $\widetilde{\bb{C}}$ if   $\tennsor{y}{A}{x}\neq 0$ implies ($y=\triangle$, $x\neq \square$ and $\tennsor{\boldsymbol{\ocircle}}{d}{x}=(2,1)$) or ($y\neq \triangle$, $x= \square$ and $\tennsor{y}{d}{\square}=(1,2)$) or ($(y,x)=(\triangle,\square),(\boldsymbol{\ocircle},\boldsymbol{\ocircle})$ and $\tennsor{y}{d}{x}=(1,1)$). 

\noindent $\RusI$ \emph{has type} $\widetilde{\bb{BC}}$ if    $\tennsor{y}{A}{x}\neq 0$ implies ($y=\triangle$, $x\neq \square$ and $\tennsor{\boldsymbol{\ocircle}}{d}{x}=(1,2)$) or ($y\neq \triangle$, $x= \square$ and $\tennsor{y}{d}{\square}=(1,2)$) or ($(y,x)=(\triangle,\square)$ and $\tennsor{y}{d}{x}=(1,4)$) or ($(y,x)=(\boldsymbol{\ocircle},\boldsymbol{\ocircle})$ and $\tennsor{y}{d}{x}=(1,1)$).

    
    
 \end{defn}

\subsection{Topological field theories with defects}\label{sec-top-field-theory}

    We exhibit  examples  of species defined using algebra maps between Frobenius algebras. 
These examples come from mathematical physics, more specifically, the study of \emph{topological field theories}. 
We recall, and refer the reader to, Sections 2.1, 2.3, 2.4 and 3.4 of \cite{field-theories}, for explicit details. 

\textbf{Motivation}. 
Let $K$ be any field. 
A (finite-dimensional, associative, unital) Frobenius $K$-algebra $A$ is said to have \emph{trace pairing} if the image of the associated linear functional $A\to K$ on each $a\in A$ is given by the trace of the map $A\to A$ given by $b\mapsto ab$. 

By \cite[Lemma~3.3]{field-theories}, if $A$ is Frobenius with trace pairing, then using a $K$-basis of $A$ and its dual, one can define a $K$-linear map $A\to Z(A)$ whose kernel is $[A,A]=\tennsor{}{\sum}{a,b\in A}K(ab-ba)$.  

Let $\tennsor{}{\ff{D}}{0}$, $\tennsor{}{\ff{D}}{1}$ and $\tennsor{}{\ff{D}}{2}$ be sets, the elements of which  are called \emph{topological junctions}, \emph{topological domain walls}, and \emph{world sheet phases}, respectively. 
Let $\ff{s},\ff{t}\colon \tennsor{}{\ff{D}}{1}\to \tennsor{}{\ff{D}}{2}$ be functions. 

We emphasize that $\tennsor{}{\ff{D}}{1}$ need not be finite,  and  $\tennsor{}{\ff{D}}{2}$ need not be finite. 
For example, in the footnote on \cite[p.~85]{field-theories}, the authors explicitly allow the possibility of a groupoid, in this context meaning  $\tennsor{}{\ff{D}}{1}$ is a small category with object set $\tennsor{}{\ff{D}}{2}$ such that every morphism $d$ is an isomorphism, with domain $\ff{s}(d)$ and codomain $\ff{t}(d)$.

We recall the input data for a \emph{lattice topological field theory} from \cite[\S3.4]{field-theories}. 
\begin{enumerate}
    \item A Frobenius algebra  with trace pairing $\tennsor{}{A}{p}$ for each world phase sheet $p\in\tennsor{}{\ff{D}}{2}$. 
    \item A finite-dimensional $\tennsor{}{A}{\ff{t}(d)}$-$\tennsor{}{A}{\ff{s}(d)}$-bimodule $\tennsor{}{B}{d}$ for each domain wall condition $d\in\tennsor{}{\ff{D}}{1}$.
\end{enumerate}
Using this data, we see how a species arises using further details from \cite{field-theories}. 

We recall detail from \cite[\S2.4]{field-theories}. 
Define a category $\cl{D}$ by taking object set $\tennsor{}{\ff{D}}{2}$, and for any $p,q\in \tennsor{}{\ff{D}}{2}$, defining  $\Hom_{\cl{D}}^{1}(p,q)$ to consist of symbols of the form 
\[
\begin{array}{cc}
(\tennsor{}{{\underline{d}}}{},\tennsor{}{{\underline{\varepsilon}}}{})=((\tennsor{}{d}{1},\tennsor{}{\varepsilon}{1}),\dots, (\tennsor{}{d}{\ff{n}},\tennsor{}{\varepsilon}{\ff{n}})),
     & 
     (\tennsor{}{d}{i}\in\tennsor{}{\ff{D}}{1},\,\tennsor{}{\varepsilon}{i}\in\{+,-\}),
\end{array}
\]
 such that  $\ff{t}(\tennsor{}{d}{1})=q$, $\ff{s}(\tennsor{}{d}{\ff{n}})=p$, and, whenever $1\leq i<\ff{n}$,  the following implications hold
 \[
 \begin{array}{cc}
(\tennsor{}{\varepsilon}{i},\tennsor{}{\varepsilon}{i+1})=(+,-)\Rightarrow\ff{s}(\tennsor{}{d}{i})=\ff{s}(\tennsor{}{d}{i+1})
&
(\tennsor{}{\varepsilon}{i},\tennsor{}{\varepsilon}{i+1})=(-,+)\Rightarrow\ff{t}(\tennsor{}{d}{i})=\ff{t}(\tennsor{}{d}{i+1}),
\\
(\tennsor{}{\varepsilon}{i},\tennsor{}{\varepsilon}{i+1})=(-,-)\Rightarrow\ff{t}(\tennsor{}{d}{i})=\ff{s}(\tennsor{}{d}{i+1}),
&
(\tennsor{}{\varepsilon}{i},\tennsor{}{\varepsilon}{i+1})=(+,+)\Rightarrow\ff{s}(\tennsor{}{d}{i})=\ff{t}(\tennsor{}{d}{i+1}).
 \end{array}
 \]
Define composition by concatenation, and maps $\tennsor{q}{*}{p}\colon \Hom_{\cl{D}}(p,q)\to \Hom_{\cl{D}}(q,p)$ by 
\[
((\tennsor{}{d}{1},\tennsor{}{\varepsilon}{1}),\dots, (\tennsor{}{d}{\ff{n}},\tennsor{}{\varepsilon}{\ff{n}}))^{*}=((\tennsor{}{d}{\ff{n}},-\tennsor{}{\varepsilon}{\ff{n}}),\dots, (\tennsor{}{d}{1},-\tennsor{}{\varepsilon}{1})),
\]
so that $\tennsor{q}{*}{p}$ is bijective with inverse $\tennsor{p}{*}{q}$. 
Using the category of $2$-\emph{dimensional bordisms with defects}, from \cite[p.~76]{field-theories}, one obtains a vector space for each pair of parallel morphisms, making $\cl{D}$ into a $2$-category. 
We unpack the algebraic details below.

\textbf{Realisation}. 
As in \cite[p.~88,~(3.6)]{field-theories}, for any $K$-algebra $A$ and any $A$-$A$-bimodule $X$  let $\circlearrowleft_{A}X$  be the cokernel of the left $Z(A)$-linear map $A\tennsor{}{\otimes}{A}X\to X$ given by $a\otimes x\mapsto ax-xa$. 

The $Z(\tennsor{}{A}{q})$-module of $2$-morphisms between $(\tennsor{}{{\underline{e}}}{},\tennsor{}{{\underline{\delta}}}{}),(\tennsor{}{{\underline{d}}}{},\tennsor{}{{\underline{\varepsilon}}}{})\in \Hom_{\cl{D}}(p,q)$  is given by
\[
\Hom_{\cl{D}}^{2}((\tennsor{}{{\underline{e}}}{},\tennsor{}{{\underline{\delta}}}{}),(\tennsor{}{{\underline{d}}}{},\tennsor{}{{\underline{\varepsilon}}}{}))=\,\circlearrowleft_{\tennsor{}{A}{q}}\tennsor{}{B}{(\tennsor{}{{\underline{d}}}{},\tennsor{}{{\underline{\varepsilon}}}{})}\tennsor{}{\otimes}{{\tennsor{}{A}{p}}} \tennsor{}{B}{(\tennsor{}{{\underline{e}}}{},\tennsor{}{{\underline{\delta}}}{})^{*}}
\]
where $\tennsor{}{B}{(\tennsor{}{{\underline{d}}}{},\tennsor{}{{\underline{\varepsilon}}}{})}=\tennsor{}{B}{(\tennsor{}{d}{1},\tennsor{}{\varepsilon}{1})}\otimes \cdots \otimes \tennsor{}{B}{(\tennsor{}{d}{\ff{n}},\tennsor{}{\varepsilon}{\ff{n}})}$ where $\tennsor{}{B}{(d,+)}=\tennsor{}{B}{d}$ and $\tennsor{}{B}{(d,-)}=\Hom_{\lMod{K}}(\tennsor{}{B}{d},K)$. 
Using that $(\tennsor{}{{\underline{e}}}{},\tennsor{}{{\underline{\delta}}}{})^{*}\in \Hom_{\cl{D}}(q,p)$ ensures $\tennsor{}{B}{(\tennsor{}{{\underline{d}}}{},\tennsor{}{{\underline{\varepsilon}}}{})}\tennsor{}{\otimes}{{\tennsor{}{A}{p}}} \tennsor{}{B}{(\tennsor{}{{\underline{e}}}{},\tennsor{}{{\underline{\delta}}}{})^{*}}$ is an $\tennsor{}{A}{q}$-$\tennsor{}{A}{q}$-bimodule. 

For horizontal composition, consider \cite[Figure~6]{field-theories} in the context of the identification
\[
\tennsor{}{B}{(\tennsor{}{{\underline{d}'}}{},\tennsor{}{{\underline{\varepsilon}'}}{})}\tennsor{}{\otimes}{{\tennsor{}{A}{q}}}\tennsor{}{B}{(\tennsor{}{{\underline{d}}}{},\tennsor{}{{\underline{\varepsilon}}}{})}\tennsor{}{\otimes}{{\tennsor{}{A}{p}}} \tennsor{}{B}{(\tennsor{}{{\underline{e}}}{},\tennsor{}{{\underline{\delta}}}{})^{*}}
\tennsor{}{\otimes}{{\tennsor{}{A}{q}}}\tennsor{}{B}{(\tennsor{}{{\underline{e}'}}{},\tennsor{}{{\underline{\delta}'}}{})^{*}}
=
\tennsor{}{B}{(\tennsor{}{{\underline{d}'}}{},\tennsor{}{{\underline{\varepsilon}'}}{}\mid\tennsor{}{{\underline{d}}}{},\tennsor{}{{\underline{\varepsilon}}}{})}\tennsor{}{\otimes}{{\tennsor{}{A}{p}}} \tennsor{}{B}{(\tennsor{}{{\underline{e}'}}{},\tennsor{}{{\underline{\delta}'}}{}\mid \tennsor{}{{\underline{e}}}{},\tennsor{}{{\underline{\delta}}}{})^{*}}
\]
where the bar  $\mid$ indicates the vertical composition in $\cl{D}$.

Write $1(p)$ for the identity on $p\in\tennsor{}{\ff{D}}{2}$. 
As observed  \cite[p.~107,~(4.2)]{field-theories},  $\End_{\cl{D}}^{2}(1(p))\cong Z(\tennsor{}{A}{q})$ is a commutative Frobenius algebra, and for each $(\underline{d},\underline{\varepsilon})\in\Hom_{\cl{D}}^{1}(p,q)$, the algebra $\End_{\cl{D}}^{2}(\underline{d},\underline{\varepsilon})$ is a possibly noncommutative Frobenius algebra. 
Furthermore, horizontal composition gives rise to $K$-algebra maps with central images, of the form
\[
\End_{\cl{D}}^{2}(1(q))\rightarrow \End_{\cl{D}}^{2}(\underline{d},\underline{\varepsilon})\leftarrow \End_{\cl{D}}^{2}(1(p))
\]
By \cite[Remarks~4.5, 4.10]{field-theories} this assignment defines a lax functor $\cl{D}\to \Central_{\Comms}^{\Rings}$. 
By composing with the  double functor $\RusSha$ from Corollary~\ref{cor-double-functors-giving-species}, we obtain a  double functor $\cl{D}\to \Byemod$, and hence a species with commutativity condition by Corollary~\ref{cor-cool-nice-cheers-ta-nice1}.

\appendix
\section{Additive tensor categories}
\label{appendix-tensor-cat}


\subsection{Bimodules over a preadditive category}
\label{subsection:additive-tensor-categories:bimodules-over-a-preadditive-category}

\begin{setup}
    Now and forever let $\cl{C}$  be a preadditive small category with object set $\ff{C}$, 
    $\lMod{\cl{C}}$ be the category of additive functors $\cl{C}\to \Ab$, and 
    $\Bimod{\cl{C}}$ be the category of biadditive functors  $\cl{C}^{\op}\times \cl{C}\to \Ab$.
    (Recall that a right $\cl{C}$-module is a left $\cl{C}^{\op}$-module.)
\end{setup}

We recall a construction going back to Yoneda \cite[\S4,~p.~548]{yoneda-ext}. 
The \emph{coend}  is the functor 
\begin{equation*}
    \int^{\cl{C}}\colon \Bimod{\cl{C}} \to \Ab,\,\sr{H}\mapsto \Coeq \left(\bigoplus_{a,b\in\ff{C},\,\gamma \in \Hom_{\cl{C}}(a,b)} \sr{H}(b,a)\rightrightarrows\bigoplus_{c\in\ff{C}}\sr{H}(c,c)\right),
\end{equation*}
where the parallel maps are uniquely defined by universal properties of the involved coproducts, with respect to   $\sr{H}(b,\gamma)$ and  $\sr{H}(\gamma,a)$. 
Consider the \emph{evaluation} functors
\begin{align*}
   \cl{C}^{\op}\times\Bimod{\cl{C}} & \stackrel{\tennsor{\cl{C}}{\text{ev}}{}}{\longrightarrow} \lMod{\cl{C}},
    &
    \cl{C}\times\Bimod{\cl{C}} & \stackrel{\tennsor{}{\text{ev}}{\cl{C}}}{\longrightarrow} \lMod{\cl{C}^{\op}},
    \\
    (x,\sr{G})&\mapsto \sr{G}(x,-),
    &
    (y,\sr{F})&\mapsto \sr{F}(-,y).
\end{align*}
Notice $\tennsor{\cl{C}}{\text{ev}}{}$ yields a left $\cl{C}$-module and $\tennsor{}{\text{ev}}{\cl{C}}$ yields a right $\cl{C}$-module.

We introduce a new functor ${\bowtie_\bb{Z}}:\lmod{\cl{C}^{\op}}\times\lmod{\cl{C}}\to\Bimod{\cl{C}}$.
Given a pair of functors $\sr{F}$ in $\lmod{\cl{C}^{\op}}$ and $\sr{G}$ in $\lmod{\cl{C}}$, we define $(\sr{F}{\ \bowtie_\bb{Z}\ } \sr{G} )(a,b)$ to be $\sr{F}(a)\otimes_\bb{Z} \sr{G}(b)$, for a pair of objects $(a,b)$ in $\cl{C}^{\op}\times\cl{C}$.
For a pair of morphisms $\gamma:a\to b$ and $\delta:c\to d$ in $\cl{C}$, we define the top row in the following display to be the bottom row:
\begin{align*}
 (\sr{F}{\ \bowtie_\bb{Z}\ }\sr{G})(\gamma, \delta) : \, (\sr{F}{\ \bowtie_\bb{Z}\ }\sr{G})(b,c) &\to (\sr{F}{\ \bowtie_\bb{Z}\ }\sr{G})(a,d)
    \\
 \sr{F}(\gamma) \otimes_{\bb{Z}} \sr{G}(\delta): \sr{F}(b)\otimes_{\bb{Z}}\sr{G}(c) &\to \sr{F}(a)\otimes_{\bb{Z}} \sr{G}(d).
\end{align*}
Considering for simplicity when $\cl{C}$ is a ring $R$ we are taking $M\otimes_{\bb{Z}} N$ for a right $R$-module $M$ and a left $R$-module $N$, and considering $M\otimes_{\bb{Z}} N$ as a left $R$-module by acting on $N$ and a right $R$-module by acting on $M$.

Combining these functors together one obtains the composition
\[
\begin{tikzcd}[column sep = 0.8cm]
\cl{C}\times\Bimod{\cl{C}}\times {\cl{C}^{\op}}\times \Bimod{\cl{C}}
  \arrow[rr, "\tennsor{}{\text{ev}}{\cl{C}}\times \tennsor{\cl{C}}{\text{ev}}{}"]
  &
  &  \lMod{\cl{C}^{\op}}\times\lMod{\cl{C}}
  \arrow[r, "{\bowtie_\bb{Z}}"]
  &
  \Bimod{\cl{C}}
  \arrow[r, "\int^{\cl{C}}"]
  &
  \Ab.
\end{tikzcd}
\]

\subsection{Additive tensor category of a biadditive functor}
\label{subsection:additive-tensor-categories:additive-tensor-category-of-a-biadditive-functor}

\begin{defn}\label{def-super-mega-tensor-product}
By un/currying the composition $\int^{\cl{C}} \circ {\bowtie_{\bb{Z}}} \circ \tennsor{}{\text{ev}}{\cl{C}}\times \tennsor{\cl{C}}{\text{ev}}{}$, we uniquely define the \emph{tensor product} on $\Bimod{\cl{C}}$, a functor $\otimes_{\cl{C}}\colon \Bimod{\cl{C}}\times \Bimod{\cl{C}}\to \Bimod{\cl{C}}$. 
Note that $\otimes_{\cl{C}}$ sends $(\sr{F},\sr{G})$ to the biadditive functor $\cl{C}^{\op}\times \cl{C}\to \Ab$ defined by
\[
(\sr{F}\otimes_{\cl{C}}\sr{G})(x,y)=\int^{\cl{C}}\sr{F}(-,y) \ \bowtie_\bb{Z}\ \sr{G}(x,-)=
\dfrac{\bigoplus_{c\in\ff{C}}\sr{F}(c,y) \otimes_{\bb{Z}}\sr{G}(x,c)}{\bigoplus_{\substack{a,b\in\ff{C},\,\gamma \in \Hom_{\cl{C}}(a,b)
\\
 f\in \sr{F}(b,y)g\in \sr{G}(x,a)}}\left\langle f\otimes \gamma g -f \gamma\otimes g 
\right\rangle},
\]
where we let $f \gamma= \sr{F}(\gamma,y)(f)$ and $\gamma g=\sr{G}(x,\gamma)(g)$.
\end{defn}

For a morphism $(\tennsor{}{\iden}{y},\delta):(y,x)\to (y,y)$  the morphism $(\sr{F}\otimes_{\cl{C}}\sr{G})(1_y,\delta)$  has the form
\[
\dfrac{\bigoplus_{c\in\ff{C}}\sr{F}(c,x) \otimes_{\bb{Z}}\sr{G}(y,c)}{\bigoplus_{\substack{a,b\in\ff{C},\,\gamma \in \Hom_{\cl{C}}(a,b)
\\
 f\in \sr{F}(b,x)g\in \sr{G}(y,a)}}\left\langle f\otimes \gamma g -f \gamma\otimes g 
\right\rangle}
\longrightarrow
\dfrac{\bigoplus_{c\in\ff{C}}\sr{F}(c,y) \otimes_{\bb{Z}}\sr{G}(y,c)}{\bigoplus_{\substack{a,b\in\ff{C},\,\gamma \in \Hom_{\cl{C}}(a,b)
\\
 f\in \sr{F}(b,y)g\in \sr{G}(y,a)}}\left\langle f\otimes \gamma g -f \gamma\otimes g 
\right\rangle}.
\]
The category $\Bimod{\cl{C}}$ equipped with $\otimes_{\cl{C}}$ is monoidal with unit $\Hom_{\cl{C}}$; see for example \cite[Proposition~1]{ElKaoutit-bocses:2020}, and also work of  Fisher  \cite{Fisher-tensor-product-satellites-derived-functors},  Mitchell \cite{Mitchell-rings-with-several-objects}, and Oberst and R{\"o}hrl \cite{Oberst-Rohrl-flatness:1970}. 
In particular, there are natural associator isomorphisms $(\sr{F}\otimes_{\cl{C}}\sr{G})\otimes_{\cl{C}}\sr{H}\cong \sr{F}\otimes_{\cl{C}}(\sr{G}\otimes_{\cl{C}}\sr{H})$, and natural unitor isomorphisms  $\Hom_{\cl{C}}\otimes_{\cl{C}}\sr{F}\cong\sr{F}\cong\sr{F}\otimes_{\cl{C}}\Hom_{\cl{C}}$. 
%

%
Likewise one defines a functor $\otimes_{\cl{C}}\colon \Bimod{\cl{C}}\times \lMod{\cl{C}}  \to \lMod{\cl{C}}$  by considering
\[
\begin{tikzcd}[column sep = 0.9cm]
  \Bimod{\cl{C}}\times \lMod{\cl{C}}
  \arrow[rr, "\tennsor{}{\text{ev}}{\cl{C}}\times\tennsor{}{\iden}{\lMod{\cl{C}}}"]
  &
  &
  \lMod{\cl{C}^{\op}}\times\lMod{\cl{C}}
  \arrow[r, "{\bowtie_\bb{Z}}"]
  &
  \Bimod{\cl{C}}
  \arrow[r, "\int^{\cl{C}}"]
  &
  \Ab.
\end{tikzcd}
\]
For $\sr{A}$ in $\Bimod{\cl{C}}$ let $\sr{A}^{\otimes [0]}=\Hom_{\cl{C}}$ and $\sr{A}^{\otimes [n]}=\sr{A}^{\otimes [n-1]}\otimes_{\cl{C}} \sr{A}$ for $n>0$. 
%

\begin{defn}
    \label{defn-add-tensor-cat-of-bifunctor}
    The \emph{additive tensor category} $\tennsor{}{\sr{T}}{\cl{C}}(\sr{A})$ of  $\sr{A}$ in $\Bimod{\cl{C}}$ has the same object set $\ff{C}$ as $\cl{C}$, and has morphism sets   $\Hom_{\tennsor{}{\sr{T}}{\cl{C}}(\sr{A})}(x,y)=\bigoplus_{n\geq 0}\sr{A}^{\otimes [n]}(x,y)$ for each $x,y\in\ff{C}$. 
    Composition is given by extending the maps  $\sr{A}^{\otimes [m]}(y,z)\times \sr{A}^{\otimes [n]}(x,y)\to \sr{A}^{\otimes [m+n]}(x,z)$. 
\end{defn}

\begin{rem}
    \label{remark-describing-functor-cat-of-add-tens-cat}
By  \cite[Proposition~1.1]{Simson1979} $\lMod{\sr{T}_{\cl{C}}(\sr{A})}$ is equivalent to  a subcategory of the morphism category $\lMod{\cl{C}}^{\rightarrow}$. 
Objects in this subcategory have the form $ \sr{A}\otimes_{\cl{C}}\sr{M}\to \sr{M}$ with $\sr{M}$ in $\lMod{\cl{C}}$. 
Morphisms in this subcategory, from $\Phi\colon \sr{A}\otimes_{\cl{C}}\sr{M}\to \sr{M}$ to $\Psi\colon \sr{A}\otimes_{\cl{C}}\sr{N}\to \sr{N}$, have the form $(\tennsor{}{\iden}{\sr{A}}\otimes_{\cl{C}}\Theta,\Theta)$ with $\Theta\in \Hom_{\lMod{\cl{C}}}(\sr{M},\sr{N})$ and $\Theta\Phi=\Psi(\tennsor{}{\iden}{\sr{A}}\otimes_{\cl{C}}\Theta)$.

We unpack this. 
Consider the sections $\iota_{n}\colon \sr{A}^{\otimes [n]}\to \Hom_{\lMod{\sr{T}_{\cl{C}}(\sr{A})}}(x,y)$ equipping the coproduct. 
Recall that  $\Hom_{\cl{C}}(x,y)=\sr{A}^{\otimes [0]}(x,y)$ and $\sr{A}(x,y)=\sr{A}^{\otimes [1]}(x,y)$.

Hence any $\sr{T}_{\cl{C}}(\sr{A})$-module $\sr{V}$ defines a $\cl{C}$-module  $\sr{M}_{\sr{V}}$ by $\sr{M}_{\sr{V}}(x)=\sr{V}(x)$ on objects $x$ and by $\sr{M}_{\sr{V}}(\alpha)=\sr{V}(\iota_{0}(\alpha))$ on morphisms $\alpha\colon x\to y$. 
For any objects  $c,x$ in $\cl{C}$ there is a map $\tennsor{x}{\alpha}{c}\colon  \sr{A}(c,x)\otimes \sr{V}(c)\to \sr{V}(x)$ sending $\alpha\otimes v$ to $\iota_{1}(\alpha)(v)$ for $\alpha\in \sr{A}(c,x)$ and $v\in \sr{V}(x)$. 
Letting $c$ vary, these maps combine to define a map $\tennsor{x}{\alpha}{}\colon  (\sr{A}\otimes_{\cl{C}}\sr{M}_{\sr{V}})(x)\to \sr{M}_{\sr{V}}(x)$. 
Letting $x$ vary, they define a natural transformation $\Phi_{\sr{V}}\colon \sr{A}\otimes_{\cl{C}}\sr{M}_{\sr{V}}\to \sr{M}_{\sr{V}}$. 
\end{rem}

\begin{rem}
\label{remark:modules-over-additive-tensor-category-for-a-discrete-category}
    Assume that $\Hom_{\cl{C}}(x,y)=0$ whenever the   objects $x,y$ in $\cl{C}$ are distinct, and so $\cl{C}$ is uniquely defined by the rings $\End_{\cl{C}}(x)$. 
    It follows that an object $\sr{A}$ in $\Bimod{\cl{C}}$ is uniquely defined by the  $\End_{\cl{C}}(y)$-$\End_{\cl{C}}(x)$-bimodules $\sr{A}(x,y)$.
    Furthermore, we have 
    \[
    \begin{array}{c}
     \sr{A}^{\otimes[n]}(x,y)=\bigoplus_{w_{1},\dots,w_{n-1}\in \ff{C}}\sr{A}(w_{n-1},w_{n})\otimes_{\End_{\cl{C}}(w_{n-1})}
    \dots \otimes_{\End_{\cl{C}}(w_{1})}\sr{A}(w_0,w_{1})
    \end{array}
    \]
    for $n\geq 1$, where $w_0=x$ and $w_n=y$.
    %
    Likewise, an object $\sr{M}$ in $\lMod{\cl{C}}$ is uniquely defined by the left $\End_{\cl{C}}(x)$-modules $\sr{M}(x)$, and as above $(\sr{A}\tennsor{}{\otimes}{\cl{C}}\sr{M})(x)=\bigoplus_{c\in\ff{C}}\sr{A}(c,x)\otimes_{\End_{\cl{C}}(c)}\sr{M}(c)$. 
    Moreover, given $\cl{C}$-modules $\sr{M}$ and $\sr{N}$,  additive maps $\Theta_{x}\colon \sr{M}(x)\to \sr{N}(x)$ form a natural transformation $\Theta\colon \sr{M}\to \sr{N}$ if and only if each $\Theta_{x}$ is $\End_{\cl{C}}(x)$-linear.

    In particular, by the universal property of the coproduct, a natural transformation of the form the form $ \sr{A}\tennsor{}{\otimes}{\cl{C}}\sr{M}\to \sr{M}$  is uniquely defined by specifying, for each $(c,x)\in \ff{C}^{2}$, an $\End_{\cl{C}}(x)$-module homomorphism of the form $\sr{A}(c,x)\otimes_{\End_{\cl{C}}(c)}\sr{M}(c)\to \sr{M}(x)$.

    Note that the equation $\Theta\Phi=\Psi(\tennsor{}{\iden}{\sr{A}}\tennsor{}{\otimes}{\cl{C}}\Theta )$ in Remark~\ref{remark-describing-functor-cat-of-add-tens-cat} now becomes $\Theta_{x}\Phi_{x}=\Psi_{x}(\bigoplus_{c}\tennsor{}{\iden}{\sr{A}(c,x)}\otimes \Theta_{c} )$ for each $x\in \ff{C}$.  
\end{rem}

\section{Double categories}
\label{appendix-double-cats}

%
%

\begin{defn} 
\label{defn-double-category}
  Firstly, assume we have a pair of categories $\cl{V}$ and $\cl{H}$, and functors
    \[
    \begin{array}{cccc}
\Delta\colon \cl{V}\to \cl{H},
&
\sr{L}\colon \cl{H}\to \cl{V},     
&
\sr{R}\colon \cl{H}\to \cl{V},
&
\odot \colon \cl{H}\tennsor{}{\times}{{\cl{V}}}\cl{H}\to \cl{H}
    \end{array}
    \]
    where $\cl{H}\tennsor{}{\times}{{\cl{V}}}\cl{H}$ is the ordinary  pullback of $\sr{L}$ and $\sr{R}$, whose objects are pairs $(\fk{g},\fk{f})$ for objects $\fk{g}$ and $\fk{f}$ of $\cl{H}$ such that $\sr{R}(\fk{f})=\sr{L}(\fk{g})$, and whose morphisms are defined by
    \[
\Hom_{\cl{H}\tennsor{}{\times}{{\cl{V}}}\cl{H}}((\fk{g},\fk{f}),(\fk{i},\fk{h}))=\{(\delta,\varepsilon) \in \Hom_{\cl{H}}(\fk{g},\fk{i})\times \Hom_{\cl{H}}(\fk{f},\fk{h})\mid \sr{R}(\delta)=\sr{L}(\varepsilon)\}. 
    \]
    Secondly assume, for each object $x$ in $\cl{V}$ and $(\fk{g},\fk{f})$ in $\cl{H}\tennsor{}{\times}{{\cl{V}}}\cl{H}$, that we have 
    \[
    \begin{array}{ccc}
    \sr{L}(\Delta(x))=x=\sr{R}(\Delta(x)),
    &
    \sr{L}(\fk{g}\odot \fk{f})=\sr{L}(\fk{g}),
    &
    \sr{R}(\fk{g}\odot \fk{f})=\sr{R}(\fk{f}). 
    \end{array}
    \]
    Thirdly assume, for each object $x$ in $\cl{V}$, that  we  have isomorphisms in $\cl{H}$ of the form 
       \[
    \begin{array}{cccc}
\tennsor{}{\bff{\alpha}}{\fk{h},\fk{g},\fk{f}}\colon (\fk{h}\odot \fk{g}) \odot \fk{f}\to \fk{h}\odot (\fk{g} \odot \fk{f} )
&
\tennsor{}{\bff{\lambda}}{\fk{i}}\colon \Delta(x)\odot \fk{i}\to \fk{i},
&
\tennsor{}{\bff{\rho}}{\fk{j}}\colon \fk{j}\odot \Delta(x)\to \fk{j},      
    \end{array}
    \]
    natural in $(\fk{h},\fk{g},\fk{f})$ in $\cl{H}\tennsor{}{\times}{{\cl{V}}}\cl{H}\tennsor{}{\times}{{\cl{V}}}\cl{H}$ and in $\fk{i}$ and $\fk{j}$ in $\cl{H}$ with $\sr{L}(\fk{i})=x=\sr{R}(\fk{j})$. 

    We call $\cl{D}=(\cl{V},\cl{H},\Delta,\sr{L},\sr{R},\odot,\bff{\alpha},\bff{\lambda},\bff{\rho})$ a \emph{double category} if, in the notation above, we have the following equalities of morphisms in $\cl{V}$:
    \[
    \begin{array}{ccc}
    \sr{L}(\tennsor{}{\bff{\alpha}}{\fk{h},\fk{g},\fk{f}})
    =
    \tennsor{}{\iden}{\sr{L}(\fk{h})},
&
    \sr{R}(\tennsor{}{\bff{\alpha}}{\fk{h},\fk{g},\fk{f}})
    =
    \tennsor{}{\iden}{\sr{R}(\fk{f})},
&
    \sr{L}(\tennsor{}{\bff{\lambda}}{\fk{i}})
    =
    \tennsor{}{\iden}{x},
\\
    \sr{R}(\tennsor{}{\bff{\lambda}}{\fk{i}})
    =
    \tennsor{}{\iden}{\sr{R}(\fk{i})},
&
   \sr{L}(\tennsor{}{\bff{\rho}}{\fk{j}})
    =
    \tennsor{}{\iden}{\sr{L}(\fk{j})},
&
    \sr{R}(\tennsor{}{\bff{\rho}}{\fk{j}})
    =
    \tennsor{}{\iden}{x},
    \end{array}
    \]
    and if each of the following pentagonal and triangular diagrams commutes
    \begin{align}
                           {\begin{tikzcd}[ampersand replacement = \&]
   \& ((\fk{i} \odot \fk{h}) \odot \fk{g}) \odot \fk{f} \ar[dl, "\tennsor{}{\bff{\alpha}}{\fk{i},\fk{h},\fk{g}} \odot \tennsor{}{\iden}{\fk{f}}"'] \ar[dr, "\tennsor{}{\bff{\alpha}}{\fk{i}\odot \fk{h},\fk{g},\fk{f}}"] \& 
    \\
    (\fk{i} \odot (\fk{h} \odot \fk{g})) \odot \fk{f} \ar[d, "\tennsor{}{\bff{\alpha}}{\fk{i},\fk{h}\odot \fk{g},\fk{f}}"'] \&\& (\fk{i} \odot \fk{h}) \odot (\fk{g}\odot \fk{f}) \ar[d, "\tennsor{}{\bff{\alpha}}{\fk{i},\fk{h},\fk{g}\odot \fk{f}}"]
    \\
    \fk{i} \odot ((\fk{h} \odot \fk{g}) \odot \fk{f}) \ar[rr, "\tennsor{}{\iden}{\fk{i}}\odot \tennsor{}{\bff{\alpha}}{\fk{h},\fk{g},\fk{f}}"] 
    \&\& \fk{i} \odot (\fk{h} \odot (\fk{g} \odot \fk{f}))
    \end{tikzcd}}\label{eqn-pentagon}
    \\
   {\begin{tikzcd}[ampersand replacement = \&]
    (\fk{j} \odot \Delta(x)) \odot \fk{i} \ar[rr, "\tennsor{}{\bff{\alpha}}{\fk{j},\Delta(x),\fk{i}}"] \ar[dr, "\tennsor{}{\bff{\rho}}{\fk{j}} \odot \tennsor{}{\iden}{\fk{i}}"'] \&\& \fk{j} \odot (\Delta(x) \odot \fk{i}) \ar[dl, "\tennsor{}{\iden}{\fk{j}} \odot \tennsor{}{\bff{\lambda}}{\fk{i}}"]  \\
    \& \fk{j} \odot \fk{i} \& 
    \end{tikzcd}}\label{eqn-triangle}
    \end{align}
    Given such a double category, we use the terms \emph{objects}, \emph{vertical morphisms}, \emph{horizontal morphisms} and \emph{square morphisms} to refer to the objects of $\cl{V}$, the morphisms of $\cl{V}$, the objects of $\cl{H}$, and the morphisms of $\cl{H}$, respectively. 
    We refer to $\odot$, $\Delta$, $\sr{L}$ and $\sr{R}$ as the  \emph{horizontal composition}, \emph{boundary}, \emph{left} and \emph{right} functors. 
\end{defn}

\begin{example}
\label{example-cats-are-trivially-double-cats}
    Let $\cl{D}$ be a category.
    Let $\cl{V}$ be the category with same object set as $\cl{D}$, with only identity morphisms. 
    Let $\cl{H}$ be the category whose objects are morphisms $\fk{f}\colon d\to e$ in $\cl{D}$, with only identity morphisms. 
    Let $\Delta(d)=\tennsor{}{\iden}{d}$ for each object $d$ in $\cl{D}$, $\sr{L}(\fk{f})=e$ and  $\sr{R}(\fk{f})=d$ for each object $\fk{f}\colon x\to y$ in $\cl{D}^{\to}=\cl{D}_1$ and define $\odot$ by composition in $\cl{D}$. 
Define $\bff{\alpha}$, $\bff{\lambda}$ and $\bff{\rho}$ pointwise, by the equalities of composition, corresponding to associativity. 
%
%
%
\end{example}

\begin{example}
\label{example-double-cat-bimodules}
    \cite[Example~2.2]{Shulman2008} Consider $\Byemod$ from Definition~\ref{defn-byemod}.
    Let $\cl{V}$ be the category of (possibly noncommutative) rings and ring maps. 
    Define an object  in $\cl{H}$ to be a $Q$-$R$-bimodule  $\tennsor{Q}{A}{R}$ for some rings $Q$ and $R$.  
    Define a morphism of the form $\tennsor{Q}{A}{R}\to \tennsor{S}{B}{T}$ in $\cl{H}$ to be a triple $(\theta,\Gamma,\varphi)$ where $\theta\colon Q\to S$ and $\varphi\colon R\to T$ are ring maps and $\Gamma\colon A\to B$ is an additive map such that $\Gamma(qar)=\theta(q)\Gamma(a)\varphi(r)$ for each $q\in Q$, $a\in A$ and $r\in R$.

     Let $\Delta(R)=\tennsor{R}{R}{R}$ for each ring $R$, and let $\sr{L}(\tennsor{Q}{A}{R})=Q$ and  $\sr{R}(\tennsor{Q}{A}{R})=R$ for each $Q$-$R$-bimodule  $\tennsor{Q}{A}{R}$. 
 Note that an object in the ordinary pullback is a pair of bimodules of the form $(\tennsor{Q}{A}{R},\tennsor{R}{B}{S})$, whose image under $\odot$ we define to be $\tennsor{Q}{A}{ }\tennsor{}{\otimes}{R}\tennsor{}{B}{S}$.  

 A morphism $(\theta,\Gamma,\varphi,\Omega,\psi)\colon(\tennsor{L}{A}{M},\tennsor{M}{B}{N})\to (\tennsor{Q}{C}{R},\tennsor{R}{D}{S})$ consists of ring maps $\theta\colon L\to Q$, $\varphi\colon M\to R$ and $\psi\colon N\to S$, together with an $L$-$M$-bilinear map $\Gamma\colon A\to C$ and $M$-$N$-bilinear map $\Omega\colon B\to D$. 
 The image of this morphism under $\odot$, where $\tennsor{\theta}{C}{\varphi}$ is the $L$-$M$-bimodule induced by $\theta$ and $\varphi$, and similarly for $\tennsor{\varphi}{D}{\psi}$, is  $(\theta,\Lambda,\psi)$  where $\Lambda$ is the composition of $\Gamma\,\otimes\, \Omega$ with the universal morphism $\tennsor{\theta}{C}{\varphi}\tennsor{}{\otimes}{M} \tennsor{\varphi}{D}{\psi}\to \tennsor{Q}{C}{}\tennsor{}{\otimes}{R}\tennsor{}{D}{S}$. 
 %
%
%
\end{example}

\begin{defn}
\label{defn-lax-double-func}
    A \emph{ double functor} between double categories of the form 
    \[
    \tennsor{}{\RusI}{}=(\tennsor{}{\RusI}{0},\tennsor{}{\RusI}{1},\tennsor{}{\RusI}{\odot},\tennsor{}{\RusI}{\Delta})\colon (\cl{V},\cl{H},\Delta,\sr{L},\sr{R},\odot,\bff{\alpha},\bff{\lambda},\bff{\rho})\to (\tennsorp{}{\cl{D}}{0},\tennsorp{}{\cl{D}}{1},\Delta',\sr{L}',\sr{R}',\odot',\bff{\alpha}',\bff{\lambda}',\bff{\rho}')
    \]
    is defined by the following data. 
    Firstly,  $\tennsor{}{\RusI}{0}\colon \cl{V}\to \tennsorp{}{\cl{D}}{0}$ and $\tennsor{}{\RusI}{1}\colon \cl{H}\to \tennsorp{}{\cl{D}}{1}$ are functors, such that $ \sr{L}'\circ \tennsor{}{\RusI}{1} =  \tennsor{}{\RusI}{0}\circ \sr{L}$ and $\sr{R}'\circ \tennsor{}{\RusI}{1} =  \tennsor{}{\RusI}{0}\circ \sr{R}$. 
    Secondly, we have morphisms
    \[
    \begin{array}{cc}
    \tennsor{}{{(\tennsor{}{\RusI}{\odot})}}{\fk{g},\fk{f}}\colon \tennsor{}{\RusI}{1}(\fk{g})\odot'\tennsor{}{\RusI}{1}(\fk{f})\to \tennsor{}{\RusI}{1}(\fk{g}\odot \fk{f}),
    &
    \tennsor{}{{(\tennsor{}{\RusI}{\Delta})}}{x}\colon \Delta'(\tennsor{}{\RusI}{0}(x))\to \tennsor{}{\RusI}{1}(\Delta(x)),
    \end{array}
    \]
    natural in $\fk{g}$, $\fk{f}$ and $x$, such that each diagram   of the following form commutes 
    \begin{align}
          \begin{tikzcd}[column sep = 2cm, row sep = 0.5cm, ampersand replacement = \&]
    (\tennsor{}{\RusI}{1}(\fk{h}) \odot' \tennsor{}{\RusI}{1}(\fk{g})) \odot' \tennsor{}{\RusI}{1}(\fk{f}) \ar[rr, "\tennsorp{}{\bff{\alpha}}{\tennsor{}{\RusI}{1}(\fk{h}),\tennsor{}{\RusI}{1}(\fk{g}),\tennsor{}{\RusI}{1}(\fk{f})}"] \ar[d, "\tennsor{}{{(\tennsor{}{\RusI}{\odot})}}{\fk{h},\fk{g}} \odot \tennsor{}{\iden}{{\tennsor{}{\RusI}{1}(\fk{f})}}"'] \& \&\tennsor{}{\RusI}{1}(\fk{h}) \odot' (\tennsor{}{\RusI}{1}(\fk{g}) \odot' \tennsor{}{\RusI}{1}(\fk{f})) \ar[d, "{\tennsor{}{\iden}{{\tennsor{}{\RusI}{1}(\fk{h})}} \odot \tennsor{}{{(\tennsor{}{\RusI}{\odot})}}{\fk{g},\fk{f}}}"] \\
    \tennsor{}{\RusI}{1}(\fk{h} \odot \fk{g}) \odot' \tennsor{}{\RusI}{1}(\fk{f}) \ar[d, "\tennsor{}{{(\tennsor{}{\RusI}{\odot})}}{\fk{h}\odot \fk{g},\fk{f}}"'] \& \& \tennsor{}{\RusI}{1}(\fk{h}) \odot' \tennsor{}{\RusI}{1}(\fk{g} \odot \fk{f})\ar[d, "\tennsor{}{{(\tennsor{}{\RusI}{\odot})}}{\fk{h},\fk{g}\odot \fk{f}}"]
    \\
    \tennsor{}{\RusI}{1}((\fk{h} \odot \fk{g}) \odot \fk{f}) \ar[rr, "\tennsor{}{\RusI}{1}(\tennsor{}{\bff{\alpha}}{\fk{h},\fk{g},\fk{f}})"'] \& \& \tennsor{}{\RusI}{1}(\fk{h} \odot (\fk{g} \odot \fk{f}))
    \end{tikzcd}\label{eqn-lax-double-functor-comm-conditions-1}
    \\
  \begin{tikzcd}[column sep = 2cm, row sep = 0.5cm, ampersand replacement = \&]
	{\Delta'(\tennsor{}{\RusI}{0}(x))\odot' \tennsor{}{\RusI}{1}(\fk{f})} \&\& {\tennsor{}{\RusI}{1}(\Delta(x))\odot' \tennsor{}{\RusI}{1}(\fk{f})} \\
	{\tennsor{}{\RusI}{1}(\fk{f})} \&\& {\tennsor{}{\RusI}{1}(\Delta(x)\odot \fk{f})}
	\arrow["{\tennsor{}{{(\tennsor{}{\RusI}{\Delta})}}{x}\odot'\tennsor{}{\iden}{{\tennsor{}{\RusI}{1}(\fk{f})}}}", from=1-1, to=1-3]
	\arrow["{\tennsor{}{\bff{\lambda}\prime}{\tennsor{}{\RusI}{1}(\fk{f})}}"', from=1-1, to=2-1]
	\arrow["{\tennsor{}{{(\tennsor{}{\RusI}{\odot})}}{\Delta(x),\fk{f}}}", from=1-3, to=2-3]
	\arrow["{\tennsor{}{\RusI}{1}(\tennsor{}{\bff{\lambda}}{\fk{f}})}", from=2-3, to=2-1]
\end{tikzcd}\label{eqn-lax-double-functor-comm-conditions-2}
\\
\begin{tikzcd}[column sep = 2cm, row sep = 0.5cm, ampersand replacement = \&]
	{\tennsor{}{\RusI}{1}(\fk{g})\odot' \tennsor{}{\RusI}{1}(\Delta(x))} \&\& {\tennsor{}{\RusI}{1}(\fk{g})\odot' \Delta'(\tennsor{}{\RusI}{0}(x))} \\
	{\tennsor{}{\RusI}{1}(\fk{g}\odot \Delta(x))} \&\& {\tennsor{}{\RusI}{1}(\fk{g})}
	\arrow["{\tennsor{}{{(\tennsor{}{\RusI}{\odot})}}{\fk{g},\Delta(x)}}"', from=1-1, to=2-1]
	\arrow["{\tennsor{}{\iden}{{\tennsor{}{\RusI}{1}(\fk{g})}}\odot'\tennsor{}{{(\tennsor{}{\RusI}{\Delta})}}{x}}"', from=1-3, to=1-1]
	\arrow["{\tennsor{}{\bff{\rho}\prime}{\tennsor{}{\RusI}{1}(\fk{g})}}", from=1-3, to=2-3]
	\arrow["{\tennsor{}{\RusI}{1}(\tennsor{}{\bff{\rho}}{\fk{g}})}"', from=2-1, to=2-3]
\end{tikzcd}\label{eqn-lax-double-functor-comm-conditions-3}
    \end{align}
    %
    %
    A \emph{double transformation} of the form below is given by the following data.  
    \[\begin{tikzcd}
	{(\cl{V},\cl{H},\Delta,\sr{L},\sr{R},\odot,\bff{\alpha},\bff{\lambda},\bff{\rho})} &&&& {(\tennsorp{}{\cl{D}}{0},\tennsorp{}{\cl{D}}{1},\Delta',\sr{L}',\sr{R}',\odot',\bff{\alpha}',\bff{\lambda}',\bff{\rho}')}
	\arrow[""{name=0, anchor=center, inner sep=0}, "{(\tennsorp{}{\RusI}{0},\tennsorp{}{\RusI}{1},\tennsorp{}{\RusI}{\odot},\tennsorp{}{\RusI}{\Delta})}"', bend left = -10, from=1-1, to=1-5]
	\arrow[""{name=1, anchor=center, inner sep=0}, "{(\tennsor{}{\RusI}{0},\tennsor{}{\RusI}{1},\tennsor{}{\RusI}{\odot},\tennsor{}{\RusI}{\Delta})}", bend left = 10, from=1-1, to=1-5]
	\arrow[Rightarrow, from=1, to=0, "{\RusEl=(\tennsor{}{\RusEl}{0},\tennsor{}{\RusEl}{1})}" description]
\end{tikzcd}\]
Firstly,  $\tennsor{}{\RusEl}{0}\colon \tennsor{}{\RusI}{0}\Rightarrow \tennsorp{}{\RusI}{0}$ and $\tennsor{}{\RusEl}{1}\colon \tennsor{}{\RusI}{1}\Rightarrow \tennsorp{}{\RusI}{1}$ are natural transformations, such that  $ \sr{L}'(\tennsor{}{{(\tennsor{}{\RusEl}{1})}}{\fk{f}} )=  \tennsor{}{{(\tennsor{}{\RusEl}{0})}}{\sr{L}(\fk{f})} $ and $ \sr{R}'(\tennsor{}{{(\tennsor{}{\RusEl}{1})}}{\fk{f}} )=  \tennsor{}{{(\tennsor{}{\RusEl}{0})}}{\sr{R}(\fk{f})}$.
Secondly, each diagram below commutes 
 \begin{align}
          \begin{tikzcd}[column sep = 1.5cm, row sep = 0.5cm, ampersand replacement = \&]
    {{\tennsor{}{\RusI}{1}(\fk{g})\odot' \tennsor{}{\RusI}{1}(\fk{f})}} \&\&\& {{\tennsorp{}{\RusI}{1}(\fk{g})\odot' \tennsorp{}{\RusI}{1}(\fk{f})}} \\
	\\
	{{\tennsor{}{\RusI}{1}(\fk{g}\odot \fk{f})}} \&\&\& {{\tennsorp{}{\RusI}{1}(\fk{g}\odot \fk{f})}}
	\arrow["{\tennsor{}{{(\tennsor{}{\RusEl}{1})}}{\fk{g}} \odot' \tennsor{}{{(\tennsor{}{\RusEl}{1})}}{\fk{f}} }", from=1-1, to=1-4]
	\arrow["{\tennsor{}{{(\tennsor{}{\RusI}{\odot})}}{\fk{g},\fk{f}}}"', from=1-1, to=3-1]
	\arrow["{\tennsor{}{{(\tennsorp{}{\RusI}{\odot})}}{\fk{g},\fk{f}}}", from=1-4, to=3-4]
	\arrow["{\tennsor{}{{(\tennsor{}{\RusEl}{1})}}{\fk{g}\odot \fk{f}} }"', from=3-1, to=3-4]
    \end{tikzcd}\label{eqn-lax-double-functor-comm-conditions-4}
    \\
  \begin{tikzcd}[column sep = 1.5cm, row sep = 0.5cm, ampersand replacement = \&]
	{\Delta'(\tennsor{}{\RusI}{0}(x))} \&\&\& {\Delta'(\tennsorp{}{\RusI}{0}(x))} \\
	\\
	{\tennsor{}{\RusI}{1}(\Delta(x))} \&\&\& {\tennsorp{}{\RusI}{1}(\Delta(x))}
	\arrow["{\Delta'(\tennsor{}{({\tennsor{}{\RusEl}{0}})}{x})}", from=1-1, to=1-4]
	\arrow["{\tennsor{}{{(\tennsor{}{\RusI}{\Delta})}}{x}}"', from=1-1, to=3-1]
	\arrow["{\tennsor{}{{(\tennsorp{}{\RusI}{\Delta})}}{x}}", from=1-4, to=3-4]
	\arrow["{\tennsor{}{({\tennsor{}{\RusEl}{1}})}{\Delta(x)}}"', from=3-1, to=3-4]
\end{tikzcd}\label{eqn-lax-double-functor-comm-conditions-5}
    \end{align}
\end{defn}

\bibliographystyle{nicebst}
\bibliography{references}

\end{document}